\documentclass[12pt, reqno]{amsart}
\usepackage{latexsym}
\usepackage{amsfonts}
\usepackage{amsmath}
\usepackage{fourier}
\usepackage{color}              % Need the color package
\usepackage{hyperref}
\usepackage{enumerate, amsfonts, amsthm}

\usepackage{amssymb}
\usepackage[textsize=small]{todonotes}       % includes TO DO LIST
\usepackage{nicefrac}

\def\bl{\begin{lemma}}
\def\el{\end{lemma}}
\def\bth{\begin{theorem}}
\def\eth{\end{theorem}}
\def\bc{\begin{corollary}}
\def\ec{\end{corollary}}
\def\bcj{\begin{conjecture}}
\def\ecj{\end{conjecture}}
\def\bpr{\begin{proposition}}
\def\epr{\end{proposition}}
\def\bde{\begin{definition}}
\def\ede{\end{definition}}
\def\E{\mathbb{E}}

\newcommand{\be}{\begin{eqnarray}}
\newcommand{\ee}{\end{eqnarray}}
\newcommand{\eps}{{\mbox{$\epsilon$}}}
\newcommand{\R}{{\mathbb R}}

\newcommand{\Z}{{\mathbb Z}}

\newcommand{\T}{{\mathcal T}}

\newcommand{\A}{{\mathcal A}}
\newcommand{\M}{{\mathcal M}}

\newcommand{\B}{{\mathcal B}}

\newcommand{\with}{\hbox{ {\rm with} }}
\newcommand{\is}{\hbox{ {\rm is} }}

\newcommand{\Tm}{T_{{\rm mix}}}
\newcommand{\C}{{\mathcal{C}}}
\newcommand{\D}{{\mathcal{D}}}
\newcommand{\prob}{\mbox{\bf P}}
\newcommand{\pab}{\prob_{A,B}}
\newcommand{\poffa}{\prob_{{\rm \small off~} A}}
\newcommand{\eab}{\E_{A,B}}
\newcommand{\pa}{\prob_{A}}
\newcommand{\ea}{\E_{A}}
\newcommand{\pb}{\prob_{B}}

\newcommand{\p}{\mbox{\bf p}}

\newcommand{\lr}{\leftrightarrow}

\newtheorem{theorem}{Theorem}[section]
\newtheorem{definition}{Definition}[section]
\newtheorem{lemma}[theorem]{Lemma}
\newtheorem{prop}[theorem]{Proposition}

\newtheorem{corollary}[theorem]{Corollary}
\newtheorem{proposition}[theorem]{Proposition}
\newtheorem{claim}[theorem]{Claim}
\newtheorem{conjecture}[theorem]{Conjecture}

\theoremstyle{definition}
\numberwithin{equation}{section}

\def\arrowfillCS#1#2#3#4{%
   \thickmuskip0mu\medmuskip\thickmuskip\thinmuskip\thickmuskip
   \relax#4#1\mkern-7mu%
   \cleaders\hbox{$#4\mkern-2mu#2\mkern-2mu$}\hfill
   \mkern-7mu#3%$% other side.
}
\def\lrfill{\arrowfillCS\leftarrow\relbar\rightarrow\relax}

\newcommand{\mnot}{m_0}
\newcommand{\lrlong}{\longleftrightarrow}

\newcommand{\lrr}{\stackrel{r}{\lr}}
\newcommand{\lrs}{\stackrel{s}{\lr}}

\newcommand{\lref}{\stackrel{r_0}{\lrfill}}

\newcommand{\lrefmp}{\stackrel{P[2\mnot,r_0]}{\lrfill}}

\newcommand{\lrefmonep}{\stackrel{P[\mnot,r_0]}{\lrfill}}
\newcommand{\lrefmpx}{\stackrel{P[2\mnot,r_0],x}{\lrfill}}
\newcommand{\lrefmpy}{\stackrel{P[2\mnot,r_0],y}{\lrfill}}

\newcommand{\lrejx}{\stackrel{=j_x}{\lrfill}}

\newcommand{\lrer}{\stackrel{[r,2r]}{\lrfill}}
\renewcommand{\and}{\hbox{ {\rm and} }}

\newcommand{\off}{\hbox{ {\rm off} }}

\newcommand{\onlyon}{\hbox{ {\rm only on} }}
\newcommand{\inn}{\hbox{ {\rm in} }}
\newcommand{\vep}{\varepsilon}
\renewcommand{\eps}{\varepsilon}
\newcommand{\epsm}{\eps_{\m}}
\newcommand{\epsn}{\eps_{n}}
\newcommand{\km}{k_{\m}}
\newcommand{\Mm}{M_{\m}}

\newcommand{\rmm}{r_{\m}}

\newcommand{\whp}{whp}

\newcommand{\sss}{\scriptscriptstyle}
\newcommand{\e}{{\mathrm e}}
\newcommand {\convd}{\stackrel{d}{\longrightarrow}}
\newcommand {\convp}{\stackrel{\sss {\mathbb P}}{\longrightarrow}}

\newcommand{\nn}{\nonumber}

\newcommand{\Park}{P_{r,r_0}}
\newcommand{\conn}{\longleftrightarrow}
\newcommand{\lbeq}[1]  {\label{e:#1}}
\newcommand{\refeq}[1] {\eqref{e:#1}}
\newcommand{\eqn}[1]{\begin{equation} #1 \end{equation}}
\newcommand{\eqan}[1]{\begin{align} #1 \end{align}}
\newcommand{\Ccal}   {\mathcal{C}}
\newcommand{\Gcal}   {\mathcal{G}}
\newcommand{\cn}{\m}
\newcommand{\cluster}{{\mathcal C}}
\newcommand{\dbc}{\Leftrightarrow}
\newcommand{\bb}{\underline{b}}
\newcommand{\tb}{\overline{b}}
\newcommand{\bbe}{\underline{e}}
\newcommand{\bbf}{\underline{f}}
\newcommand{\te}{\overline{e}}
\newcommand{\tf}{\overline{f}}

\newcommand{\expec}{{\mathbb E}}
\newcommand{\Piv}{{\rm Piv}}
\newcommand{\alpham}{\alpha_m}
\newcommand{\omegam}{\omega_m}

\def\1{{\mathchoice {1\mskip-4mu\mathrm l}      % Blackboard bold 1
{1\mskip-4mu\mathrm l}
{1\mskip-4.5mu\mathrm l} {1\mskip-5mu\mathrm l}}}
\newcommand{\indic}[1]{\1_{\{#1\}}}

\newcommand{\m}{{m}}

\input epsf.sty

\begin{document}

\title{Hypercube percolation}
\author{Remco van der Hofstad}
\address{Department of Mathematics and
	Computer Science, Eindhoven University of Technology, P.O.\ Box 513,
	5600 MB Eindhoven, The Netherlands.}
\email{rhofstad@win.tue.nl}
\author{Asaf Nachmias}
\address{Department of Mathematics, University of British Columbia, 121-1984 Mathematics Rd, Vancouver, BC, Canada V6T1Z2.}
\email{asafnach@math.ubc.ca}

%\date{Version January 18th, 2012}

\begin{abstract} We study bond percolation on the Hamming hypercube $\{0,1\}^\m$ around the critical probability $p_c$. It is known that if $p=p_c(1+O(2^{-\m/3}))$, then with high probability the largest connected component $\C_1$ is of size $\Theta(2^{2\m/3})$ and that this
quantity is non-concentrated. Here we show that for any sequence $\epsm$ such that
$\epsm =o(1)$ but $\epsm \gg 2^{-\m/3}$ percolation on the hypercube at $p_c(1+\epsm)$ has
    $$
    {|\C_1| = (2+o(1))\epsm 2^\m }\quad \and \quad |\C_2| = o(\epsm 2^\m) \, ,
    $$
with high probability, where $\C_2$ is the second largest component. This resolves a conjecture
of Borgs, Chayes, the first author, Slade and Spencer \cite{BCHSS3}.
\end{abstract}

\maketitle

% Remove parskip for toc
%\setlength{\parskip}{1ex plus 7ex minus 3ex}
\tableofcontents

% Dutch style of paragraph formatting, i.e. no indents.
%\setlength{\parskip}{1.3ex plus 0.2ex minus 0.2ex}
%\setlength{\parskip}{0ex plus 0.2ex minus 0.2ex}

\section{Introduction}
\label{sec-intro}
Percolation on the Hamming hypercube $\{0,1\}^\m$ is a combinatorial model proposed in 1979 by Erd\H{o}s and Spencer \cite{ES}. The study of its phase transition poses two inherent difficulties. Firstly, its non-trivial geometry makes the combinatorial ``subgraph count'' techniques unavailable. Secondly, the critical probability where the phase transition occurs is significantly larger than $1/(\m-1)$, making the method of stochastic domination by branching processes very limited. Unfortunately, these are the two prominent techniques for obtaining scaling windows (see e.g., \cite{Aldo97, B1, BCKW01, ER, HL, JKL93, Lu, MolReed95, Nach09, NP2}).
%\todo{Check whether these are the right references!}

In light of the second difficulty, Borgs, Chayes, the first author, Slade and Spencer \cite{BCHSS1, BCHSS2, BCHSS3} suggested that the precise location $p_c$ of the phase transition is the unique solution to the equation
    \be
    \label{pcdef}
    \E_{p_c} |\C(0)| = \lambda 2^{\m/3} \, .
    \ee
where $\C(0)$ is the connected component containing the origin, $|\C(0)|$ denotes its size, and $\lambda\in(0,1)$ denotes an arbitrary constant. Later it will become clear how $\lambda$ is chosen. The lace expansion was then employed by the authors to show that at $p=p_c(1+O(2^{-\m/3}))$ the largest connected component $\C_1$ is of size $\Theta(2^{2\m/3})$ \whp\  --- the same asymptotics as in the critical Erd\H{o}s and R\'{e}nyi
random graph both with respect to the size of the cluster and the width of the scaling window
(see Section \ref{sec-gnp} for more details). However, this result does not rule out the possibility that this critical behavior proceeds beyond the $O(2^{-\m/3})$ window and does not give an upper bound on the width of the scaling window.

The authors conjectured that the giant component ``emerges'' just above this window (see \cite[Conjecture 3.2]{BCHSS3}). They were unable to prove this primarily because their combination of lace expansion and sprinkling methodology breaks for $p$ above the scaling window. In this paper we resolve their conjecture:
%
% The same authors conjectured  \cite[Conjecture 3.2]{BCHSS3} that whenever $\epsm=o(1)$ is a sequence with $\epsm \gg 2^{-m/3}$, then percolation
%with $p=p_c(1+\epsm)$ has $|\C_1| = (2+o(1)) \epsm 2^\m$ \whp. This is natural in view
%of Bollob\'as' \cite{B1} study of the $G(n,p)$ scaling window. In this paper,
%we resolve this conjecture:
%
\begin{theorem} \label{mainthm} Consider bond percolation on the Hamming hypercube $\{0,1\}^\m$ with $p=p_c(1+\eps)$,
where $p_c=p_c(\lambda)$ with $\lambda\in(0,\infty)$ a fixed constant, and $\eps=\epsm=o(1)$ is a positive sequence with $\epsm \gg 2^{-\m/3}$. Then
    $$
    \frac{|\C_1|}{2\epsm 2^\m}\convp 1\, ,
    $$
where $\convp$ denotes convergence in probability, and
    $$
    \E |\C(0)| = (4+o(1))\epsm^2 2^\m \, .
    $$
Furthermore, the second largest component $\C_2$ satisfies
    $$
    \frac{|\C_2|}{\epsm 2^m}\convp 0 \, .
    $$
\end{theorem}

The main novelty of our approach is showing that large percolation clusters behave in some sense like uniform random sets. We use this to deduce that two large clusters tend to ``clump'' together and form a giant component. This analysis replaces the appeal to the hypercube's isoperimetric inequality which is key in all the previous works on this problem (see further details in Section \ref{sec-intsprinkling}). It essentially rules out the possibility that two large percolation clusters are ``worst-case'' sets, that is, sets which saturate the isoperimetric inequality (e.g., two balls of radius $\m/2-\sqrt{\m}$ around the two poles of the hypercube). The precise behavior of the random walk on the  hypercube plays a key role in proving such statements. Our proof combines this idea with some combinatorial ideas (the ``sprinkling'' method of \cite{AKS82}, see Section \ref{sec-intsprinkling}),
and ideas originating in statistical physics (Aizenman and Barsky's \cite{AB87} differential
inequalities and variants of the triangle condition). Our proof methods are general and apply for other families of graphs such as various expanders of high degree and high girth, finite tori of dimension growing with the length and products of complete graphs of any dimension (answering a question asked in \cite{HL}). We state our most general theorem in Section \ref{sec-gen-thm} and illustrate its use with some examples. \\

The problem of establishing a phase transition for the appearance of a component of size order $2^m$ was solved in the breakthrough work of Ajtai, Koml\'{o}s and Szemer\'{e}di \cite{AKS82}. They proved that when the retention probability of an edge is scaled as $p=c/\m$ for a fixed constant $c>0$ the model exhibits a phase transition: if $c<1$, then the largest component has size of order $\m$ and if $c>1$, then the largest component has size linear in $2^\m$, with high probability.

At about the same time, Bollob\'as \cite{B1} initiated a study of zooming in onto the large scale properties of the phase transition on the Erd\H{o}s and R\'{e}nyi \cite{ER} random graph $G(n,p)$ (see Section \ref{sec-gnp} below). However, unlike $G(n,p)$, the phase transition in the hypercube does {\em not} occur around $p=1/(\mbox{deg}-1)$,
where deg denotes the degree of the graph. In fact, it was shown by the first author and Slade \cite{HS05, HS06} that $p_c$ of the hypercube $\{0,1\}^{\m}$ satisfies
    \be
    \label{pcest}
    p_c = {1 \over \m-1} + {7/2 \over \m^3} + O(\m^{-4}) \, ,
    \ee
Here and below we write $f(\m)=O(g(\m))$ if $|f(\m)|/|g(\m)|$ is uniformly bounded
from above by a positive constant, $f(\m)=\Theta(g(\m))$ if $f(\m)=O(g(\m))$ and
$g(\m)=O(f(\m))$ and $f(\m)=o(g(\m))$ if $f(\m)/g(\m)$ tends to $0$ with $m$.
We also say that a sequence of events $(E_\m)_{m\geq 1}$ \emph{occurs with high probability} (whp) when
$\lim_{\m\rightarrow \infty} \prob(E_{\m})=0$.

The first improvement to \cite{AKS82} was obtained by Bollob\'as, Kohayakawa and \L uczak \cite{BolKohLuc92}. They showed that if $p=(1+\epsm)/(\m-1)$ with $\epsm=o(1)$ but $\epsm \geq 60\m^{-1} (\log \m)^3$, then $|\C_1|=(2+o(1))\epsm 2^\m$ \whp. In view of (\ref{pcest}), it is clear that one cannot improve the regime of $\eps_m$ in their result to more than $\eps_m \geq m^{-2}$. In \cite{BCHSS3}, the authors show that when $\epsm \geq \e^{-cm^{1/3}}$ and $p=p_c(1+\epsm)$, then $|\C_1| \geq c\epsm 2^{\m}$ \whp. Note that  $\e^{-c\m^{1/3}} \gg 2^{-\alpha m}$ for any $\alpha>0$ so the requirement on $\epsm$ of Theorem \ref{mainthm} is much weaker. Our result,
combined with those in \cite{BCHSS1, BCHSS2, BCHSS3}, shows that it is sharp and therefore fully identifies the phase transition on the hypercube.

Other models of statistical physics, such as random minimal spanning trees and bootstrap percolation on the hypercube have been studied before, we refer the reader to \cite{BBM09, BBM10, Penrose98}. In the remainder of this section we present some of the necessary background and context of the result, briefly describe our techniques (we provide a more detailed overview of the proof in the next section) and present a general theorem which is used to establish scaling windows for percolation on various other graphs studied in the literature.

\subsection{The Erd\H{o}s and R\'{e}nyi random graph}
\label{sec-gnp}
Recall that $G(n,p)$ is obtained from the complete graph by retaining each edge of the complete graph on $n$ vertices with probability $p$ and erasing it otherwise, independently for all edges. Write $\C_j$ for the $j$th largest component obtained this way. An inspiring discovery of Erd\H{o}s and R\'{e}nyi \cite{ER} is
that this model exhibits a phase transition when $p$ is scaled like $p=c/n$.
When $c<1$ we have $|\C_1|=\Theta(\log n)$ \whp\ and $|\C_1|=\Theta(n)$ \whp{} when $c>1$.

The investigation of the case $c \sim 1$, initiated by Bollob\' as \cite{B1}
and further studied by {\L}uczak \cite{Lu},
revealed an intricate picture of the phase transition's nature. See
\cite{Bol85} for results up to 1984, and \cite{Aldo97,JKL93,JLR00,LPW94} for
references to subsequent work. We briefly describe these here. \\

\noindent {\bf \em The critical window.} When $p=(1 + O(n^{-1/3}))/n$,
for any fixed integer $j\geq 1$,
    $$
    \Big( {|\C_1| \over n^{2/3}}, \ldots, {|\C_j| \over n^{2/3}} \Big )
    \stackrel{d}{\longrightarrow} (\chi_1, \ldots, \chi_j) \, ,
    $$
where $(\chi_i)_{i=1}^j$ are random variables supported on $(0,\infty)$,
and $\convd$ denotes convergence in distribution.\newline

\noindent {\bf \em The subcritical phase.} Let $\epsn=o(1)$ be a non-negative
sequence with $\epsn \gg n^{-1/3}$ and put $p=(1-\epsn)/n$, then, for any
fixed integer $j\geq 1$,
    $$
    \frac{|\C_j|}{2\epsn^{-2} \log(\epsn^3 n)}\convp 1\, .
    $$
\noindent {\bf \em The supercritical phase.} Let $\epsn=o(1)$ be a non-negative sequence with
$\epsn \gg n^{-1/3}$ and put $p=(1+\epsn)/n$, then
    $$
    \frac{|\C_j|}{2\epsn n}\convp 1 \, ,
    $$
and, for any fixed integer $j \geq 2$,
    $$
    \frac{|\C_j|}{2\epsn^{-2} \log(\epsn^3 n)}\convp 1\, .
    $$
Thus, the prominent qualitative features of this phase transition are:

\begin{enumerate}
\item The emergence of the giant component occurs just above the scaling window.
That is, only in the supercritical phase we have that $|\C_2| \ll |\C_1|$, and that
$|\C_1|/n$ increases suddenly but smoothly above the critical value
(in mathematical physics jargon, the phase transition is of second order).

\item Concentration of the size of the largest connected components
outside the scaling window and non-concentration inside the window.

\item Duality: $|\C_2|$ in the supercritical phase has the same asymptotics as
$|\C_1|$ in the corresponding subcritical phase.\\
\end{enumerate}

Theorem \ref{mainthm} shows that (1) and (2) occurs on the hypercube (the non-concentration of $|\C_1|$ at $p_c$ was proved
in \cite{HH2}). Property (3) on the hypercube remains an open problem (see Section \ref{sec-open}).

\subsection{Random subgraphs of transitive graphs}
\label{sec-bchss}
Let us briefly review the study of random subgraphs of general finite transitive graphs
initiated in \cite{BCHSS1, BCHSS2, BCHSS3}. We focus here only on some of the many
results obtained in these papers. Let $G$ be a finite transitive graph
and write $V$ for the number of vertices of $G$. Let $p\in [0,1]$ and write
$G_p$ for the random graph obtained from $G$
by retaining each edge with probability $p$ and erasing it with probability $1-p$,
independently for all edges. We also write $\prob_p$ for this probability measure.
We say an edge is $p$-open ($p$-closed) if was retained
(erased). We say that a path in the graph is
$p$-open if all of its edges are $p$-open. For two
vertices $x,y$ we write $x \lr y$ for the event that there exists a $p$-open path
connecting $x$ and $y$. For an integer $j \geq 1$ we write $\C_j$ for the $j$th largest
component of $G_p$ (breaking ties arbitrarily) and for a vertex $v$ we write $\C(v)$
for the component in $G_p$ containing $v$.

For two vertices $x,y$ we denote
    \be \label{nabledef}
    \nabla_p(x,y) = \sum_{u,v} \prob_p(x \lr u)\prob_p(u \lr v)\prob_p(v \lr y) \, .
    \ee
The quantity $\nabla_p(x,y)$, known as the {\em triangle diagram}, was introduced by
Aizenman and Newman \cite{AN} to study critical percolation on high-dimensional infinite
lattices. In that setting, the important feature of an infinite graph $G$ is
whether $\nabla_{p_c}(0,0)<\infty$.
This condition is often referred to as the {\em triangle condition}.
In high-dimensions, Hara and Slade \cite{HS90}
proved the triangle condition. It allows to deduce
that numerous critical exponents attain the same values as they do on an infinite regular
tree, see e.g.\ \cite{BA91, AN, KN08, KN11}.

When $G$ is a finite graph, $\nabla_p(0,0)$ is obviously finite, however, there is still a {\em finite} triangle condition which in turn guarantees that random critical subgraphs of $G$ have the same geometry as random subgraphs of the complete graph on $V$ vertices, where $V$ denotes the number of vertices in
$G$. That is, in the finite setting the role of the infinite regular tree is played by the complete graph. Let us make this heuristic formal.

We always have that $V\to \infty$ and that $\lambda\in(0,1)$ is a
fixed constant. Let $p_c=p_c(\lambda)$ be defined by
    \be
    \label{pcdefgeneral}
    \E _{p_c(\lambda)} |\C(0)| = \lambda V^{1/3} \, .
    \ee
The finite triangle condition is the assumption that
$\nabla_{p_c(\lambda)}(x,y) \leq {\bf 1}_{\{x=y\}} + a_0$,
for some $a_0=a_0(\lambda)$ sufficiently small.
The strong triangle condition, defined in  \cite[(1.26)]{BCHSS2},
is the statement that there exists a constant $C$ such that
for all $p\leq p_c$,
    \be
    \label{finitetriangle}
    \nabla_{p}(x,y) \leq {\bf 1}_{\{x=y\}} + C\chi(p)^3/V + \alpham,
    \ee
where $\alpham \to 0$ as $m\to \infty$ and $\chi(p)=\E _{p} |\C(0)|$ denotes the expected cluster size.
Throughout this paper, we will assume that the strong triangle
condition holds. In fact, in all examples where the finite triangle
condition is proved to hold, actually the strong triangle
condition \eqref{finitetriangle} is proved. In \cite{BCHSS2},
\eqref{finitetriangle} is shown to hold for
various graphs: the complete graph,
the hypercube and high-dimensional tori $\Z_n^d$.
In particular, the next theorem states \eqref{finitetriangle}
for the hypercube.

\begin{theorem}[\cite{BCHSS2}]
\label{thm-BCHSS2}
Consider percolation on the hypercube $\{0,1\}^\m$.
% at $p=p_c$ defined in
%\eqref{pcdefgeneral}.
Then for any $\lambda$, there exists a constant $C=C(\lambda)>0$ such that for any $p\leq p_c(\lambda)$ (as defined in (\ref{pcdefgeneral}))
    \be
    \label{nabla-hamming-cube}
    \nabla_{p}(x,y) \leq {\bf 1}_{\{x=y\}}+ C(\chi(p)^3/V+1/\m) \, .
    \ee
\end{theorem}

%\todo{Discuss how we wish to refer to `unlacing the lace expansion'.}
As we discuss in Remark $4$ below Theorem \ref{mainthmgeneral},
our methodology can be used to yield a simple proof of Theorem \ref{thm-BCHSS2}
without relying on the lace-expansion methods derived in
\cite{BCHSS2}. The main effort in \cite{BCHSS1} is to show
that under condition (\ref{finitetriangle}) the phase transition
behaves similarly to the one in $G(n,p)$  described in the
previous section. The main results obtained in \cite{BCHSS1} are the following:\\

\noindent {\bf \em The critical window.} Let $G$ be a finite transitive graph
for which (\ref{finitetriangle}) holds. Then, for $p=p_c(1+O(V^{-1/3}))$,
    $$
    \prob (A^{-1} V^{2/3} \leq |\C_1| \leq AV^{2/3})= 1- O(A^{-1}) \, .
    $$

\noindent {\bf \em The subcritical phase.} Let $G$ be a finite transitive graph
for which (\ref{finitetriangle}) holds. Let $\eps=o(1)$ be a non-negative sequence
with $\eps \gg V^{-1/3}$ and put $p=p_c(1-\eps)$. Then,
for all fixed $\delta>0$,
    $$
    \prob( |\C_1| \leq (2+\delta)\eps^{-1} \log(\eps^3 V)) = 1-o(1) \, .
    $$

\noindent {\bf \em The supercritical phase.} Let $G$ be a finite transitive graph
for which (\ref{finitetriangle}) holds. Let $\eps=o(1)$ be a non-negative sequence
with $\eps \gg V^{-1/3}$ and put $p=p_c(1+\eps)$. Then,
    $$
    \prob(|\C_1| \geq A \eps V) = O(A^{-1}) \, .
    $$

%\todo{Check: Is this really the best BCHSS obtained in the supercritical regime?}
%% Yes, that's (1.24) in BCHSS1

Thus, while these results hold in a very general setting, they are incomplete.
Most notably, in the supercritical phase there is no matching
lower bound on $|\C_1|$. So, a priori, it is possible that $|\C_1|$
is still of order $V^{2/3}$ when $p=p_c(1+\eps)$
for some $\eps\gg V^{-1/3}$ and that the scaling window is in fact much larger than
$V^{-1/3}$. It remains an open problem to show whether (\ref{finitetriangle}) by itself
implies that $|\C_1|/(\eps V)$ converges in probability to
a constant in the supercritical phase.

As we mentioned before, the particular case of the hypercube was addressed in \cite{BCHSS3}. There the authors employed some of the result of \cite{BCHSS1, BCHSS2} together with a {\em sprinkling} argument to provide a lower bound of order $\eps 2^m$ on $|\C_1|$ valid only when $\eps \geq \e^{-cm^{1/3}}$. We will rely on the sprinkling method for the arguments in this paper, so let us briefly expand on it.

\subsection{Sprinkling}
\label{sec-intsprinkling}
The sprinkling technique was invented by Ajtai, Koml\'{o}s and Szemer\'{e}di
\cite{AKS82} to show that $|\C_1|=\Theta(2^\m)$ when $p=(1 +\eps)/\m$ for fixed $\eps>0$
and can be described as follows. Fix some small $\theta>0$ and write
$p_1 = {(1 + (1-\theta)\eps)/\m}$ and $p_2\geq \theta \eps/ \m$
such that $(1-p_1)(1-p_2)=1-p$. It is clear that $G_p$ is distributed
as the union of the edges in two independent copies of $G_{p_1}$
and $G_{p_2}$. The sprinkling method consists
of two steps. The first step is performed in $G_{p_1}$ and uses a branching
process comparison argument together with Azuma-Hoeffding concentration inequality to obtain
that \whp\ at least $c_2 2^{\m}$ vertices are contained in connected
components of size at least $2^{c_1 \m}$ for some small but fixed
constants $c_1, c_2>0$. In the second step we
add the edges of $G_{p_2}$ (these are the ``sprinkled'' edges) and show that they
connect many of the clusters of size at least $2^{c_1 \m}$ into a giant cluster
of size $\Theta(2^\m)$.

Let us give some details on how the last step is done. A key tool here is the
\emph{isoperimetric inequality} for the hypercube stating that two disjoint
subsets of the hypercube of size at least $c_2 2^{\m}/3$ have at least $2^\m/\m^{100}$ disjoint
paths of length $C(c_2)\sqrt{\m}$ connecting them, for some constant $C(c_2)$.
(The $\m^{100}$ in the denominator
is not sharp, but this is immaterial as long as it is a polynomial in $\m$.)
This fact is used in the following way. Write $V'$ for the set of vertices which
are contained in a component of size at least $2^{c_1 \m}$ in $G_{p_1}$ so that
$V'\geq c_2 2^\m$. We say that \emph{sprinkling fails} when
$|\C_1| \leq c_2 2^\m/3$ in the union $G_{p_1} \cup G_{p_2}$. If
sprinkling fails, then we can partition $V'=A \uplus B$ such that both
$A$ and $B$ have cardinality at least $c_2 2^\m/3$ and {\em any}  path of length
at most $C(c_2)\sqrt{\m}$ between them has an edge which is $p_2$-closed. The number
of such partitions is at most $2^{2^\m/2^{c_1 \m}}$. The probability that a path
of length $k$ has a $p_2$-closed edge is $1- p_2^k$. Applying the isoperimetric inequality
and using that the paths guaranteed to exist by it are disjoint so that
the edges in them are independent, the probability that
sprinkling fails is at most
    \be
    \label{easysprinkle}
    2^{2^\m/2^{c_1 \m}} \cdot
    \Big (1 - \big ( {\theta\eps \over \m} \big )^{C(c_2)\sqrt{\m}} \Big )^{2^\m/\m^{100}}
    = \e^{-2^{(1+o(1))\m}}\, ,
    \ee
which tends to 0.
%In BCHSS3 the first step in the approach was changed, and instead

\subsection{Revised sprinkling}
\label{sec-improv-sprinkling}
The sprinkling argument above is not optimal due to the use of the isoperimetric
inequality. It is wasteful because it assumes that large percolation clusters
can be ``worst-case'' sets, that is, sets which saturate the isoperimetric inequality (e.g., two balls of radius $\m/2-\sqrt{\m}$ around two vertices at Hamming distance $\m$).
However, it is in fact very improbable for percolation clusters to be similar to this kind of worst-case sets. Our approach replaces the use the isoperimetric inequality by proving statements showing that large percolation clusters are ``close'' to uniform random sets of similar size. It allows us to deduce that two large clusters share many closed edges with the property that if
we open even \emph{one} of them, then the two clusters connect. While previously we had paths of length $\sqrt{\m}$ connecting the two clusters, here we will have paths of length precisely $1$. The final line of our proof, replacing (\ref{easysprinkle}), will be
    \be
    \label{hardsprinkle}
    2^{{2\eps V /(\km\eps^{-2})}} \cdot \Big ( 1-{\theta\eps \over \m} \Big )^{\m\eps^2 V} \leq  \e^{-\theta \eps^3 V(1+o(1))} \, ,
    \ee
where $\km$ is some sequence with $\km\to \infty$ very slowly. This tends to $0$ since $\eps^3 V \to \infty$. Compared with the logic leading to (\ref{easysprinkle}), this line is rather suggestive. We will obtain that \whp\ $2\eps V$ vertices are in components of size at least $\km \eps^{-2}$, explaining the $2^{{2\eps V/(\km\eps^{-2})}}$ term in (\ref{hardsprinkle}). The main effort in this paper is to justify the second term showing that for any partition of these vertices into two sets of size $\eps V$, the number of closed edges between them is at least $\eps^2 \m V$ --- the same number of edges one would expect two uniform random sets of size $\eps V$ to have between them. Therefore, given a partition,
the probability that sprinkling fails for it is bounded by $\Big ( 1-{\theta\eps \over \m} \Big )^{\m\eps^2 V}$.

\subsection{The general theorem}
\label{sec-gen-thm}
Our methods use relatively simple geometric properties of the hypercube and apply to a larger set of underlying graphs. We present this general setting that the majority of the paper assumes here and briefly discuss some other cases for which our main theorem holds asides from the hypercube. We remark that the impatient reader may proceed assuming the underlying graph is always the hypercube $\{0,1\}^{\m}$ --- we have set up the notation to support this since the hypercube, in some sense, is our most ``difficult'' example.

The geometric conditions of our underlying graphs will be stated in terms of random walks. The main advantage of this approach is that these conditions are relatively easy to verify. Let $G$ be a finite transitive graph on $V$ vertices and degree $m$. Consider the non-backtracking random walk on it (this is just a simple random walk not allowed to traverse back on the edge it just came from, see Section \ref{sec-NBW} for a precise definition). For any two vertices $x,y$, we put $\p^t(x,y)$ for the probability that the walk started at $x$ visits $y$ at time $t$. For any $\xi>0,$ we write $\Tm(\xi)$ for the $\xi$-{\em uniform mixing time} of the walk, that is,
    $$
    \Tm(\xi) = \min \Big \{ t \colon  \max_{x,y} \,\,{\p^t(x,y) + \p^{t+1}(x,y) \over 2}
    \leq (1+\xi) V^{-1} \Big \} \, .
    $$
%
%| 1- V \p^t(x,y) | \leq \xi \big \} \, ,$$
%and when $G$ is bipartite (as in the case of the Hamming cube) with partitions $G=G_1 \uplus G_2$ (note that since $G$ is transitive we must have that $|G_1|=|G_2|=V/2$)
%\begin{eqnarray*} \Tm(\xi) = \max \Big ( &\min& \big \{ t \hbox{ {\rm even}} : \max_{x,y \in G_1^2 \cup G_2^2} | 1- V \p^t(x,y)/2| \leq \xi \big \} \\ &\min& \big \{ t \hbox{ {\rm odd}} : \max_{x,y \in G_1 \times G_2 \cup G_2 \times G_1} | 1- V \p^t(x,y)/2| \leq \xi \big \} \Big ) \, . \end{eqnarray*}
%Usually one resolves the pesky matter of periodicity by either considering continuous time random walks or introducing laziness (the walker stays put each step with probability $1/2$), see \cite{LPW}. However, in our scenario we will use random walk estimates in a combinatorial fashion in order to count paths and we therefore must handle the periodicity issue directly.
%We assume the degree of $G$ is $m$ and we that $p_c$ is defined as in (\ref{pcdef}).

\begin{theorem}
\label{mainthmgeneral}
Let $G$ be a transitive graph on $V$ vertices with degree $\m$ and define $p_c$ as in (\ref{pcdef})
with $\lambda=1/10$. Assume that there exists a sequence $\alpham=o(1)$ with $\alpham \geq 1/\m$ such that if we put
$\mnot = \Tm(\alpham)$, then the following conditions hold:
\begin{enumerate}
\item $\m \to \infty$,
\item $[p_c(\m-1)]^{\mnot} = 1 + O(\alpham)$,
\item For any vertices $x,y$,
    $$
    \sum_{u,v} \sum_{\substack{t_1,t_2,t_3=0\\t_1+t_2+t_3\geq 3}}^{\mnot} \p^{t_1}(x,u) \p^{t_2}(u,v) \p^{t_3}(v,y) = O(\alpham/\log{V}).
    $$
%\item We have
%$$ \sum_{v} \sum_{t=2}^{\mnot} \sum_{s=1}^{t} s \p^t(0,v) \p^s(0,v) = O(\log^{-2} V) \, .$$
\end{enumerate}
Then,
\begin{enumerate}
\item[(a)] the finite triangle condition (\ref{finitetriangle}) holds (and hence
the results in \cite{BCHSS1} described in Section \ref{sec-bchss} hold),
\item[(b)] for any sequence $\eps=\epsm$ satisfying $\epsm \gg V^{-1/3}$ and $\epsm = o(\mnot^{-1})$,
    $$
    \frac{|\C_1|}{2\epsm V}\convp 1 \,  , \quad\qquad \E |\C(0)| = (4+o(1))\epsm^2 V \,  , \quad\qquad \frac{|\C_2|}{\epsm V}\convp 0 \, . \\
    $$
\end{enumerate}
\end{theorem}

\noindent {\bf Remark 1.} In the case of the hypercube $\{0,1\}^m$ we will take $\alpham=\m^{-1} \log{\m}$ and verify the conditions of Theorem \ref{mainthmgeneral}. This is done in Section \ref{sec-sprinkling}. Although the behavior of random walk on the hypercube is well understood, we were not able to find an estimate on the uniform mixing time yielding $\Tm(\m^{-1} \log m) = \Theta(\m \log \m)$ in the literature. To show this we use the recent paper of Fitzner and the first author \cite{FitHof11a} in which the non-backtracking walk transition matrix on the hypercube is analyzed. We use this result in Lemma \ref{lem-NBW-high-d-tori} to verify condition (3) and condition (2) follows directly from \eqref{pcest} (though condition (2) can be verified by elementary means without using \eqref{pcest}, see Remark 4).\\

\noindent {\bf Remark 2.} Note that part (b) of Theorem \ref{mainthmgeneral} only applies when
$\epsm = o(\mnot^{-1})$ and not for any $\epsm=o(1)$. Thus, for a complete proof of
Theorem \ref{mainthm}, we also require a separate argument dealing with the regime
$\epsm \geq c\mnot^{-1}$ --- in the case of the hypercube and other graphs mentioned
in this paper, this is a much easier regime in which previous
techniques based on sprinkling and
isoperimetric inequalities are effective. \\

\noindent {\bf Remark 3.} Random walk conditions for percolation on finite graphs
were first given by the second author in \cite{Nach09}. The significant difference between the two approaches is that in \cite{Nach09} the condition requires controlling the random walk behavior of a period of time which is as long as the critical cluster diameter, that is, $V^{1/3}$. The outcome is that the results of \cite{Nach09} only apply when $p_c = (1+O(V^{-1/3}))/(\m-1)$ and hence do not apply in the case of the hypercube. Here we are only interested in the behavior of the random walk up to the mixing time, even though that typical percolation paths are much longer. The reason for this is that it turns out that it is enough to randomize the beginning of a percolation path in order to obtain that the end point is uniformly distributed, see Section \ref{sec-unif-conn-NBW}. Another difference is that the results in \cite{Nach09} only show that $|\C_1| \geq c\eps V$ for some $c>0$ and do not give the precise asymptotic value of $|\C_1|$ as we do here.\\

%%%%% We are no longer remarking this:
%$\todo{Discuss how we wish to refer to `unlacing the lace expansion'.}
\noindent {\bf Remark 4.} Our approach also enables us to give a
simple proof for the fact that the finite triangle condition (Theorem \ref{thm-BCHSS2})
holds for the hypercube without using the lace expansion as in \cite{BCHSS2}. Our proof of this fact relies on the estimate $p_c = 1/(\m-1) + O(\m^{-3})$ (which is much weaker than (\ref{pcdef}) but also much easier to prove) and on the argument presented in Section \ref{sec-unif-conn-NBW}. We defer this derivation to a future publication.
In fact, in this paper we only rely on this easy
estimate for $p_c$ so our main results here, Theorem \ref{mainthm}, are
in fact self-contained and do not rely at any time on results obtained
via the lace expansion in \cite{BCHSS2} (we do use arguments of \cite{BCHSS1} which
rely on the triangle condition). We hope that this may be useful in future attempts
to ``unlace'' the lace expansion.\\

%
%With Theorem \ref{mainthmgeneral} at hand,
%we can directly show Theorem \ref{mainthm} \emph{without} relying on the asymptotics for $p_c$
%in \eqref{pcest} and proved in \cite{HS05, HS06} nor on the lace-expansion results
%from \cite{BCHSS2}. For this, we show that $\expec_p|\cluster(0)|\gg 2^{-\m/3}$ for
%$p=1/(\m-1)+C/\m^3$ and $C>0$ sufficiently large, so that $p_c\leq 1/(\m-1)+C/\m^3$.
%This is achieved by a comparison to percolation on the $\m$-ary tree combined with
%a sprinkling argument. With the bound on $p_c$ at hand, we can directly apply
%Theorem \ref{mainthmgeneral}. The details of this proof are presented in
%Appendix .\\

% We proceed by discussing several examples to which

In many cases, verifying the conditions of Theorem \ref{mainthmgeneral} is done using known methods from the theory of percolation and random walks (note that condition (2) involves both a random walk and a percolation estimate). We illustrate how to perform this in Section \ref{sec-sprinkling} in the case of the hypercube (thus proving Theorem \ref{mainthm}) and for expander families of high degree and high girth (see \cite{HLW} for an introduction to expanders). This is a class of graphs that contains various examples such as Payley graphs (see e.g.\ \cite{CGW89}), products of complete graphs $K_n^d$ and many others. Percolation on products of complete graphs were studied in \cite{HL, HLS, Nach09} in the cases $d=2,3$; our expander theorem allows us to provide a complete description of the phase transition in any fixed dimension $d$, answering a question posed in \cite{HL}. Recall that a sequence of graphs $G_n$ is called an {\em expander} family if there exists a constant $c>0$ such that the second largest eigenvalue of the transition matrix of the simple random walk is at most $1-c$ (the largest eigenvalue is $1$). Also, the girth of a graph is the length of the shortest cycle. It is a classical fact that on expanders $\Tm(V^{-1}) = O(\log V)$, where $V$ is the number of vertices of the graph, see e.g.\ \cite[below (19)]{ABLS}.

\begin{theorem}\label{expander} Let $G_\m$ be a transitive family of expanders with degree $\m\to \infty$ and $V$ vertices. Assume that $\m \geq c \log V$ and that the girth of $G$ is at least
$c \log_{m-1}{V}$ for some fixed $c>0$. Then the conditions of Theorem \ref{mainthmgeneral} hold and hence the conclusions of that theorem hold.
\end{theorem}

For products of complete graphs $K_n^d$, the girth equals 3, $V=n^d$ and $m=d(n-1)$, so that  the girth assumption is satisfied
for $c\leq 3(1-o(1))/d$ and $n$ sufficiently large. 
Theorem \ref{mainthmgeneral} applies to other examples of graphs, not included in the last theorem, for example, products of complete graphs $K_n^d$ where $d$ may depend on $n$ (as long as $n+d \to \infty$) and finite tori ${\mathbb Z}_n^d$ but only when $d=d(n)$ grows at some rate with $n$. We omit the details since they are rather similar. We emphasize, however, that there are important examples which are methods are insufficient to solve. Most prominently are bounded degree expanders with low girth and finite tori ${\mathbb Z}_n^d$ where $d$ is large but fixed. It seems that new ideas are required to study percolation on these graphs, see Section \ref{sec-open}.

\subsection{Organization}
\label{sec-organisation}
This paper is organized as follows. In Section \ref{sec-overview} we give an
overview of our proof, stating the main results upon which the
proof is based. In Section \ref{sec-prelim} we prove several estimates on
the number of vertices satisfying various properties, such as having large clusters,
or surviving up to great depth. We further prove detailed estimates on connection
probabilities. In Section \ref{sec-exp-vol-est} we prove expected volume estimates,
both in the critical as well as in the supercritical regime. In Section \ref{sec-intrin-reg}
we prove an intrinsic-metric regularity theorem, showing that most vertices that survive
long and have a large cluster size have neighborhoods that are sufficiently regular.
In Section \ref{sec-large-clusters-close} we show that most large clusters
have many closed edges between them, which is the main result in our proof.
In Section \ref{sec-sprinkling} we perform the improved sprinkling argument
as indicated in Section \ref{sec-improv-sprinkling}, and complete the
proof of Theorem \ref{mainthm}. In Section \ref{sec-open} we discuss several open problems. We close the paper in
In Appendix \ref{sec-clustertail} we sharpen the arguments
in \cite{BA91} and \cite{BCHSS2} to obtain the asymptotics of the supercritical
cluster tail.

\subsection{Acknowledgements} %Lemma \ref{similar} is a variant on an
%unpublished argument of Gady Kozma. We thank Gady for his permission to
%use it here.
The work of RvdH was supported in part by the Netherlands
Organization for Scientific Research (NWO).
The work of AN was partially supported by NSF grant \#6923910.
RvdH warmly thanks the Theory Group
at Microsoft Research Redmond for the hospitality during his visit in January 2010,
where this work was initiated, the Department of Mathematics at M.I.T.\ for
their hospitality during his visit in April 2011, and the Department of
Mathematics at U.B.C.\ for their hospitality during
his visit in November 2011. We thank Gady Kozma for useful discussions and Gordon Slade for
valuable comments on a preliminary version of this paper.
AN thanks Christian Borgs for introducing him to \cite[Conjecture 3.2]{BCHSS3} when he was an intern at
Microsoft Research in the summer of 2005.

\section{Overview of the proof}
\label{sec-overview}
In this section we give an overview of the key steps in our proof. Throughout the rest of the paper we assume that $\eps=\epsm$ is a sequence
such that $\eps=o(1)$ but $\eps^3 V \to \infty$.

\subsection{Notations and tools}
\label{sec-notation}
Let $G$ be a transitive graph and recall that $G_p$ is obtained from
$G$ by independently retaining each edge with probability $p$.
Recall that $x \lr y$ denotes that there exists a $p$-open path
connecting $x$ and $y$. We write $d_{G_p}(x,y)$ for the length of a shortest
$p$-open path between $x,y$ and put $d_{G_p}(x,y)=\infty$ if $x$ is not connected
to $y$ in $G_p$. We write $x \lrr y$ if $d_{G_p}(x,y) \leq r$, we write
$x \stackrel{=r}{\lrfill} y$ if $d_{G_p}(x,y)=r$ and $x \stackrel{[a,b]}{\lrfill} y$
if $d_{G_p}(x,y) \in [a,b]$. Further, we write $x \stackrel{P[a,b]}{\lrfill} y$
when there exists an open path of length in $[a,b]$ between $x$ and $y$ (not
necessarily a shortest path). The event
$\{x \stackrel{[a,b]}{\lrfill} y\}$ is not increasing with respect to adding edges,
but the event $\{x \stackrel{P[a,b]}{\lrfill} y\}$ \emph{is}, which often makes
it easier to deal with. Whenever the sign $\lr$ appears it will be clear
what $p$ is and we will drop it from the notation. The {\em intrinsic metric} ball of
radius $r$ around $x$ and its boundary are defined by
    $$
    B_x^{G}(r) = \{ y \colon d_{G_p}(x,y) \leq r\} \, , \qquad
    \partial B_x^{G}(r) = \{ y \colon d_{G_p}(x,y)=r\} \, .
    $$
Note that these are random sets of the graph and not the balls in shortest path metric
of the graph $G$. We often drop the $G$ from the above notation and write $B_x(r)$ when
it is clear what the underlying graph $G$ is. We also denote
    $$
    B_x^{G}([a,b])= \{ y \colon x \stackrel{[a,b]}{\lrfill} y \} \, .
    $$
Our graphs always contain a marked vertex that we call {\em the origin} and denote by $0$.
In the case of the hypercube this is taken to be the all $0$ vector. We often drop $0$
from notation and write $B(r)$ for $B_0(r)$ whenever possible.

We now define the intrinsic metric {\em one-arm} event. This was introduced
in \cite{NP3} to study the mixing time of critical $G(n,p)$ clusters and was very useful
in the context of high-dimensional percolation in \cite{KN08}. Define the event
    $$
    H^{G}(r) = \big \{ \partial B_0^{G}(r) \neq \emptyset \big \} \, ,
    $$
for any integer $r\geq 0$ and
    \be
    \label{gammadef}
    \Gamma(r) = \sup _{G' \subset G} \prob(H^{G'}(r)) \, ,
    \ee
where the supremum is over all subgraphs $G'$ of $G$.

The reason for the somewhat unnatural definition of $\Gamma$ above is the fact
that the event $\partial B_0^G(r) \neq \emptyset$ is not monotone with
respect to the addition of edges. Indeed, turning an edge from closed to open may
shorten a shortest path rendering the configuration such that the event
$\partial B_0^{G}(r) \neq \emptyset$ no longer occurs.

The following theorem studies the survival probability and
expected ball sizes at $p_c$, and is the finite graph analogue of a
theorem of Kozma and the second author \cite{KN08}.
The proof is almost identical to the one in \cite{KN08}
and is stated explicitly in \cite{HH2, KN2}.

\begin{theorem}[Volume and survival probability \cite{KN08}]
\label{KozmaNachmias}
Let $G$ be a finite transitive graph on $V$ vertices such that the finite
triangle condition (\ref{finitetriangle}) holds, and consider percolation on $G$
at $p=p_c(\lambda)$ with any $\lambda>0$. Then there exists a constant $C=C(\lambda)>0$
such that for any $r>0$,
\begin{enumerate}
\item $ \E |B(r)| \leq Cr$, and
\item $\Gamma(r) \leq C/r \, .$
\end{enumerate}
\end{theorem}

We often need to consider percolation performed at different values of $p$. We
write $\prob_{p}$ and $\E_p$ for the probability distribution and the corresponding
expectation operator with parameter $p$ when necessary.
Furthermore, we sometimes need to consider
percolation configurations at different $p$'s on the same probability space. This is
a standard procedure called the {\em simultaneous coupling} and it works as follows.
For each edge $e$ of our graph $G$, we draw an independent uniform random variable
$U(e)$ in $[0,1]$. We say that the edge $e$ receives the value $U(e)$. For any
$p\in [0,1]$, the set of $p$-open edges is distributed precisely as $\{e \colon U(e) \leq p\}$.
In this way, $G_{p_1} \subset G_{p_2}$ with probability $1$ whenever
$p_1 \leq p_2$.

\subsection{Tails of the supercritical cluster size}
\label{sec-super-clus-tail-over}
We start by describing the tail of the cluster size in the supercritical regime. Note that
the following theorem requires only the finite triangle condition, and not the
stronger assumptions of Theorem \ref{mainthmgeneral} and so the restriction
$\eps = o(\mnot)$ is not needed.

\begin{theorem}[Bounds on the cluster tail]
\label{clustertail}
Let $G$ be a finite transitive graph of degree $\m$ on $V$ vertices such that the finite
triangle condition (\ref{finitetriangle}) holds and put $p=p_c(1+\eps)$
where $\eps = o(1)$ and $\eps \gg V^{-1/3}$.
%Assume that $\m$ is sufficiently large.
Then, for any $k$ satisfying $k \gg \eps^{-2}$
    \be
    \label{ub-cluster-tail}
    \prob ( |\C(0)| \geq k ) \leq 2 \eps \big(1+O(\vep+(\vep^3V)^{-1}+(\vep^2k)^{-1/4}+\alpham)\big) \, ,
    \ee
and, for the sequence $k_0=\eps^{-2} (\eps^3 V)^{\alpha}$ for any $\alpha\in(0,1/3)$,
there exists a $c=c(\alpha)>0$ such that
    \be
    \label{lb-cluster-tail}
    \prob ( |\C(0)| \geq k_0 ) \geq 2 \eps \big(1+O(\vep+(\eps^3V)^{-c}+\alpham)\big) \, .
    \ee
\end{theorem}

Theorem \ref{clustertail} is reminiscent of the fact that a branching process
with Poisson progeny  distribution of mean $1+\eps$ has survival
probability of $2\eps(1+O(\vep))$ when $\eps=o(1)$. Upper and lower
bounds of order $\eps$ for the cluster tail were proved already in
\cite{BCHSS2} using Barsky and Aizenman's differential inequalities \cite{BA91}.
However, to get the precise constant $2$ we need to sharpen these differential
inequalities and handle some error terms in them that were neglected in the past.
This derivation and the proof of Theorem \ref{clustertail} are presented in
Appendix A. Its proof is entirely self-contained.

Let $Z_{\sss \geq k}$ denote the number of vertices with cluster size at least $k$, i.e.,
    \be
    Z_{\sss \geq k} = \big | \big \{ v \colon |\C(v)| \geq k\big \} \big | \, .
    \ee
We use Theorem \ref{clustertail} to show that $Z_{\sss \geq k_0}$, with $k_0$ as
in the theorem, is concentrated. This advances us towards the first term on the left hand side of (\ref{hardsprinkle}).

\begin{lemma}[Concentration of $Z_{\sss \geq k_0}$]
\label{lem-conc-Z-I}
In setting of Theorem \ref{clustertail}, if $m\to \infty$, then
    $$
    \frac{Z_{\sss \geq k_0}}{2\eps V}\convp 1,
    $$
and
    $$
    \E |\C(0)| \leq (4+o(1))\eps^2 V \, .
    $$
\end{lemma}

\subsection{Many boundary edges between large clusters}
\label{sec-bound-edges-large}
The term $\big (1-{\theta\eps \over \m} \big )^{\m\eps^2 V}$ in (\ref{hardsprinkle})
suggests that after partitioning the large clusters into two sets of order $\eps V$
vertices, as we did before, the number of closed edges connecting them is of order $\m \eps^2 V$. 
%The number $1-{\theta\eps \over \m}$ is the probability of an edge remaining closed in the sprinkling phase. 
This is
the content of Theorem \ref{alphapairs} below which is the main effort of this paper.
It is rather intuitive if one believes that large clusters are uniform random sets. Indeed,
let $v$ be a vertex in one of the sets of the partition. It has degree $\m$ and hence
we expect $\m \eps$ of these neighbors to belong to the second set of the partition.
Summing over all vertices $v$ we obtain of the order $\eps^2 \m V$ edges. Making this a
precise statement requires some details which we now provide. \\

We work under the general assumptions of Theorem \ref{mainthmgeneral}. In particular, we are given sequences $\epsm,\alpham$ such that both are $o(1)$ and $\epsm^3 V \to \infty$ and $\alpham \geq 1/\m$. Without loss of generality we assume that
\be\label{alphamcond} \alpham \geq (\epsm^3 V)^{-1/2} \, ,\ee
otherwise we replace the original $\alpham$ by $(\epsm^3 V)^{-1/2}$, and note that in both cases $\alpham =o(1)$ and satisfies $\alpham \geq 1/m$.

Let us start by introducing some notation. For vertices $x,y$ and radii $j_x,j_y$, we define the event
   \eqn{
    \label{Acal-def}
    \A(x,y,j_x,j_y)=\{\partial B_x(j_x) \neq \emptyset, \partial B_y(j_y)\neq \emptyset
    \mbox{ and } B_x(j_x) \cap B_y(j_y) = \emptyset\}
    }
%no need to repeat this with words
%for the event that $\partial B_x(j_x)$ and $\partial B_y(j_y)$ are non-empty and
%$B_x(j_x)$ and $B_y(j_y)$ are disjoint.
Intuitively, if $\A(x,y,j_x,j_y)$ occurs
for $j_x$ and $j_y$ sufficiently large, then $x$ and $y$ are both in the giant
component. The event $\A(x,y,j_x,j_y)$ plays a central role throughout our paper.
We continue by choosing some parameters. The role of each will become clear later. We put
    \be \label{Mchoice}
    M=\Mm=\log\log\log(\eps^3V \wedge \alpham^{-1} \wedge (\eps \mnot)^{-1})
    \qquad
    \mbox{and}
    \qquad
    r=\rmm=\Mm\eps^{-1} \, .
    \ee
%Our choice of the next radius $r_0$ is more subtle and requires showing that this choice is
%possible.
Note that $\Mm \to \infty$ in our setting. We choose $r_0$ by
    \begin{eqnarray}
    \label{rzerocond}
    r_0 = {\eps^{-1} \over 2}\log(\alpham\eps^3 V)\, .
    \end{eqnarray}
In Corollary \ref{r0ball}
we prove that $\E |B(j)| = \Theta(\eps^{-1} (1+\eps)^j)$ as long as
$j \leq \eps^{-1} [\log (\eps^3 V)-4\log\log(\eps^3 V)]$ --- the same asymptotics as in a
Poisson $(1+\eps)$ branching process (though in the branching process
the estimate is valid for all $j$). This implies that $\E|B(r_0)|
= \Theta\big (\sqrt{\alpham \eps V}\big)$, a fact that we use
throughout the paper, but we do not use this right now.

%%% These are the actual properties of r_0 that we are using throughout the paper.
%\begin{enumerate}
%\item $(\E|B(r_0)|)^4 \gg V$ - used in Lemma \ref{squarediagram}.
%\item $\E|B(r_0)| \ll \eps^2 V$ - used in Lemma \ref{lowerbound}.
%\item $\E|B(r_0)|^2 \ll \eps V$ - used in Lemma \ref{lowerboundfinal}, and \ref{sedgefirstmomerror}.
%\item $r_0 \gg \eps^{-1}$ - used in Lemma \ref{morefitnessone}.
%\item $\E|B(r_0)| \leq V\eps / r_0$ - used in Lemma \ref{sedgefirstmom}.
%\end{enumerate}

For vertices $x,y$ and radius $\ell$ we define
    \begin{eqnarray*}
    S_\ell(x,y) = \big | \big \{ (u,u')\in E(G) &\colon& \{x \stackrel{\ell}{\lrfill} u\} \circ \{y \stackrel{\ell}{\lrfill} u'\} \, , \\ &\ & |B_u(2r+r_0)|\cdot |B_{u'}(2r+r_0)| \leq \e^{40M} \eps^{-2} (\E |B(r_0)|)^2 \big \} \big | \, .
    \end{eqnarray*}
The edges counted in $S_\ell(x,y)$ are the ones that we are going to sprinkle.
Informally, a pair of vertices $(x,y)$ is \emph{good} when their clusters
are large and $S_{\ell}(x,y)$ is large, so that their clusters have many bonds
between them. We make this quantitative in the following definition:

\bde[$(r,r_0)$-good pairs]\label{goodpairs}
We say that $x,y$ are $(r,r_0)$-good if all of the following occur:
\begin{enumerate}
\item $\A(x,y,2r,2r)$,
%\item $|B_x(2r+r_0)| \leq \eps^{-1}(1+\eps)^{3r}\E|B(r_0)|$ and $|B_y(2r+r_0)| \leq \eps^{-1} (1+\eps)^{3r}\E|B(r_0)|$,
\item $|\C(x)|\geq (\eps^3 V)^{1/4} \eps^{-2}$ and $|\C(y)| \geq (\eps^3 V)^{1/4} \eps^{-2}$,
\item $S_{2r+r_0}(x,y) \geq (\log M)^{-1} V^{-1} \m \eps^{-2} (\E|B(r_0)|)^2$.
\end{enumerate}
Write $\Park$ for the number of $(r,r_0)$-good pairs.
\ede

%We already know that, with high probability, there are roughly $(2\eps V)^2$
%pairs of vertices for which $\A(x,y,2r+r_0,2r+r_0)$ occurs. Our main goal for this %section is to prove the following theorem, which shows that most of these pairs in %fact also satisfy requirement (4) above.

\begin{theorem} [Most large clusters share many boundary edges]
\label{alphapairs}
Let $G$ be a graph on $V$ vertices and degree $\m$ satisfying the assumptions in Theorem \ref{mainthmgeneral}. Assume that $\eps=\eps_m$ satisfies
    \be
    \label{epscond}
     \epsm \gg V^{-1/3}
     \qquad
     \and
     \qquad
     \epsm = o(\mnot^{-1}) \, ,
    \ee
as in part (b) of Theorem \ref{mainthmgeneral}. Take $M$ and $r= M \eps^{-1}$ as in \eqref{Mchoice}, and $r_0$ as in \eqref{rzerocond}.
Then,
    $$
    \frac{\Park}{(2\eps V)^2}\convp 1\, .
    $$
\end{theorem}

In light of Theorem \ref{clustertail}, we expect that the number of pairs
of vertices $(x,y)$ with $|\C(x)|\geq (\eps^3 V)^{1/4}\eps^{-2}$ and
$|\C(y)| \geq (\eps^3 V)^{1/4}\eps^{-2}$
is close to $(2\eps V)^2$. Theorem \ref{alphapairs} shows that the majority of these pairs
have clusters that share many edges between them. This allows us to proceed with
the sprinkling argument leading to \eqref{hardsprinkle} and we perform this
in Section \ref{sec-improved-sprinkling-performed} leading to the proof of Theorem \ref{mainthmgeneral}. Since the latter proof assumes only Theorem \ref{alphapairs}, the curious reader
may skip now to Section \ref{sec-improved-sprinkling-performed} to see how this is done.

%\noindent
%\todo{Remco: This is just for ourselves to keep notation consistent.
%Remove this afterwards.}
%{\bf Notation:} I will use the following notation:\\
%$\Mm$ is the quantity in \eqref{Mchoice}, $M$ is a generic value that mostly satisfies
%$M\leq \Mm$;\\
%$\rmm$ is the quantity in \eqref{Mchoice}, $r$ is a generic value that mostly satisfies
%$r\leq \rmm$ and is related to $M$ above by $r=M/\eps$;\\
%$r_0$ is the quantity defined in \eqref{rzerocond}, $r_1$ is a generic value
%that mostly satisfies
%$r_1\leq r_0$.

\subsection{Uniform connection bounds and the role of the random walk}
\label{sec-unif-conn-NBW}
We briefly expand here on one of our most useful percolation inequalities and
its connection with random walks. In the analysis of the Erd\H{o}s-R\'enyi random graph
$G(n,p)$ symmetry plays a special role.
One instance of this symmetry is that the function $f(x) = \prob(0 \lr x)$ is constant whenever $x \neq 0$ and its value is
precisely $(V-1)^{-1}(\E|\C(0)|-1)$ and $1$ when $x=0$. Such a statement
clearly does not hold on the hypercube at $p_c$: the probability that
two neighbors are connected is at least $\m^{-1}$ (recall (\ref{pcest})) and the probability that
$0$ is connected to one of the vertices in the barycenter of the cube is
at most $\sqrt{\m} 2^{-\m} \E|\C(0)|$ by symmetry.

A key observation in our proof is that one can recover this symmetry as long as we
require the connecting paths to be longer than the mixing time of the random walk. A precise statement is that percolation at
$p_c$ on any graph $G$ satisfying the assumptions of Theorem \ref{mainthmgeneral},
    \be
    \label{exampleunif}
    \prob (0 \stackrel{[\mnot, \infty)}{\lrfill} x) \leq (1+ o(1)){\E |\C(0)|\over V}\, ,
    \ee
where $\mnot$ is uniform mixing time, as defined in Theorem \ref{mainthmgeneral}.
This is the content of Lemma \ref{uniformconnbd}
(or rather by taking $r\to \infty$ in the lemma).
%It is also important to have that
%    $$
%    \sum_{x} \prob(0 \stackrel{P[\mnot, \infty)}{\lrfill} x) \geq (1-o(1)) \E |\C(0)| \, ,
%    $$
%and this indeed occurs since typical distance at $p_c$ are of order $v^{1/3}$ which is much bigger than $\mnot$ in the $m$-cube and most of our examples.
%We deduce that the function $\prob (0 \stackrel{P[\mnot, \infty)}{\lrfill} x)$ is approximately
%constant in $x$.
This gives us an advantage at estimating difficult sums
such as $\nabla_{p}(0,0)$ as in (\ref{nabledef}), see Section \ref{sec-nbwtriangle}.

\subsection{Sketch of proof of Theorem \ref{alphapairs}}\label{sec-sketch} The difficulty in Theorem \ref{alphapairs}
is the requirement (3) in Definition \ref{goodpairs}. Indeed, conditioned on survival (that is,
on the event $\A(x,y,2r,2r)$), the random variable $S_{2r+r_0}(x,y)$ is not concentrated and
hence it is hard to prove that it is large with high probability. In fact, even the variable
$|B(r_0)|$ is not concentrated. This is not a surprising fact: the number of descendants at
generation $n$ of a branching process with mean $\mu > 1$ divided by $\mu^n$ converges as
$n\to \infty$ to a non-trivial random variable. Intuitively, this non-concentration
occurs because the first generations of the process have a strong and lasting effect
on the future of the population.

In order to counteract this, we condition on the event $\A(x,y,r,r)$ {\em and} on
the entire balls $B_x(r)$ and $B_y(r)$ including all the open and closed edges
touching them (during the actual proof we will use some other radii $j_x,j_y$
between $r$ and $2r$ but this is a technical matter). We will prove
that given this conditioning the variable $S_{r+r_0}(x,y)$ is concentrated around the value
    \be
    \label{sketch.expectedval}
    |\partial B_x(r)| |\partial B_y(r)| V^{-1} m (\E|B(r_0)|)^2 \, ,
    \ee
and that $|\partial B_x(r)||\partial B_y(r)|\geq \eps^{-2}$ with high probability,
yielding that requirement (3) in Definition \ref{goodpairs} occurs with high probability
conditioned on the event above. Our choice of $r_0$ in \eqref{rzerocond} is made in such
a way that the above quantity is large (however, later we will see that $r_0$ cannot be too large). Let us elaborate on the estimate (\ref{sketch.expectedval}). Assume
that $B_x(r)=A$ and $B_y(r)=B$ and write $\pab$ and $\eab$ for the conditional probability
and expectation given $B_x(r)=A$ and $B_y(r)=B$. We have
    $$
    \eab S_{r+r_0}(x,y) \approx \sum_{a \in \partial A, b\in \partial B} \,\,
    \sum_{(u,u')} \pab(\{a \lref u\} \circ \{b \lref u'\}) \, ,
    $$
where we did not write equality because a) we ignored the second condition in the definition of $S_{r+r_0}(x,y)$ on $|B_u(2r+r_0)|\cdot |B_{u'}(2r+r_0)|$, and b) some edges $(u,u')$ may be over-counted
in the sum, and c) we have neglected to count the closed edges $(u,u')$ that
connect $A$ and $B$ (that is, occurring in height smaller than $r$).
However, it turns out that all of these contributions are small compared to (\ref{sketch.expectedval}).
It is a standard matter by now to manipulate, using the triangle condition, to obtain that
for {\em most} edges $(u,u')$
    $$
    \pab(\{a \lref u\} \circ \{b \lref u'\}) \approx \pab(a \lref u) \pab(b \lref u') \, ,
    $$
so in order to proceed we need a good lower bound on $\pab(a \lref u)$. Lemma \ref{uniformconnbd}
(see also (\ref{exampleunif})) immediately gives that $\prob(a \lref u) \geq (1-o(1)) V^{-1} \E|B(r_0)|$
for {\em most} vertices $u$ (since $\sum_u \prob(a\lref u) = \E|B(r_0)|$). Had we had the same estimate for $\pab(a \lref u)$, the lower bound on the conditional first moment required to prove the estimate (\ref{sketch.expectedval}) would follow immediately. However, the probability $\pab(a \lref u)$ may heavily depend on the sets $A$ and $B$. 

To that aim, in Section \ref{sec-intrin-reg}  we present an {\rm intrinsic metric regularity}
theorem, similar in spirit to the extrinsic metric regularity theorem presented in
\cite{KN11}. Roughly, it states that for {\em most} sets $A$ (more precisely, the weight
of sets not having this is $o(\eps)$) for which $B_x(r)=A$ satisfies
$\partial B_x(r) \neq \emptyset$, we have that {\em most} vertices
$a \in \partial A$ satisfy
    $$
    \sum_{u} \pa (a \lref u) \geq (1-o(1)) \E |B(r_0)| \, ,
    $$
where $\pa$ is the conditional probability given $B_x(r)=A$. Thus, the expected size of
the ``future'' of most vertices on the boundary is not affected by the conditioning on
a typical ``past''.

At this point comes another crucial application of the uniform connection bounds
described in the section above. Indeed, even if the expected ``future'' of a
vertex has the same asymptotics with or without conditioning, we cannot a
priori rule out the possibility that this conditional ``future'' concentrates
on a small remote portion of the underlying graph $G$ --- this can potentially
violate the concentration around the value in (\ref{sketch.expectedval}).
However, our uniform connection bounds stated in Lemma \ref{uniformconnbd}
are robust enough to deal with conditioning and immediately imply that
$\pa(a \lref u) = (1-o(1)) V^{-1} \E|B(r_0)|$ for most $a$ in $\partial A$
and for most vertices $u$. In other words, not only did the conditioning
not influence the size of the ``future'', it also left its distribution
approximately unaltered. These considerations allow us to give a
lower bound of (\ref{sketch.expectedval}) on the conditional
expectation. This and the conditional second moment calculation required to show concentration are
performed in Section \ref{sec-large-clusters-close}.

\section{Preliminaries}
\label{sec-prelim}
In this section we provide some preliminary results that we will use.
These involve various expectations and probabilities
related to the random variable $|\partial B(r)|$ in Section \ref{sec-survprob} and
\ref{sec-disj-surv-prob}, non-backtracking random walks in Section \ref{sec-NBW}
and its relation to uniform bounds for connection probabilities in
Section \ref{sec-unif-bds-conn}. In Section \ref{sec-nbwtriangle} we
use these results to prove part (a) of Theorem \ref{mainthmgeneral}.
Finally, in Section \ref{sec-triangle-square} we bound triangle
and square diagrams. The results in this section do {\em not} rely
on the assumptions of Theorem \ref{mainthmgeneral} but sometimes we do
assume the finite triangle condition \eqref{finitetriangle}.
%
%{\em only} rely on
%the finite-graph triangle condition \eqref{finitetriangle}.
%In Section \ref{sec-off}, we start by discussing disjoint occurrence and
%the BK-Reimer inequality, which plays
%a crucial role throughout our proof.

\subsection{The ``off'' method and BK-Reimer inequality}
\label{sec-off} We will frequently handle the events $\partial B(r) \neq \emptyset$ and
$x \stackrel{=t}{\lrfill} y$. These events are non monotone with respect to adding edges, indeed, adding an edge may shorten a shortest path and prevent the events from holding. This non-monotonicity
is a technical difficulty which unfortunately manifests itself in many of
the arguments in this paper. Our main tools to deal with this problem are the
BK-Reimer inequality \cite{BerKes85, Reim00} and the notion of events
occurring ``off'' a set of vertices. For the BK-Reimer inequality we
use the formulation in \cite{ReimerRef}.

For a subset of vertices $A$, we say that an event $\M$ occurs {\em off $A$}, intuitively, if it occurs in $G_p \setminus A$. Formally, for a percolation configuration
$\omega$, we write $\omega_A$ for the configuration obtained from $\omega$ by
turning all the edges touching $A$ to closed. The event ``$\M$ occurs off $A$''
is defined to be $\{\omega \colon \omega_A \in \M\}$. We also frequently write $\poffa$
to denote the measure $\poffa(\M) = \prob_p(\M \off A)$.
Equivalently, $\poffa$ can be thought of as a percolation measure in which all edges
touching $A$ are closed with probability $1$ and the rest are distributed independently as before.
We often drop $p$ from the notation when it is clear what $p$ is. This framework also allows us to address the case when $A=A(\omega)$ is a \emph{random} set measurable
with respect to $G_p$ (the most prominent example is $A= B_0(r)$). In this case,
the event $\{\M \hbox{ {\rm occurs off} } A(\omega)\}$ is defined to be
   $$
    \{\M \hbox{ {\rm occurs off} } A(\omega)\} = \{\omega \colon \omega_{A(\omega)} \in \M \} \, .
    $$

%Throughout the paper we will perform the following recurring abuse of notation. For a subset of vertices $A \subset V$ we write %$B_x(r)=A$ both for the event that the set of vertices $B_x(r)$ is $A$ and both for
%In many of our arguments we will want to condition on the event $B_x(r)=A$ for a subset of vertices $A \subset V$. We will abuse

Let us review an example occurring frequently in our arguments in which $\M$ is an arbitrary event and $A=B_x(s)$. In this case,
    \be
    \label{offdescription}
    \prob(\M \off B_x(s)) =
    \sum_{A} \prob(B_x(s)=A) \prob(\M \off A) \, ,
    \ee
where we have used the fact that
    $$
    \prob(\M \off B_x(s) \mid B_x(s)=A) = \prob(\M \off A) \, ,
    $$
since the events do not depend on edges
touching $A$ in both sides of the equation, and the marginal of the two
distributions on the edges not touching $A$ is the same product measure.
%%% Ugh, this wasn't true at all.
% We emphasize that the equality (\ref{offdescription}) does
%{\em not} hold for any event $\M$. A good example to test one's understanding of %this subtlety
%is the event $\M=\{ 0 \stackrel{=r}{\lrfill} y \and B_0(r) \cap B_x(s) = %\emptyset\}$.
In terms of this notation, for a subset of vertices $A$, we define
    $$
    B_x^G(r;A) = \{ y \colon d_{G_p}(x,y) \leq r \off A\} \, , \quad
    \partial B_x^G(r;A) = \{ y \colon d_{G_p}(x,y)=r \off A\}
    $$
to be the intrinsic ball off $A$ and its boundary. We finally say that $\M$
occurs {\em only on $A$} if $\M$ occurs but $\M$ off $A$ does not occur.
We frequently rely on the following inclusion:

\begin{claim}
\label{offmethod}
For any event $\M$ and any subset of vertices $A \subset V$,
    $$
    \M \setminus \{\M \onlyon A\} \subset \{\M \off A\} \, .
    $$
\end{claim}

\begin{proof} By definition of ``$\M \onlyon A$'' the event on the left hand side equals
    $$
    \M \cap \big \{\M^c \cup \{\M \off A\} \big \} \, .
    $$
From this, it is easy to see that this event implies $\M$ off $A$.
%In fact it implies $\M \cap \{\M \off A\}$.
\end{proof}
\noindent {\bf Remark.} Equality in Claim \ref{offmethod} does not hold
(unless the right hand side is replaced by $\M \cap \{\M \off A\}$). This
can easily be seen by taking a non-monotone event, say
$\partial B_x(r) \neq \emptyset$.\\

The following lemmas are inequalities valid for any graph $G$ and any $p$.
% These are non-monotone versions of well-known
%(and rather trivial) estimates with the additional restriction on the length
%of connections. Their proof requires a delicate choice of \emph{witnesses} in order
%to be able to apply the BK-Reimer inequality.

%%%%%% Asaf July 22: Don't think I need this lemma any more. Keeping it just in case.
%\begin{lemma}\label{connectonlyon} For any subset of vertices $A$ we have
%    $$
%    \prob(x \stackrel{=r}{\lrfill} y \onlyon A)
%    \leq \sum_{t \leq r, z\in A} \prob(x \stackrel{=t}{\lrfill} z)
%    \prob(z \stackrel{=r-t}{\lrfill} y \off B_x(t)) \, .
%    $$
%\end{lemma}
%\begin{proof} If $x \stackrel{=r}{\lrfill} y$ only on $A$, then there exists
%a shortest $p$-open path $\gamma$ of length $r$ between $x$ and $y$ and $t \leq r$
%and $z\in A$ such that $\gamma(t) = z$. This implies that the event
%    $$
%    \{ x \stackrel{=t}{\lrfill} z\} \circ \{z \stackrel{=r-t}{\lrfill} y \off B_x(t)\}
%    $$
%occurs. Indeed,  the witnesses for the first event is the part of $\gamma$ between $x$ and $z$
%together with all the closed edges touching the ball $B_x(t-1)$ and witness
%for the second event is the part of $\gamma$ between $z$ and $a$ together
%with the closed edges touching the ball $B_y(r-t-1)$ that do not
%touch the ball $B_x(t)$. BK-Reimer inequality concludes the proof.
%\end{proof}

\begin{lemma}[Disjoint survival]
\label{offplusreimer} For any graph $G$, $p\in [0,1]$, vertices $x,y,z$ and integers $r,s,$
    $$
    \prob( \partial B_x(r)\neq \emptyset, \partial B_y(s) \neq \emptyset, B_x(r) \cap B_y(s) = \emptyset )
    \leq \prob(\partial B_x(r) \neq \emptyset) \max_{A \subset V} \poffa(\partial B_y(s) \neq \emptyset) \, .
    $$
%%%%%% Asaf July 22: Don't think I'm using this anymore too.
%and
%    $$
%    \prob( x \stackrel{=r}{\lrfill} z, \partial B_y(s) \neq \emptyset, B_x(r) \cap B_y(s) = \emptyset )
%    \leq \prob(x \stackrel{=r}{\lrfill} z) \prob(\partial B_y(s) \neq \emptyset \off B_x(r)) \, ,
%    $$
\end{lemma}
\begin{proof} We condition on $B_x(r)=A$ such that $A$ satisfies $\partial B_x(r) \neq \emptyset$ and $\prob(B_x(r)=A)>0$. The left hand side then equals
    $$
    \sum_{A \colon \partial B_x(r) \neq \emptyset} \prob(B_x(r)=A) \prob(\partial B_y(s) \neq \emptyset \off A) \, ,$$
as in \eqref{offdescription}. The lemma now follows.
%
%We claim that if $\partial B_x(r)\neq \emptyset$, $\partial B_y(s) \neq \emptyset$
%and $B_x(r) \cap B_y(s) = \emptyset$, then the event
%    $$
%    \{ \partial B_x(r) \neq \emptyset \} \circ \{ \partial B_y(s) \neq \emptyset \off B_x(r) \}
%    $$
%occurs. Indeed, the witness to the first event are all the edges, open and closed, touching
%the ball $B_x(r-1)$ and the witness to the second event are all the open edges touching the
%ball $B_y(s-1)$ together with all the closed edges touching $B_y(s-1)$ which do {\em not}
%touch $B_x(r-1)$. These are disjoint witnesses and hence BK-Reimer inequality concludes
%the proof.
%%The second assertion is proved using the exact same witnesses.
\end{proof}

\begin{lemma}[Tree-graph bound revisited]
\label{goodbkreimer}
For any graph $G$, $p\in[0,1]$, vertices $u,v$ and integers $r,\ell>0,$
    $$
    \prob(0 \stackrel{=r}{\lrfill} u \and 0 \stackrel{\ell}{\lrfill} v)
    \leq \sum_{z} \sum_{t=0}^{r} \prob(0 \stackrel{=t}{\lrfill} z)
     \prob(z \stackrel{\ell}{\lrfill} v) \max_{A \subset V} \poffa(z \stackrel{=r-t}{\lrfill} u)\, .
    $$
\end{lemma}

\begin{proof} We claim that if $0 \stackrel{=r}{\lrfill} u$ and $0 \stackrel{\ell}{\lrfill} v$, then
there exist $z\in V$ and $t \leq r$ such that the two following events occur disjointly
\begin{enumerate}
\item[(a)] There exists a shortest open path $\eta$ of length $r$ between $0$ and $u$ such that $\eta(t)=z$, and
\item[(b)] There exist an open path between $z$ and $\ell$ of length at most $\ell$.
\end{enumerate}
Indeed, if the event occurs, let $\eta$ be the lexicographical first shortest path of length $r$ between $0$ and $u$ and let $\gamma$ be an open path of length at most $\ell$ between $0$ and $v$. We take $z$ to be the last vertex on $\gamma$ which belongs to $\eta$ and put $t$ such that $\eta(t)=z$. The witness for the first event is the set of open edges of $\eta$ together with all the closed edges in $G_p$ (the closed edges determine that $\eta$ is a shortest path) and the second witness is the edges of $\gamma$ from $z$ to $v$. These are disjoint witnesses so we may use the BK-Reimer inequality and bound
    $$ \prob(0 \stackrel{=r}{\lrfill} u \and 0 \stackrel{\ell}{\lrfill} v) \leq \sum_{z \in V, t \leq r} \prob((a)) \prob(z \stackrel{\ell}{\lrfill} v) \, .$$
To bound $\prob((a))$ we condition on $B_0(t)=A$ such that $A$ satisfies $0 \stackrel{=t}{\lrfill} z$, so
    $$ \prob((a)) = \sum_{A : 0 \stackrel{=t}{\lrfill} z} \prob(B_0(t)=A) \prob(z \stackrel{=r-t}{\lrfill} u \off A) \, ,$$
and the lemma follows.
%%%%% OLD argument just ain't true!
%the event
%    $$
%    \{0 \stackrel{=t}{\lrfill} z\}\circ \{z \stackrel{=r-t}{\lrfill} u \off B_0(t)\} \circ \{z \stackrel{\ell}{\lrfill} v\} \, ,
%    $$
%occurs. To see this, let $\eta$ be an arbitrary (say, lexicographically first) shortest
%open path of length $r$ between $0$ and $u$ and let $\gamma$ be an arbitrary open path
%of length at most $\ell$ between $0$ and $v$. Write $z$ for the last vertex on $\gamma$
%which belongs to $\eta$ and put $t$ for the index $\eta(t)=z$. Then the witness for the
%first event are the open edges of $\eta[0,t]$  together with all the closed edges
%touching the ball $B_0(t)$ --- these edges dictate that $\eta[0,t]$ is the shortest
%open path between $0$ and $z$. The witness for the second event is the open edges of
%$\eta[t,r]$ together with all the closed edges touching the ball $B_u(t-r)$ but not
%those touching $B_0(t)$ --- these edges dictate that the distance between $z$ and $u$
%off $B_0(t)$ is $t-r$. (This is where we cannot do without the ``off $B_0(t)$'' in the
%statement, because the set of open edges of $\eta[t,r]$ together
%with the closed edges touching $B_u(t-r)$ may intersect the first witness.) The
%witness for the third event is just the open edges of the part of $\gamma$ starting
%at $z$ and ending at $v$. These witnesses are disjoint, and so the BK-Reimer inequality
%finishes the proof.
\end{proof}

\subsection{Survival probabilities}
\label{sec-survprob}
In this section, we prove Lemma \ref{lem-conc-Z-I} and a few other useful estimates
of a similar nature. In the rest of this section we only rely on the finite triangle
condition (\ref{finitetriangle}), Theorem \ref{KozmaNachmias}
and Theorem \ref{ub-cluster-tail} (both follow from the triangle condition).

\begin{lemma}[Relating connection probabilities for different $p's$]
\label{corrlength1}
Let $p_1,p_2\in [0,1]$ satisfy $p_1\leq p_2$ and let $r>0$ be integer.
The following bounds hold for any graph $G$ and vertex $v$:
\begin{itemize}
\item[(1)] $\displaystyle \prob_{p_2} ( \partial B_v(r) \neq \emptyset )
\leq \Big(\frac{p_2}{p_1}\Big)^{r} \prob_{p_1} (\partial B_v(r) \neq \emptyset)$,
\item[(2)] $\displaystyle \E_{p_2} |\partial B_v(r)| \leq \Big(\frac{p_2}{p_1}\Big)^{r} \E_{p_1}| \partial B_v(r)|$.
\end{itemize}
\end{lemma}
\begin{proof} We recall the standard simultaneous coupling between percolation measure at different $p$'s
discussed in Section \ref{sec-notation}. Order all the paths in $G$ of length
$r$ starting at $v$ lexicographically. Write $\A$ for the event that $\partial B_v(r) \neq \emptyset$ in
$G_{p_2}$ and that the lexicographical first  $p_2$-open shortest path of length $r$ starting at $v$ is
in fact $p_1$-open. We claim that
    \eqn{
    \label{PA-prob}
    \prob(\A) = \Big(\frac{p_1}{p_2}\Big)^{r}\prob(\partial B_v(r) \neq \emptyset \hbox{ {\rm in} } G_{p_2}) \, .
    }
Indeed, conditioned on the edges of $G_{p_2}$, the value $U(e)$ of each edge in $G_{p_2}$ is
distributed uniformly on the interval $[0,p_2]$. Hence, the probability of the first shortest path
being $p_1$-open in this conditioning is precisely $(p_1/p_2)^r$, which proves \eqref{PA-prob}.
To see (1), note that if the first $p_2$-open shortest path is $p_1$-open, then it is a shortest path of
length $r$ in $G_{p_1}$, so that $\A$ implies that $\partial B_v(r) \neq \emptyset$ in $G_{p_1}$ whence
    $$
    \prob_{p_1} ( \partial B_v(r) \neq \emptyset )
    \geq \Big(\frac{p_1}{p_2}\Big)^{r}\prob_{p_2} (\partial B_v(r) \neq \emptyset) \, .
    $$
The proof of (2) is similar and we omit the details.
\end{proof}

\begin{corollary}[Correlation length is $1/\eps$]
\label{corrlength} Let $G$ be a transitive finite graph for which
(\ref{finitetriangle}) holds and put $p=p_c(1+\eps)$. The following bounds
hold for any subset of vertices $A$ and any vertex $v$ and any integer $r$:
\begin{itemize}
\item[(1)] $\displaystyle \poffa \big ( \partial B_v(\eps^{-1}) \neq \emptyset \big ) = O(\eps) \, ,$ and
\item[(2)] $\displaystyle \E |B_v(r;A)| = O(r(1+\eps)^r)$.
%\item[(3)] $\E |B_{p}(v,r_0)| \leq C r_0$ for any $r_0 \leq \eps^{-1}$.
\end{itemize}
\end{corollary}
\begin{proof} The result is immediate by combining Lemma \ref{corrlength1}
and Theorem \ref{KozmaNachmias}.
\end{proof}

\noindent {\bf Remark.} In Section \ref{sec-exp-vol-est} we will show a sharp estimate replacing (2) in the above corollary.

\begin{lemma}[Supercritical survival probability]
\label{inttail}
Let $G$ be a transitive finite graph for which (\ref{finitetriangle})
holds and put $p=p_c(1+\eps)$. Then, for any $M\to \infty$ and any subset of vertices $A$,
    $$
    \poffa (\partial B(M\eps^{-1}) \neq \emptyset) \leq (2+o(1)) \eps \, ,
%    \prob(|\C(0)| \geq \sqrt{M} \eps^{-2})+O(\eps M^{-1/2})=(2+o(1))\eps \, ,
    $$
and, for any $M \leq \log\log(\eps^3 V)$ such that $M \to \infty$,
    $$
    \prob(\partial B(M\eps^{-1}) \neq \emptyset) \geq
%    \prob(|\C(0)| \geq k_0)+O(\vep (\eps^3 V)^{-\alpha}\log(\eps^3 V))) \eps
    (2-o(1))\eps \, .
    $$
%where $k_0=\eps^{-2} (\eps^3 V)^{\alpha}$.
\end{lemma}

%% Asaf: July 23, I don't think we really rely on this error.
%In Lemma \ref{inttail}, we keep the error terms involved explicit, as
%we shall later rely on this.
\begin{proof} To prove the upper bound we write
    \begin{eqnarray*}
    \poffa (\partial B(M\eps^{-1}) \neq \emptyset)
    &=& \poffa(\partial B(M\eps^{-1}) \neq \emptyset \, , \,
    |\C(0)| > \sqrt{M} \eps^{-2}) \\ &&\quad + \poffa(\partial B(M\eps^{-1}) \neq \emptyset \, , \,
    |\C(0)|\leq \sqrt{M} \eps^{-2}) \, .
    \end{eqnarray*}
The first term on the right hand side is at most $(2+o(1))\eps$ by
Theorem \ref{clustertail} --- note that we used the fact that the event $|\C(0)|\geq k$
is monotone so $\poffa(|\C(0)|\geq k) \leq \prob(|\C(0)|\geq k)$. It remains to
show that the second term is $o(\eps)$. Indeed, if this event occurs, then
there exists a radius $j \in [M\eps^{-1}/3, 2M\eps^{-1}/3]$ such that
    $$
    0 <\partial B(j;A) \leq 3M^{-1/2}\eps^{-1} \and \partial B(M\eps^{-1}; A)\neq \emptyset \, .
    $$
Let $J$ be the first level satisfying this. By Corollary \ref{corrlength} and the union bound
    $$
    \poffa \big (\exists\,\, y \in \partial B(J; A)
    \with \partial B_y(M\eps^{-1}/3; A) \neq \emptyset \off B(J; A) \, \mid \,
    B(J; A) \big )
    \leq C \eps |\partial B(J; A)| = O(M^{-1/2}) \, .
    $$
Corollary \ref{corrlength} also shows that $\poffa( \partial B(M\eps^{-1}/3)\neq\varnothing) = O(\eps)$,
so putting this together gives
    \be
    \label{inttail.willuselater}
    \poffa(\partial B(M\eps^{-1}) \neq \emptyset \, , \,
    |\C(x)|\leq \sqrt{M} \eps^{-2}) = O(M^{-1/2}\eps) \, ,
    \ee
concluding the proof of the upper bound. For the lower bound, take $k_0 = \eps^{-2}(\eps^3 V)^\alpha$ for some
fixed $\alpha\in (0,1/3)$. We have
    \begin{eqnarray*}
    \prob(\partial B(M\eps^{-1}) \neq \emptyset)
    &\geq& \prob(\partial B(M\eps^{-1}) \neq \emptyset \and |\C(0)|\geq k_0) \\
    &=& \prob(|\C(0)| \geq k_0) - \prob(|\C(0)| \geq k_0 \and \partial B(M\eps^{-1}) = \emptyset) \, ,
    \end{eqnarray*}
so by Theorem \ref{clustertail} it suffices to bound the last term on the right hand side from above.
Indeed, by Markov's inequality and Corollary \ref{corrlength},
    \eqan{
    \label{survive.toshow}
    \prob(|\C(0)| \geq k_0 \and \partial B(M\eps^{-1}) = \emptyset)
    &\leq \prob(|B(M \eps^{-1})| \geq k_0)\\
    &\leq {C M \e^{M} \eps^{-1} \over k_0} = O\Big(\eps (\eps^3 V)^{-\alpha}\log(\eps^3 V)\Big)
    = o(\eps) \, ,
    \nn
    }
since $M\leq \log\log(\eps^3V)$.
% (this works as long as $M \leq c\log(\eps^3 V)$ for any $c<1/3$,
%but this is not important).
\end{proof}

We proceed with preparations towards the proof of Lemma \ref{lem-conc-Z-I}.
For an integer $r>0$, we write $N_r$ for the random variable
    $$
    N_r = \big | \big \{ x \colon \partial B_x(r) \neq \emptyset \big \} \big | \, .
    $$
We think of $1/\eps$ as the correlation length, see \cite{Grim99}. In other
words, when $r \gg 1/\eps$, then the vertices contributing to
$N_r$ should be those in the giant component.

\begin{lemma}[$N_r$ is concentrated]
\label{manyvertssurvived}
Let $G_m$ be a sequence of transitive finite graph with degree $\m$ for which (\ref{finitetriangle}) holds and $\m \to \infty$.
Put $p=p_c(1+\eps)$, then for any $r\gg 1/\vep$ satisfying
$r \leq \eps^{-1} \log\log(\eps^3 V)$,
    $$
    \frac{N_r}{2\eps V}\convp 1.
    $$
\end{lemma}

\begin{proof}
We use a second moment method on $N_r$. Lemma \ref{inttail} and our assumption on $r$ shows that
    $$
    \E N_r = (2+o(1)) \eps V \, .$$
The second moment is
    $$
    \E N_r^2 = \sum_{x,y} \prob(\partial B_x(r) \neq \emptyset \and \partial B_y(r) \neq \emptyset) \, .
    $$
We have
    \begin{eqnarray*}
    \prob(\partial B_x(r) \neq \emptyset \and \partial B_y(r) \neq \emptyset)
    &\leq& \prob(\partial B_x(r) \neq \emptyset \and \partial B_y(r) \neq \emptyset , B_x(r) \cap B_y(r) = \emptyset) \\
    &&\quad +\prob(x \stackrel{2r}{\lrfill} y) \, .
    \end{eqnarray*}
We sum the first term on the right hand side using Lemmas \ref{offplusreimer} and \ref{inttail}
and the second term using Corollary \ref{corrlength}. We get that
    $$
    \E N_r^2 \leq (4+o(1))\eps^2 V^2 + O\big(Vr(1+\eps)^{2r}\big)
    = (1+o(1)) [\E N_r]^2 \, ,
    $$
since $Vr(1+\eps)^{2r} = o(\eps^2 V^2)$ by our assumption on $r$
and since $\eps^3 V \to \infty$. The assertion of the lemma now follows by Chebychev's inequality.
\end{proof}

\noindent {\bf Proof of Lemma \ref{lem-conc-Z-I}.} Take $M=\Mm$ and $r$ as in (\ref{Mchoice}) and write
    \begin{eqnarray*}
    Z_{\sss \geq k_0}&=N_r+\big |\big \{x \colon \partial B_x(r)=\emptyset,
    |\C(x)|\geq k_0\big \} \big | -\big |\big \{ x \colon \partial B_x(r)\neq \emptyset,
    |\C(x)|< k_0\big \} \big | \, .
    \end{eqnarray*}
By Lemma \ref{manyvertssurvived}, $N_r/(2\eps V)\convp 1$, so it suffices
to show that the expectation of both remaining terms is $o(\eps V)$.
The expectation of the first term is
    \be
    V\prob(x \colon \partial B_x(r)=\emptyset,
    |\C(v)|\geq k_0)=o(\eps V),
    \ee
by \eqref{survive.toshow}. The expectation of the second term now must be $o(\eps V)$
since both $N_r$ and $Z_{\sss \geq k_0}$ have mean $(2+o(1))\eps V$
by Theorem \ref{clustertail} and Lemma \ref{inttail}. This shows that $Z_{\sss \geq k_0}/(2\eps V)\convp 1$
as stated in Lemma \ref{lem-conc-Z-I}.

To prove the upper bound on $\expec|\C(0)|$ we write
    $$
    \E |\C(0)| = \sum_{y} \prob(0 \lr y) = \sum_{y} \prob ( 0 \stackrel{[0,2r)}{\lrfill} y)
    + \sum_y \prob(0 \stackrel{[2r,\infty)}{\lrfill} y) \, .
    $$
By Corollary \ref{corrlength},
    \be \label{willuselater}
    \sum_{y} \prob ( 0 \stackrel{[0,2r)}{\lrfill} y) = \E |B(2r)| \leq C \eps^{-1} \log^3 (\eps^3 V) = o(\eps^2 V) \, ,
    \ee
since $\eps^3 V \gg 1$. If $0 \stackrel{[2r,\infty)}{\lrfill} y$, then the event
    $$
    \{\partial B_0(r)\neq \emptyset, \partial B_y(r) \neq \emptyset, B_0(r) \cap B_y(r) = \emptyset\} \, ,
    $$
occurs. Hence Lemmas \ref{offplusreimer} and \ref{inttail} give that $\prob(0 \stackrel{[2r,\infty)}{\lrfill} y) \leq (4+o(1)) \eps^2$ and summing this over $y$ gives the required upper bound on $\E |\C(0)|$. \qed

\subsection{Disjoint survival probabilities}
\label{sec-disj-surv-prob}
%We cannot expect that $\prob(\A(x,y,r,r))$ is asymptotic to $4\eps^2$ for all $x,y$. Indeed,
%if $x,y$ are neighbors this quantity is of order $m^{-1} \eps$. However, the $4\eps^2$
%asymptotics occurs for {\em most} pairs $x,y$.
In this section we show that for most pairs $x,y$ the event $\A(x,y,r,r)$ occurs with probability asymptotic to $4\eps^2$.
The point is that $r$ is chosen in such that $r \gg \eps^{-1}$, where $\eps^{-1}$ is the correlation length,
but $r \ll \eps^{-1} \log(\eps^3 V)$ which is the diameter of $\C_1$.

\begin{lemma}[Number of pairs surviving disjointly]
\label{alotofpairs}
Let $G_m$ be a sequence of transitive finite graph with degree $\m$ for which (\ref{finitetriangle}) holds and $\m \to \infty$.
Put $p=p_c(1+\eps)$, then for any $r\gg \vep^{-1}$ satisfying $r \leq \eps^{-1} \log\log(\eps^3 V)$
    $$
    \frac{\big | \big \{ x,y \, \colon \, \A(x,y,r,r)\mbox{ {\rm occurs}}\big \}  \big |}
    {(2\eps V)^2}\convp 1\, .
    $$
\end{lemma}

\begin{proof}
Define
    $$
    N^{\sss (2)}_r = \big | \big \{ x,y \colon \A(x,y,r,r) \hbox{ {\rm occurs}} \big \} \big | \, .
    $$
Then,
    $$
    N_r^2 - |\{ x,y \colon x \stackrel{2r}{\lrfill} y\}| \leq N^{\sss (2)}_r \leq N_r^2 \, ,
    $$
and $\E |\{ x,y \colon x \stackrel{2r}{\lrfill} y\}| = o(\eps^2 V^2)$ as we did in (\ref{willuselater}). The result now follows by Markov's inequality and Lemma \ref{manyvertssurvived}.
\end{proof}

\begin{lemma}[Most pairs have almost independent disjoint survival probabilities]
\label{xysurvive}
Let $G_m$ be a sequence of transitive finite graph with degree $\m$ for which (\ref{finitetriangle}) holds and $\m \to \infty$. Put $p=p_c(1+\eps)$. Then, for any $j_x,j_y \leq \eps^{-1} \log \log \eps^3 V$ such that $j_x,j_y \gg \eps^{-1}$, there exist at least $(1-o(1))V^2$
pairs of vertices $x,y$ such that
    $$
    \prob(\A(x,y,j_x,j_y)) = (1+o(1))4\eps^2 \, .
    $$
\end{lemma}
\begin{proof} The upper bound $\prob( \A(x,y,j_x,j_x) ) \leq (1+o(1))4\eps^2$
follows immediately from Lemmas \ref{offplusreimer} and \ref{inttail} and is valid for all pairs $x,y$.
We turn to showing the corresponding lower bound. First note that Cauchy-Schwartz's inequality
gives $\E [N_{2}^2] \geq (\E N_{2})^2$ which can be written as
    $$
    \sum_{x,y} \prob( \partial B_x(r) \neq \emptyset \and \partial B_y(r) \neq \emptyset)
    \geq V^2 \prob(\partial B(r)\neq \varnothing)^2 \, .
    $$
We take $r= \eps^{-1} \log\log \eps^3 V$. Since $\prob(\partial B_x(j) \neq \emptyset)$ is decreasing in $j$, by Lemma \ref{inttail} and our assumption on $j_x$ and $j_y$ we get
    $$
    \sum_{x,y} \prob(\partial B_x(j_x) \neq \emptyset \and \partial B_y(j_y) \neq \emptyset)
    \geq (4-o(1)) V^2 \eps^2 \, .
    $$
Secondly, Corollary \ref{corrlength} implies that
    $$
    \sum_{x,y} \prob(x \stackrel{2r}{\lrfill} y)
    =V \expec[|B(2r)|]\leq CV r (1+\eps)^{2r} = o(\eps^2 V^2) \, ,
    $$
by our choice of $r$ and since $\eps^3 V \to \infty$. Since
    $$
    \A(x,y,j_x,j_y) \subseteq \{ \partial B_x(j_x) \neq \emptyset ,
    \partial B_y(j_y) \neq \emptyset\} \setminus \{x \stackrel{2r}{\lrfill} y\}\, ,
    $$
we deduce that
    $$
    \sum_{x,y} \prob(\A(x,y,j_x,j_y)) \geq (4-o(1)) \eps^2 V^2 \, ,
    $$
and since the upper bound was valid for all $x,y$, the lemma follows.
\end{proof}

\begin{lemma}
\label{morefitnesstwo}
Let $G_m$ be a sequence of transitive finite graph with degree $\m$ for which (\ref{finitetriangle}) holds and $\m \to \infty$. Put $p=p_c(1+\eps)$. Then for any $r \leq \eps^{-1} \log \log \eps^3 V$ such that $r \gg \eps^{-1}$ we have
    $$
    {\Big | \Big \{ x,y \colon \A(x,y,r,r) \and |\C(x)| \leq (\eps^3 V)^{1/4} \eps^{-2} \Big \}
    \Big | \over (\eps V)^2} \convp 0  \, .
    $$
\end{lemma}
\begin{proof} By Lemma \ref{manyvertssurvived}, the assertion follows from the statement that
    $$
    \big|\big\{ x \colon \partial B_x(r) \neq \emptyset \and |\C(x)| \leq (\eps^3 V)^{1/4} \eps^{-2}\big\}\big|/(\eps V)
    \convp 0 \, .
    $$
To show this, note that
    $$
    \prob\big(\partial B_x(r) = \emptyset \and |\C(x)| \geq (\eps^3 V)^{1/4} \eps^{-2}\big)
    \leq \prob\big(|B_x(r)| \geq (\eps^3 V)^{1/4} \eps^{-2}\big)
    \leq {Cr(1+\eps)^{r} \over (\eps^3 V)^{1/4} \eps^{-2}} = o(\eps) \, .
    $$
Hence, by Theorem \ref{clustertail},
    $$
    \prob(\partial B_x(r) \neq \emptyset \and |\C(x)| \geq (\eps^3 V)^{1/4} \eps^{-2}) \geq (2-o(1)) \eps \, .
    $$
Together with Lemma \ref{inttail}, this yields that
    $$
    \prob(\partial B_x(r) \neq \emptyset \and |\C(x)| \leq (\eps^3 V)^{1/4} \eps^{-2}) = o(\eps) \, ,
    $$
concluding our proof.
\end{proof}

\subsection{Using the non-backtracking random walk}
\label{sec-NBW}
In the rest of this section we provide several basic percolation estimates which we
use throughout the paper. These include bounds on long and short connection probabilities
and bounds on various triangle and square diagrams. It is here that  we make crucial use
of the geometry of the graph and the behavior of the random walk on it, namely, the assumptions of Theorem \ref{mainthmgeneral}. We frequently use \emph{non-backtracking} random walk estimates. This walk is a simple random walk on a graph
that is not allowed to traverse back on an edge it has just walked on. Let us first define
it formally. % and then study its asymptotics on the hypercube.

The non-backtracking random walk on an undirected graph $G=(V,E)$, starting from a
vertex $x \in V$, is a Markov chain $\{X_t\}$ with transition
matrix $\prob ^x$ on the state space of {\em directed} edges
    $$
    \overrightarrow{E} = \Big \{ (x,y) \colon \{x,y\} \in E \Big \} \, .
    $$
If $X_t = (x,y)$, then we write $X_t^{\sss (1)} = x$ and $X_t^{\sss (2)} = y$.
Also, for notational convenience, we write
    $$
    \prob _{(x,w)}(\cdot)=\prob^x\big( \cdot \mid X_0 = (x,w)\big),
    \quad \and \quad
    \p^t(x,y)=\prob^x \big(X_t^{\sss (2)} = y\big).
    $$
The non-backtracking walk starting from
a vertex $x$ has initial state given by
    $$
    \prob ^x(X_0 = (x,y)) = {\bf 1}_{\{(x,y)\in \overrightarrow{E}\}} {1 \over {\rm deg}(x)} \, ,
    $$
and transition probabilities given by
    $$
    \prob^x_{(u,v)}(X_1 = (v,w))
    = {\bf 1}_{\{(v,w)\in \overrightarrow{E} \, , w \neq
    v \}} {1 \over {\rm deg}(v) - 1} \, ,
    $$
where we write deg$(x)$ for the degree of $x$ in $G$. The following lemma will be useful.
\begin{lemma}
\label{lem-NBW-ext}
Let $G$ be a regular graph of degree $\m$. Then,
for $t\geq 0$,
    $$
    \sum_{u}\p^1(0,u)\p^t(u,z)
    = \p^{t+1}(0,z)+\frac{1}{\m-1}\p^{t-1}(0,z),
    $$
with the convention that $\p^{-1}(0,z)=0$.
\end{lemma}

\proof The claim follows by conditioning on whether the first
step of the NBW counted in $\p^t(u,z)$ is equal to $(u,0)$ or not.
We omit further details.
%This first step is equal to
%$(u,0)$ with probability $1/\m$, after which the walk
%needs to go from $0$ to $z$ in $t-1$ steps, and thus
%gives a contribution $\frac{1}{\m-1}\p^{t-1}(0,z)$
%When it is different, the path from 0 to u,
%followed by the path of length $t$ from $u$ to $z$
%forms a NBW of $t+1$ steps between 0 and $u$.
\qed

\subsection{Uniform upper bounds on long connection probabilities}
\label{sec-unif-bds-conn}
In this section we show that long percolation paths are asymptotically equally likely to end at any vertex
in $G$.
\begin{lemma}
\label{moreunifconnbd} Let $G$ be a graph satisfying the assumptions in
Theorem \ref{mainthmgeneral} and consider percolation on it with $p\leq p_c(1+\eps)$.
Then, for any integer $t \geq m_0$ and any vertex $x$,
    $$
    \prob_{p}(0 \stackrel{=t}{\lrfill} x) + \prob_{p}(0 \stackrel{=t+1}{\lrfill} x)
    \leq {2 + O(\alpham+\eps\mnot) \over V} \E |\partial B(t-\mnot)| \, .
    $$
\end{lemma}

\begin{proof} If $0 \stackrel{=t}{\lrfill} x$, then there exists a vertex $v$ and simple path
$\omega$ of length $\mnot$ from $x$ to $v$ such that the event
    $$
    \{ \omega \hbox{ {\rm is open}}\} \circ \{ v \stackrel{=t-\mnot}{\lrfill} x\} \, ,
    $$
occurs. Indeed, consider a shortest path $\eta$ of length $t$ between $0$ to $x$. Take
$v= \eta(\mnot)$ and $\omega = \eta[1,\mnot]$. Now, the first witness is the path $\omega$ and
the witness for $\{v \stackrel{=t-\mnot}{\lrfill} x\}$ is the path $\eta[\mnot,t]$ together
with all the closed edges in $G_p$ (which determine that $\eta[\mnot,t]$ is a shortest path). These are disjoint witnesses.
If $0 \stackrel{=t+1}{\lrfill} x$ occurs, then we get the same conclusion with $\omega$ of
length $\mnot+1$. We now apply the BK-Reimer inequality and the fact that the probability
that $\omega$ is open is precisely $p^{|\omega|}$. This yields
    $$
    \prob_{p}(0 \stackrel{=t}{\lrfill} x) \leq p^{\mnot} \sum_v
    \sum_{\substack{\omega \colon |\omega|=\mnot \\ \omega[\mnot]=v}} \prob_p(v \stackrel{=t-\mnot}{\lrfill} x) \, ,
    $$
and
    $$
    \prob_{p}(0 \stackrel{={t+1}}{\lrfill} x) \leq p^{\mnot+1} \sum_{v} \sum_{\substack{\omega \colon |\omega|=\mnot+1 \\ \omega[\mnot+1]=v}} \prob_p(v \stackrel{=t-\mnot}{\lrfill} x) \, .$$
We now bound these above by relaxing the requirement that $\omega$ is simple and
only require that it is non-backtracking. Since $\mnot=\Tm(\alpham)$, we get by definition that
    $$
    {\big |\{ \omega \colon |\omega|=\mnot, \omega[\mnot]=v \} \big |
    \over \m(\m-1)^{\mnot-1}} + { \big  |\{ \omega \colon |\omega|=\mnot+1, \omega[\mnot+1]=v \} \big |
    \over \m(\m-1)^{\mnot}} =\p^{\mnot}(0,v)+\p^{\mnot+1}(0,v)\leq {{2 + 2\alpham} \over V} \, ,
    $$
where we enumerated only non-backtracking paths in the above. Using this,
condition (2) and summing over $v$ gives
    \begin{eqnarray*}
    \prob_{p}(0 \stackrel{=t}{\lrfill} x) +
    \prob_{p}(0 \stackrel{={t+1}}{\lrfill} x)
    &\leq & {2+O(\alpham) \over V} [p(m-1)]^{\mnot}
    \E |\partial B(t-\mnot)| \\ &\leq& {2 + O(\alpham+\eps\mnot) \over V}
    \E |\partial B(t-\mnot)| \, ,
    \end{eqnarray*}
concluding our proof.
%
%summing over {\em all}
%non-backtracking paths $\omega$ of length $\mnot$ (note that these are not necessarily simple paths).
%The number of such paths is $m(m-1)^{\mnot -1}$ and by Lemma \ref{lem-srwmix},
%the number of such path having $\omega [\mnot]=v$ is at most
%    $$
%    (2+O(\e^{-\eta \m}))V^{-1} m(m-1)^{\mnot -1}.
%    $$
%We get that
%    \begin{eqnarray*}
%    \prob_p \big ( 0 \stackrel{=t}{\lrfill} x \big ) &\leq& {m \over m-1}
%    (p (m-1))^{\mnot} (2+O(\e^{-\eta \m})) V^{-1}
%    \sum_{v} \prob_p ( v \stackrel{=t-\mnot}{\lrfill} x ) \\ &\leq&
%    {(2 + O(\mnot(1/\m^2+\eps)) \E|\partial B(t-\mnot)| \over V}  \, ,
%    \end{eqnarray*}
%where in the last inequality we used our bound on $p$ in the form
%    \eqn{
%    \label{mixing-time-cond-used}
%    (p (m-1))^{\mnot}\leq (1+\eps)^{\mnot}((m-1)p_c)^{\mnot}=1+O(\mnot (1/\m^2+\eps')).
%    }
\end{proof}

\begin{lemma}
\label{uniformconnbd}
Let $G$ be a graph satisfying the assumptions in Theorem \ref{mainthmgeneral}
and consider percolation on it with $p\leq p_c(1+\eps)$.
Then, for any $r \geq \mnot$ and any vertex $x$,
    $$
    \prob_p \big (0\stackrel{P[\mnot,r]}{\lrfill} x \big ) \leq
    {1+O(\alpham + \eps \mnot) \over V} \E |B(r)|\, .
    $$
\end{lemma}

\begin{proof} The proof is similar to that of Lemma \ref{moreunifconnbd}. If the event occurs,
then there exists a vertex $v$ and a simple path $\omega$ of length $\mnot$
from $0$ to $v$ such that the event
    $$
    \{ \omega \hbox{ {\rm is open}}\} \circ \{v\stackrel{r}{\lrfill} x\} \, ,
    $$
occurs. Hence,
    $$
    \prob_{p}(0 \stackrel{P[\mnot,r]}{\lrfill} x) \leq p^{\mnot} \sum_v
    \sum_{\substack{\omega \colon |\omega|=\mnot \\ \omega[\mnot]=v}}
    \prob_p(v \stackrel{r}{\lrfill} x) \, ,
    $$
by the BK inequality. By the same argument,
    $$
    \prob_{p}(0 \stackrel{P[\mnot,r]}{\lrfill} x) \leq p^{\mnot+1} \sum_v
    \sum_{\substack{\omega \colon |\omega|=\mnot+1 \\ \omega[\mnot+1]=v}}
    \prob_p(v \stackrel{r}{\lrfill} x) \, .
    $$
The reason we make two such similar estimates is that due to possible periodicity, in each of the estimates the sum over $v$ may include only half the vertices of $V$ and not all of them. We now take the average of these two estimates, sum over $v$ to get the $\E|B(r)|$ factor and use the same analysis as in Lemma \ref{moreunifconnbd} using condition (2). This gives the required assertion of the lemma. \end{proof}

\begin{lemma}
\label{unifconnbd2}
Let $G$ be a graph satisfying the assumptions in Theorem \ref{mainthmgeneral}
and consider percolation on it with $p\leq p_c(1+\eps)$.
Then, for any $r \geq \mnot$ and any vertex $x$,
    $$
    \prob(0 \lrer x) \leq {1+O(\alpham+\eps \mnot)\over V}\E |B([r-m_0,2r-m_0])|\, .
    $$
\end{lemma}
\begin{proof} This follows by summing the estimate of Lemma \ref{moreunifconnbd}
and using the fact that
    $$
    \E |\partial B(t)| \leq p(m-1) \E |\partial B(t-1)|
    \leq (1+\m^{-1}+\eps) \E |\partial B(t-1)| \, ,
    $$
where the last inequality is due to condition (2) and the first inequality holds since,
given $\partial B(t-1)$, $|\partial B(t)|$ is stochastically bounded above by a sum of
$|\partial B(t-1)|$ binomial random variables with parameters $\m-1$ and $p$. The lemma follows since $\alpham\geq \m^{-1}$.
%Now we have that for any $k$
%$$ \E |\partial B(k+m_0)| \leq \E |\partial B(k)| m(m-1)^{\mnot-1}p^{\mnot} = (1+o(1)) \E |\partial B(k)| \, ,$$
%because the number of simple paths of length $m_0$ emanating from a vertex of $\partial B(k)$ is at most %$m(m-1)^{\mnot-1}$. We deduce that
%$$ \prob(0 \lrer x) \leq {1 +o(1) \over V} \sum_{r \leq t \leq 2r} \E |\partial B(t)| \, ,$$
%concluding the proof.
\end{proof}

We close this section with a remark which will become useful later. We often would like to have these uniform
connection bounds off some subsets of vertices. The proofs in the lemmas of this sections
immediately generalize to such a setting. This is because the claim ``the number of paths
from $0$ to $v$ of length $\mnot$ is at most $n$'' still holds
even if we are in $G \setminus A$, for any subset of vertices $A$. We state the required
assertion here and omit their proofs:

\begin{lemma}[Uniform connection bounds off sets]
\label{unifconnbdoff}
Consider percolation on $G=\{0,1\}^\m$ with $p=p_c(1+\eps)$,
and let $A$ be any subset of vertices.
Then, for any $r \geq \mnot$ and any vertex $x$,
\begin{eqnarray*}
\poffa(0 \stackrel{=r}{\lrfill} x) &\leq& {(2+O(\alpham+ \eps \mnot)) V^{-1}\E |\partial B(r-\mnot; A)| } \, , \\
\poffa \big ( 0 \stackrel{[\mnot,r]}{\lrfill} x \big ) &\leq& {(1+O(\alpham+ \eps \mnot))V^{-1} \E |B(r; A)|} \, .
\end{eqnarray*}
\end{lemma}

\subsection{Proof of part (a) of Theorem \ref{mainthmgeneral}.} \label{sec-nbwtriangle}
We demonstrate the use of Lemma \ref{uniformconnbd} by showing that the finite triangle condition holds under the assumptions of Theorem \ref{mainthmgeneral}. We begin with an easy calculation.
    \begin{claim}
    \label{reallysmalltriangle}
    On any regular graph $G$ of degree $\m$ and any vertices $x,y$,
    $$
    \sum_{u,v} \sum_{t_1, t_2, t_3 : t_1 + t_2 +t_3 \in \{0,1,2\}} \p^{t_1}(x,u) \p^{t_2}(u,v) \p^{t_3}(v,y) = {\bf 1}_{\{x=y\}} + O(\m^{-1}) \, ,
    $$
and
    $$
    \sum_{u,v} \sum_{t_1, t_2, t_3 : t_1 + t_2 +t_3 \in \{1,2\}} \p^{t_1}(x,u) \p^{t_2}(u,v) \p^{t_3}(v,y) = O(\m^{-1}) \, .
    $$
    \end{claim}
\begin{proof} We prove both statements simultaneously.
The contribution coming from $t_1+t_2+t_3=0$ is the one we get when $x=u=v=y$ giving ${\bf 1}_{\{x=y\}}$. The contributions coming from $t_1+t_2+t_3=1$ can only come from the cases $u=v=y$ and $(t_1,t_2,t_3) = (1,0,0)$, or $u=v=x$ and $(t_1,t_2,t_3) = (0,0,1)$, or $u=x$ and $v=y$ and $(t_1,t_2,t_3) = (0,1,0)$. These are easily bounded using the fact that $\max_{z} \p^t(w,z)\leq 1/(\m-1)$ for any $t\geq 1$. We perform a similar case analysis to bound the contributions of $t_1+t_2+t_3=2$. If $(t_1,t_2,t_3)=(0,0,2)$, then we must have $u=v=x$ and $\p^2(v,y) = O(m^{-1})$, this argument also handles the case where one of the other $t_i$'s is $2$. In the case $(t_1,t_2,t_3)=(1,1,0)$ we must have that $v=y$ and that $u$ is a neighbor both of $x$ and $y$. There are at most $m$ such $u$'s and for each we have that $\p^1(x,u)\p^1(u,y) = O(\m^{-2})$. The case $(t_1,t_2,t_3)\in \{(0,1,1), (1,0,1)\}$ are handled similarly.
\end{proof}

%
%In fact, let us demonstrate the
%use of the inequality by showing that the finite triangle condition holds under our assumptions.
%\\

\noindent {\bf Proof of part (a) of Theorem \ref{mainthmgeneral}}. Let $p \leq p_c$. If one of the connections
in the sum $\nabla_{p}(x,y)$ is of length in $[\mnot, \infty)$, say between $x$ and $u$, then we may estimate
    \begin{eqnarray*}
    \sum _{u,v} \prob_{p}(x\stackrel{[\mnot, \infty)}{\lrfill} u) \prob_{p}(u \lr v) \prob_{p}(v \lr y)
     &\leq& {(1+o(1)) \E_{p} |\C(0)| \over V} \sum_{u,v} \prob_{p}(u \lr v) \prob_{p}(v \lr y)
     \\ &=& {(1+o(1)) (\E_{p} |\C(0)|)^3 \over V} \, ,
    \end{eqnarray*}
where we used Lemma \ref{uniformconnbd} (and took $r\to \infty$ in both sides of the lemma) for the first inequality. Thus, we are only left to deal with short connections,
    $$
    \nabla_{p}(x,y) \leq \sum_{u,v} \prob_{p}(x \stackrel{\mnot}{\lrfill} u)
    \prob_{p}(u \stackrel{\mnot}{\lrfill} v)\prob_{p}(v \stackrel{\mnot}{\lrfill} y) + O(V^{-1} \chi(p)^3) \, .
    $$
We write
    $$
    \prob_{p}(x \stackrel{\mnot}{\lrfill} u) = \sum_{t_1=0}^{\mnot} \prob_{p}(x \stackrel{=t_1}{\lrfill} u) \, ,
    $$
and do the same for all three terms so that
    \be\label{triangle.midstep} \nabla_{p}(x,y) \leq \sum_{u,v} \sum_{t_1,t_2,t_3}^{\mnot} \prob_{p}(x \stackrel{=t_1}{\lrfill} u) \prob_{p}(u \stackrel{=t_2}{\lrfill} v) \prob_{p}(v \stackrel{=t_3}{\lrfill} y) + O(V^{-1} \chi(p)^3) \, .
    \ee
We bound
    $$
    \prob_{p}(x \stackrel{=t_1}{\lrfill} u) \leq \m(\m-1)^{t_1-1} \p^{t_1}(x,u) p^{t_1} \, ,
    $$
simply because $\m(\m-1)^{t_1-1} \p^{t_1}(x,u)$ is an upper bound on the number of
simple paths of length $t_1$ starting at $x$ and ending at $u$. Hence
    $$ \nabla_{p}(x,y) \leq {\m^3 \over (\m-1)^3} \sum_{u,v} \sum_{\substack{t_1,t_2,t_3}}^{\mnot} [p(m-1)]^{t_1+t_2+t_3} \p^{t_1}(x,u) \p^{t_2}(u,v) \p^{t_3}(v,y) + O(V^{-1} \chi(p)^3) \, .$$
Since $p \leq p_c$, assumption (2) gives that $[p(m-1)]^{t_1+t_2+t_3} = 1 + O(\alpham)$, and condition (3) together with Claim \ref{reallysmalltriangle} yields that
$$ \nabla_{p}(x,y) \leq {\bf 1}_{\{x=y\}} + O(V^{-1} \chi(p)^3) + O(\m^{-1} + \alpham / \log V) \, ,$$
concluding the proof. \qed

\subsection{Extended triangle and square diagrams}
\label{sec-triangle-square}
In this section, we provide several extensions to the triangle condition \eqref{finitetriangle}. We will bound the triangle diagram in the supercritical phase (which requires bounding the length of connections, otherwise the sums blow up) and estimate a square diagram which will be useful in a key second moment calculation in Section \ref{sec-large-clusters-close}.
%We will use the estimates of the previous section in the same manner we did in Section \ref{sec-unif-conn-NBW}.

\begin{lemma}[Short supercritical triangles]
\label{shorttriangle}
Let $G$ be a graph satisfying the assumptions in Theorem \ref{mainthmgeneral}
and consider percolation on it with $p\leq p_c(1+\eps)$.
Then,
    $$
    \max_{x,y}\sum_{u,v\colon \{u,v\} \neq \{0,0\}} \prob_p(x \stackrel{\mnot}{\lrfill} u)\prob_p(u \stackrel{\mnot}{\lrfill} v)
    \prob_p(v \stackrel{\mnot}{\lrfill} y) = O(\alpham +\eps \mnot) \, .
    $$
\end{lemma}
\begin{proof}
This follows immediately by assumptions (2) and (3) of
Theorem \ref{mainthmgeneral} and the usual bound
    $$
    \prob_p(x \stackrel{=s}{\lrfill} u) \leq p^s \m (\m-1)^{s-1} \p^s(x,u) \, ,
    $$
as we did before.
\end{proof}

\begin{corollary}[Long supercritical triangles]
\label{extendedtriangle}
Let $G$ be a graph satisfying the assumptions in Theorem \ref{mainthmgeneral}
and consider percolation on it with $p\leq p_c(1+\eps)$.
Let $r_1, r_2, r_3$ be integers that are all at least $\mnot$. Then,
    \eqan{
    &\max_{x,y}\sum_{u,v\colon \{u,v\} \neq \{0,0\}} \prob(x \stackrel{r_1}{\lrfill} u)
    \prob(u \stackrel{r_2}{\lrfill} v) \prob(v \stackrel{r_3}{\lrfill} y)\nn\\
    &\qquad\quad \leq O(\alpham + \eps \mnot) +
    {(3+O(\alpham+\eps \mnot)) \over V}\E|B(r_1)|\E|B(r_2)|\E|B(r_3)| \, .
    }
\end{corollary}

\begin{proof} We split the sum into two cases. The first case is that at least one of the
connection events occurs with a path of length at least $\mnot$. For instance,
if $0 \stackrel{P[\mnot, r_1]}{\lrfill} u$ occurs, then we use Lemma \ref{uniformconnbd} to bound,
uniformly in $x,y$,
    \eqan{
    &\sum_{u,v} \prob(x \stackrel{P[\mnot,r_1]}{\lrfill} u)
    \prob(u \stackrel{r_2}{\lrfill} v) \prob(v \stackrel{r_3}{\lrfill} y) \nn\\
    &\qquad \leq {(1 + O(\alpham+ \eps \mnot))\over V}  \E |B(r_1)| \sum_{u,v} \prob(u \stackrel{r_2}{\lrfill} v) \prob(v \stackrel{r_3}{\lrfill} y) \nn\\
    &\qquad = {(1 + O(\alpham + \eps \mnot))\over V} \E |B(r_1)| \E |B(r_2)| \E |B(r_3)| \, .\nn
    }
The second case is when all the connections occur with paths of length at most $\mnot$,
in which case we use Lemma \ref{shorttriangle} and get a $O(\alpham+\eps \mnot)$ bound. This concludes the proof.
\end{proof}

\begin{lemma}[Supercritical square diagram]
\label{squarediagram}
Let $G$ be a graph satisfying the assumptions in Theorem \ref{mainthmgeneral}
and consider percolation on it with $p\leq p_c(1+\eps)$. Let $r_1,r_2$ be both
at least $\mnot$. Then,
    $$
    \sum_{(u_1,u_1'), (u_2,u_2'), z_1,z_2} \!\!\!\!\prob(z_1 \stackrel{r_1}{\lrfill} u_1)
    \prob(z_1 \stackrel{r_1}{\lrfill} u_2) \prob(z_2 \stackrel{r_2}{\lrfill} u_1')
    \prob(z_2 \stackrel{r_2}{\lrfill} u_2') \leq C\m^2 (\E|B(r_1)|)^2(\E|B(r_2)|)^2 +CV\m^2 \mnot\alpham\, .
    $$
%When $p=p_c(1+\vep)$ with $\eps=o(1)$ and $\eps\gg V^{-1/3}$,
%and $r_1,r_2\geq r_0$ with $r_0$ chosen as in
%\eqref{rzerocond}, the second term can be omitted.
\end{lemma}

\begin{figure}
\begin{center}
\includegraphics{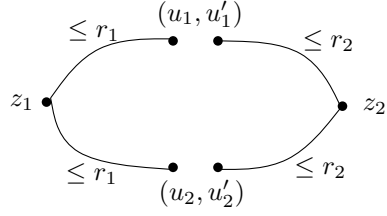}
\caption{A square diagram.}
\label{fig.square}
\end{center}
\end{figure}

\begin{proof} See Figure \ref{fig.square}. If one of the connections is of
length at least $\mnot$, then we use Lemma \ref{uniformconnbd} and the summation simplifies.
For instance, if $u_1 \stackrel{[\mnot,r_1]}{\lrfill} z_1$, then we use Lemma \ref{uniformconnbd} and
sum over $z_1$, followed by sums over $u_1$ and $u_2$. This gives a bound of
    $$
    { C \m^2(\E|B(r_1)|)^2 \over V}
    \sum_{u_1',u_2',z_2} \prob(z_2 \stackrel{r_2}{\lrfill} u_1') \prob(z_2 \stackrel{r_2}{\lrfill} u_2')
    \leq C\m^2[\E |B(r_1)|^2\E|B(r_2)|]^2\, ,
    $$
where $C>1$ is an upper bound on $1+O(\alpham+\eps \mnot)$.
%When all connections have length zero, meaning that
%$z_1=u_1=u_2$ and $z_2=u_1'=u_2'$, then we
%obtain a bound of $\m V\leq \m^2\alpham V$ since $\alpham\geq 1/\m$,
%which satisfies the required bound.
%Therefore, without loss of generality, assume that $z_2\neq u_1'$
%(the other cases being equal by symmetry).

We are left to bound the sum
    \eqan{
    \lbeq{square-diagram}
    &\sum_{(u_1,u_1'), (u_2,u_2'), z_1,z_2} \prob(z_1 \stackrel{\mnot}{\lrfill} u_1)
    \prob(z_1 \stackrel{\mnot}{\lrfill} u_2) \prob(z_2 \stackrel{\mnot}{\lrfill} u_1')
    \prob(z_2 \stackrel{\mnot}{\lrfill} u_2')\\
    &\qquad =V \sum_{(u_1,u_1'), (u_2,u_2'),z_2} \prob(0 \stackrel{\mnot}{\lrfill} u_1)
    \prob(0 \stackrel{\mnot}{\lrfill} u_2) \prob(z_2 \stackrel{\mnot}{\lrfill} u_1')
    \prob(z_2 \stackrel{\mnot}{\lrfill} u_2')\, ,\nn
    }
by transitivity. We write this sum as $V \sum_{u_2'} f(u_2') g(u_2')$,
where
    \eqn{
    g(u_2')=\sum_{(u_1,u_1'),z_2} \prob(0 \stackrel{\mnot}{\lrfill} u_1)
    \prob(z_2 \stackrel{\mnot}{\lrfill} u_1')
    \prob(z_2 \stackrel{\mnot}{\lrfill} u_2'),
    \qquad
    f(u_2')=\sum_{u_2\sim u_2'} \prob(0 \stackrel{\mnot}{\lrfill} u_2).
    }
We then bound
    \eqan{
    V \sum_{u_2'} f(u_2') g(u_2')
    &\leq V \big(\sum_{u_2'} f(u_2')\big)\big(\max_{u_2'} g(u_2')\big)\\
    &=V \m \expec|B(\mnot)|
    \max_x \sum_{(u_1,u_1'),z_2} \prob(0 \stackrel{\mnot}{\lrfill} u_1)
    \prob(z_2 \stackrel{\mnot}{\lrfill} u_1')
    \prob(z_2 \stackrel{\mnot}{\lrfill} x).\nn
    }
%%% Do We really need the explanation about transitivity here? I think it's pretty obvious.
%Since $G$ is transitive, there exists a bijective map $\phi_{\sss (u_1,u_1')}\colon V(G)\to V(G)$
%such that $\phi_{\sss (u_1,u_1')}(u_1)=u_1'$. Then, $\prob(0 \stackrel{\mnot}{\lrfill} u_1)=\prob(\phi_{\sss (u_1,u_1')}(0) \stackrel{\mnot}{\lrfill}u_1'),$
%so that
%    \eqan{
%    V \sum_{u_2'} f(u_2') g(u_2')
%    &=V \m \expec|B(\mnot)|
%    \max_x \sum_{(u_1,u_1'),z_2\colon z_2\neq u_1'} \prob(\phi_{\sss (u_1,u_1')}(0) \stackrel{\mnot}{\lrfill} u_1')
%    \prob(z_2 \stackrel{\mnot}{\lrfill} u_1')
%    \prob(z_2 \stackrel{\mnot}{\lrfill} x)\\
%    &\leq V \m^2 \expec|B(\mnot)|
%    \max_{x,y} \sum_{u_1'\neq z_2} \prob(y\stackrel{\mnot}{\lrfill} u_1')
%    \prob(z_2 \stackrel{\mnot}{\lrfill} u_1')
%    \prob(z_2 \stackrel{\mnot}{\lrfill} x).\nn
%    }
By condition (2) in Theorem \ref{mainthmgeneral},
we can write the above as
    \eqan{
    V \sum_{u_2'} f(u_2') g(u_2')
    &\leq CV\m^2
    \expec|B(\mnot)|
    \max_x \sum_{u_1,u_1',z_2} \sum_{t_1, t_2, t_3\geq 0}^{\mnot}
    \p^{t_1}(0,u_1)\p^1(u_1,u_1')
    \p^{t_2}(u_1',z_2)
    \p^{t_3}(z_2,x)\\
    &\leq CV\m^2\expec|B(\mnot)|
    \max_x \sum_{u_1,z_2} \sum_{t_1, t_2, t_3\geq 0}^{\mnot}
    \p^{t_1}(0,u_1)[\p^{t_2+1}(u_1,z_2)+\frac{1}{\m-1}\p^{t_2-1}(u_1,z_2)]
    \p^{t_3}(z_2,x)\nn\\
    &\leq CV\m^2 \expec|B(\mnot)| (\alpham+O(1/\m))
    \leq CV\m^2\expec|B(\mnot)|\alpham,\nn
    }
where we use Lemma \ref{lem-NBW-ext} in the second inequality, and
Claim \ref{reallysmalltriangle}, condition (2) in Theorem \ref{mainthmgeneral}
and $\alpham\geq 1/\m$ in the final inequality.
Further, $\expec|B(\mnot)|=O(\mnot)$ by Corollary \ref{corrlength}
and the fact that $\mnot=o(\vep^{-1})$.
%
%To show that the second term can be omitted when $p=p_c(1+\vep)$ with
%$\eps=o(1)$ and $\eps\gg V^{-1/3}$,
%and $r_1,r_2\geq r_0$ with $r_0$ chosen as in \eqref{rzerocond}, we note that
%    $$
%    \m^2 (\E|B(r_1)|)^2(\E|B(r_2)|)^2
%    \geq \Theta(1)\m^2 \frac{(\eps V)^2}{(\log(\eps^3 V))^4}
%    =\Theta(1) (\m^2 V)\frac{\eps^2 V}{(\log(\eps^3 V))^4},
%    $$
%which is much larger than $\m^2 V \mnot \alpham$, since $\vep\mnot=o(1)$,
%$\alpham=o(1)$, while $\eps^3 V/(\log(\eps^3 V))^4\rightarrow \infty$.
This concludes our proof.
\end{proof}

\section{Volume estimates}
\label{sec-exp-vol-est}
In this section, we study the expected volume of intrinsic balls and their boundaries at various radii in both the critical and supercritical phase.
%We start with a result that allows us to effectively decouple connection events.
%This result is in the spirit of the \emph{tree-graph inequalities} in \cite{AN}
%but since we are dealing with non-monotone events, we need to use BK-Reimer inequality rather delicately.
\subsection{In the critical regime}
\label{sec-volcrit} Given a subset of vertices $A$ and integer $r\geq 0$ we write
    $$
    G(r; A)=\E|\partial B(r;A)|,
    \qquad
    G(r)=\max_{A \subseteq V(G)} G(r; A) \, .
    $$

\begin{theorem}[Expected boundary size] \label{bdry}Let $G$ be a graph satisfying the assumptions of Theorem \ref{mainthmgeneral} and consider percolation on it with $p= p_c$. Then there exists a constant $C>0$ such that for any integer $r$,
    $$
    G(r) \leq C \, .
    $$
\end{theorem}
\begin{proof} Define $F(r) = \E |B(r)|$ and $F(r; A) = \E|B(r;A)|$, so that
$F(r; A) \leq F(r)$ for all subsets $A$. Define $G^*(r)=\max _{s \leq r} G(s),$ and
let $r\geq 2m_0$ be a maximizer of $G^*$, that is, $r$ is such that  $G(r') \leq G(r)$
for any $r'< r$. Let $A=A(r)$ be the subset of vertices which maximizes
$G(r; A)$ so that $G(r; A)=G(r)=G^*(r)$. Given such $r$ and $A=A(r)$ we will prove that
there exists $c>0$ such that for any integer $s \geq 0$,
    \be
    \label{bdry.recurse}
    F(r + s; A) \geq c G^*(r) F(s; A) \, .
    \ee
%First we may assume that $r$ is the maximizer in the definition of $G^*(r)$. Otherwise, let $r_0 < r$ be %the maximizer and then we prove that
%$$ F(r_0 + s) \geq c G^*(r_0) F(s) \, ,$$
%which implies the first inequality since $F$ is a monotone function. Let $A$ be the subset of vertices which maximizes $G(r; A)$ so that $G^*(r) = G(r; A)$.
We begin by bounding
    $$
    F(r+s; A) \geq \sum_{v} \poffa (0 \stackrel{(r,r+s]}{\lrfill} v) \, .
    $$
For a vertex $u$ we define $\C^u(0; A) = \{x : 0 \lr x \off A \cup \{u\}\}$. Now, for any
vertex $v$, if there exists $u\neq v$ such that $0 \stackrel{=r}{\lrfill} u \off A$ and
$u \lrs v \off \C^u(0; A)$, then $0 \stackrel{(r,r+s]}{\lrfill} v \off A$. Furthermore,
if such $u$ exists, then it is unique because otherwise $v\in \C^u(0; A)$. We deduce that
    $$
    F(r+s; A) \geq \sum_{v\neq u} \poffa(0 \stackrel{=r}{\lrfill} u \and \{u \lrs v \off \C^u(0; A)\}) \, .
    $$
We now condition on $\C^u(0; A)=H$ for some admissible $H$ (that is, for which
the probability of the event $\C^u(0;A)=H$ is positive, and in which $0 \stackrel{=r}{\lrfill} u$ occurs). In this conditioning, we also condition on the status of all edges touching $H$. Note that by definition
$A \cap H = \emptyset$. We can write the right hand side of the last inequality as
    $$
    \sum_{v\neq u} \sum_{H \colon 0 \stackrel{=r}{\lrfill}{u} \off A}
    \poffa (\C^u(0; A)=H) \poffa(u \lrs v \off H) \, ,
    $$
in the same way as we derived (\ref{offdescription}). This can be rewritten as
    $$
    \sum_{v \neq u} \sum_{H \colon 0 \stackrel{=r}{\lr}{u} \off A} \poffa(\C^u(0; A)=H)
    \big [ \poffa(u \lrs v) - \poffa(u \lrs v \onlyon H) \big ] \, .
    $$
We open the parenthesis and have that the first part of this sum equals precisely $G^*(r) F(s; A)$ since $r$ and $A$ were  maximizers. We need to show
that the second part of the sum is of lower order. To that aim, if $u \lrs v \onlyon H$,
then there exists $h\in H$ such that $h \neq u$ and $\{u \lrs h\}\circ\{h \lrs v\}$. By the BK inequality,
we bound the second part of the sum above by
    $$
    \sum_{u} \sum_{H \colon 0 \stackrel{=r}{\lrfill}{u} \off A}
    \sum_{h\in H, h \neq u, v} \poffa(\C^u(0; A)=H) \poffa(u \lrs h) \poffa(h \lrs v) \, .
    $$
Summing over $v$ and changing the order of the summation gives that the last sum is at most
    $$
    F(s; A) \sum_{u \neq h}  \poffa(0 \stackrel{=r}{\lr}{u}, 0 \lr h) \poffa(u \lrs h) \, .
    $$
We bound this from above using Lemma \ref{goodbkreimer} by
    \be\label{bdry.midstep}
    F(s; A) \sum_{u \neq h,z} \sum_{t\leq r} \poffa(0 \stackrel{=t}{\lrfill} z)
    \max_{D \subseteq V(G)} \prob(z \stackrel{=r-t}{\lrfill} u \off A\cup D) \poffa(z \lr h) \poffa(u \lrs h) \, .
    \ee
We sum this separately for $t \leq r-\mnot$ and $t \in [r-\mnot,r]$. For $t \leq r-\mnot$,
we bound for any $D \subset V$
    $$
    \prob(z \stackrel{=r-t}{\lrfill} u \off A\cup D) \leq {3G(r-t-\mnot;A \cup D) \over V}
    \leq {3 G^*(r) \over V} \, ,
    $$
where the first inequality is by Lemma \ref{unifconnbdoff} and the second
by definition of $G^*(r)$. Hence, the sum over $t \leq r-\mnot$ in (\ref{bdry.midstep}) is at most
    $$ {3G^*(r)F(s; A) \over V} \sum_{u,h,z}  \poffa(0 \lrr z) \poffa(z \lr h) \poffa(h \lrs u) \nonumber\\
    \leq {3G^*(r)F(s; A) (\E|\C(0)|)^3 \over V} = 3\lambda^3 G^*(r)F(s; A) \, ,
    $$
where the inequality we got by summing over $u, h$ and $z$ (in that order) and the equality is due to the definition of $p_c$ in (\ref{pcdef}). Our $\lambda=1/10$ is chosen small enough so that $3 \lambda^3 \leq 1/2$.

We now bound the sum in (\ref{bdry.midstep}) for $t \in [r-\mnot,r]$. We first bound
    $$ \poffa(0 \stackrel{=t}{\lrfill} z) \leq {3 G^*(r) \over V} \, ,$$
as we did before using Lemma \ref{unifconnbdoff} and pull that term out of the sum. This gives an upper bound of
    $$ {3G^*(r)F(s; A) \over V} \sum_{u\neq h,z} \sum_{s_1=0}^{\mnot} \max_{D \subseteq V(G)} \prob(z \stackrel{=s_1}{\lrfill} u \off A\cup D) \poffa(z \lr h) \poffa(u \lrs h) \, .$$
We would like to sum the first term in the sum over $s_1$ and get a contribution of $\prob(z \stackrel{\mnot}{\lrfill} u)$. We cannot do that however, because the maximizing set $D$ may depend on $s_1$ so these are not necessarily disjoint events. Instead we bound for all $D \subseteq V(G)$
    $$ \prob(z \stackrel{=s_1}{\lrfill} u \off A\cup D) \leq m(m-1)^{s_1} p_c^{s_1} \p^{s_1}(z,u) \leq (1+o(1)) \p^{s_1}(z,u) \, \,$$
where the first inequality is since $m(m-1)^{s_1}\p^{s_1}(z,u)$ bounds the number of simple paths of length $s_1$ connecting $z$ to $u$ and the second inequality is due to condition (2) of Theorem \ref{mainthmgeneral}. Now, if one of the connections $z \lr h$ or $u \lrs h$ is in fact a connection of length at least $\mnot$ we use Lemma \ref{uniformconnbd} to simplify the sum. For instance, if the connection is $z \stackrel{[\mnot,\infty)}{\lrfill} h$, then we bound the probability of this by $2V^{-1}\E|\C(0)|$ and the sum simplifies to
    $$ {4 G^*(r)F(s; A) \E|\C(0)| \over V^2} \sum_{u\neq h,z} \sum_{s_1=0}^{\mnot} \p^{s_1}(z,u) \prob(u \lrs h) \, ,$$
we then sum over $h,u, z$ and $s_1$ and get a contribution of
    $$ {4 G^*(r)F(s; A) \mnot (\E|\C(0)|)^2 \over V} = o(G^*(r)F(s; A)) \, ,$$
since $(\E|\C(0)|)^2 = O( V^{2/3})$ and $\mnot = o(\eps^{-1}) = o(V^{1/3})$. Thus, it remains to bound
    $$ {3G^*(r)F(s; A) \over V} \sum_{u\neq h,z} \sum_{s_1=0}^{\mnot} \p^{s_1}(z,u) \poffa(z \stackrel{\mnot}{\lrfill} h) \poffa(u \stackrel{\mnot}{\lrfill} h) \, .$$
We bound this by
    $$ {CG^*(r)F(s; A) \over V} \sum_{u\neq h,z} \sum_{s_1=0,s_2=1,s_3=0}^{\mnot} \p^{s_1}(z,u) \p^{s_2}(u,h) \p^{s_3}(h,z) = o(1) \cdot G^*(r)F(s; A) \, ,$$
where we used Claim \ref{reallysmalltriangle} and condition (3) of Theorem \ref{mainthmgeneral}. This concludes the proof of (\ref{bdry.recurse}).

%by the finite triangle condition (\ref{finitetriangle}) and $u \neq 0$ ((\ref{finitetriangle}) holds by part (a) of Theorem \ref{mainthmgeneral}). Recall that $\lambda$ is chosen so that $a_0(\lambda)$
%is small ($a_0(\lambda) < 1/3$ suffices here). For $t \leq \mnot$ we have that $r-t \geq \mnot$ and so we may bound
%    $$
%    \poffa(z \stackrel{=r-t}{\lrfill} u \off B_0(t))
%    \leq {3G(r-t-\mnot) \over V} \leq {3G^*(r) \over V} \, ,
%    $$
%again by Lemma \ref{unifconnbdoff} and the definition of $G^*$. Thus, the sum over $t \leq \mnot$ in (\ref{bdry.midstep}) is at most
%    $$
%    3F(s; A) \sum _{u,h,z} {G^*(r) \over V}
%    \poffa(0 \stackrel{\mnot}{\lrfill} z)\poffa(z \lr h) \poffa(u \lrs h) \, .
%    $$
%We sum this over $u$, then over $h$ and then over $z$
%and use Theorem \ref{KozmaNachmias} to bound this by
%    $$
%    {C F(s;A) G^*(r) V^{-1} \mnot (\E_{p_c}|\C(0)|)^2} = o(1) F(s;A) G^*(r) \, ,
%    $$
%by the definition of $p_c$ in (\ref{pcdef}) and since $\mnot = o(\vep^{-1})=o(V^{1/3})$.

We now turn to prove the main result assuming (\ref{bdry.recurse}).
First, for any $r \leq 2\mnot$ we have that the number of non-backtracking
paths emanating from $0$ is at most $\m(\m-1)^{r-1}$ and hence, for any $A$,
    $$
    G(r;A) \leq \m(\m-1)^{r-1} p_c^r = 1+o(1) \, ,
    $$
by condition (2) of Theorem \ref{mainthmgeneral}. It remains to consider the case $r \geq 2\mnot$.
Assume by contradiction that there exists some $r \geq 2\mnot$ such that $r$ is
the maximizer in the definition $G^*(r)$ and that $G^*(r) \geq 2/c$ where $c$ is
the constant from (\ref{bdry.recurse}). Fix such $r$ and let $A=A(r)$ be the
maximizing set as in (\ref{bdry.recurse}). Now, putting $s=r$ in (\ref{bdry.recurse}) gives
    $$
    F(2r;A) \geq c G^*(r) F(r;A) \geq 2 F(r;A) \, .
    $$
Putting $s=2r$ in (\ref{bdry.recurse}) gives
    $$
    F(3r; A) \geq c G^*(r) F(2r;A) \geq 4 F(r;A) \, ,
    $$
and so, by induction, for any $k$,
    $$ F(kr; A) \geq 2^{k-1} F(r; A) \, .
    $$
We have reached a contradiction, since on the right hand side we have a quantity
going to $\infty$ in $k$ (note that $A$ cannot contain $0$, otherwise it will not be maximizing,
so $F(r; A)\geq 1$) and on the left hand side our quantity is bounded by $V$.
\end{proof}

We now wish to obtain the reverse inequality to Theorem \ref{bdry}, that is,
a lower bound to $\E |\partial B(r)|$. Of course, this cannot hold for all $r$
but it turns out to hold as long as $r \ll V^{1/3}$. This is the correct upper bound on $r$ because the diameter of critical
clusters is of order $V^{1/3}$ (see \cite{NP3}).

\begin{lemma}[Lower bound on critical expected ball]
\label{lowercritball} Let $G$ be any transitive finite graph and put $p=p_c$
where $p_c$ is defined in
\eqref{pcdefgeneral}. Then, there exists a constant $\xi>0$ such
that for all $r \leq \xi V^{1/3}$,
    $$
    \E |B(r)| \geq r/4 \, .
    $$
\end{lemma}

\begin{proof}
%By monotonicity, $\E |B_0(r;A)| \leq \E|B_0(r)|$
%for any subset $A$.
For convenience write $c=1/4$. Assume by contradiction that $\E |B(r)|\leq cr$.
Given this assumption, we will prove by
induction that for any integer $k \geq 0$,
    \be
    \label{indhyp}
    \E\big |B\big ([r(1 + k/2), r(1 + (k+1)/2)] \big ) \big | \leq 2^{k+1} c^{k+2} r \,
    .
    \ee
Let us begin with the case $k=0$. Since $\E |B(r)|\leq cr$ there exists $r'\in[r/2,r]$
such that $\E |\partial B(r')| \leq 2c$. Hence,
    \begin{eqnarray*}
    \E |B([r,3r/2])| &=& \sum_{A}\prob(B(r')=A)\E [ |B([r,3r/2])| \, \mid \, B(r')=A ] \\
    &=&  \sum_{A}\prob(B(r')=A)\E \big [ \sum_{a \in \partial A} \E |B_a(3r/2-r';A)| \big ]\\
    &\leq& cr \E |\partial B(r')| \leq 2c^2 r \,
    ,\end{eqnarray*}
where the inequality follows since $\E |B_a(3r/2-r';A)| \leq \E|B(r)|\leq cr$ by monotonicity
and for any $A$. Assume now that (\ref{indhyp}) holds for some $k$ and we prove it for
$k+1$. Since it holds for $k$,
we have that there exists $r' \in  [r(1 + k/2), r(1 + (k+1)/2)]$ such that
$\E |\partial B(r')| \leq 2^{k+2}c^{k+2}$. By conditioning on $B(r')=A$ as before we get that
    $$
    \E\big |B\big ([r(1 + (k+1)/2), r(1 + (k+2)/2)] \big ) \big | \leq cr \cdot  2^{k+2}c^{k+2} \, ,
    $$
concluding the proof (\ref{indhyp}). Now, since $c<1/2$ it is clear that the
sum over $k$ of (\ref{indhyp}) is at most $Cr$, contradicting the fact that
$\E_{p_c}|\C(0)|=\lambda V^{1/3}$ by our definition of $p_c$ in (\ref{pcdefgeneral}).
Note that the constant $\xi$ may depend on $\lambda$.
\end{proof}

\begin{lemma}[Lower bound on expected boundary size]
\label{lowercritbdry} Let $G$ be a transitive finite graph for which (\ref{finitetriangle}) holds
 and put $p=p_c$. Then there exists constants $c,\xi>0$ such
that for any $r \leq \xi V^{1/3}$,
    $$
    \E |\partial B(r)| \geq c \, .
    $$
\end{lemma}
\begin{proof}
By Lemma \ref{lowercritball} and Theorem \ref{KozmaNachmias} we have that
$\E |B([2r,Cr])| \geq r$ for some large fixed $C>0$. Also,
    \begin{eqnarray*}
    \E |B(Cr)|^2 \leq \sum_{x,y} \prob(0 \stackrel{Cr}{\lrfill} x, 0 \stackrel{Cr}{\lrfill} y)
    \leq  \sum_{x,y,z} \prob(0 \stackrel{Cr}{\lrfill} z) \prob(z \stackrel{Cr}{\lrfill} x)
    \prob( z \stackrel{Cr}{\lrfill} y ) \leq Cr^3 \, ,
    \end{eqnarray*}
by Theorem \ref{KozmaNachmias}. By the inequality
    \eqn{
    \label{X>a-prob-bd}
    \prob(X > a)\geq (\E X-a)^2/\E X^2
    }
valid for any non-negative random variable $X$ and $a<\E X$, with $a=0$,
    $$
    \prob(\partial B(2r) \neq \emptyset) \geq c/r \, ,
    $$
for some $c>0$. Furthermore, given $B(r)$, each vertex of $\partial B(r)$
has probability at most $Cr^{-1}$ of reaching $\partial B(2r)$ by
Theorem \ref{KozmaNachmias}. Hence, for any $\zeta>0$,
    $$
    \prob(\partial B(2r) \neq \emptyset \and |\partial B(r)| \leq \zeta r) \leq C\zeta/r \, .
    $$
We now have
    \eqan{
    \prob(|\partial B(r)| \geq \zeta r ) &
    \geq \prob(\partial B(2r) \neq \emptyset) - \prob(\partial B(2r) \neq \emptyset \and |\partial B(r)| \leq \zeta r)\nn\\
    &\geq {c \over r} - {C \zeta \over r} \, ,\nn
    }
and the lemma follows by choosing $\zeta>0$ small enough.
\end{proof}

\subsection{In the supercritical regime}
\label{sec-volume-super}
In this section, we extend the volume estimates to
the supercritical regime. The following is an immediate corollary of Theorem \ref{bdry}:
\begin{lemma}[Upper bound on supercritical volume]
\label{upperball} Let $G$ be a graph satisfying the assumptions in Theorem \ref{mainthmgeneral} and consider percolation on it with $p\leq p_c(1+\eps)$. Then for any $r$ and any $A \subset V$ we have
    $$
    \E |\partial B(t; A)|\leq C (1+\eps)^t,
    \quad \and \quad \E |B(r; A)| \leq C\eps^{-1} (1+\eps)^r \, .
    $$
\end{lemma}
\begin{proof} The first assertion is immediate by Theorem \ref{bdry} and Lemma \ref{corrlength1}. The second assertion follows by summing the first over $t \leq r$.
\end{proof}

The corresponding lower bound is more complicated to obtain, and as before,
can only hold up to some value of $r$. In conjunction with Lemma \ref{upperball},
it identifies $\E |B(r)|=\Theta(\eps^{-1} (1+\eps)^r)$ for appropriate $r$'s.

\begin{theorem}[Lower bound on supercritical volume] \label{lowerball} Let $G$ be a graph satisfying the assumptions in Theorem \ref{mainthmgeneral} and consider percolation on it with $p = p_c(1+\eps)$. Then for any $r$ satisfying
    \be
    \label{ballcond}
    \E |B(r)| \leq {\eps^2 V \over (\log \eps^3 V)^4},
    \ee
the following bound holds:
    $$
    \E |B(r)| \geq c \eps^{-1} (1+\eps)^r \, .
    $$
\end{theorem}

%\todo{Remco: Make this a theorem? This is one of the main technical results...}
%Asaf: Yes.

\begin{proof} First, we may assume that
    \be
    \label{lowerball.rcond}
    r \leq \eps^{-1} \log(\eps^3 V) \, ,
    \ee
since otherwise the assumption of the lemma cannot hold together with
the conclusion. Recall now the simultaneous coupling (described at the end of Section \ref{sec-notation}) between percolation at $p_1=p_c$ and $p_2 = p_c(1+\eps)$. Let
    $$
    \A_\ell(x)=\{0 \stackrel{=\ell}{\lrfill} x \text{ in }G_{p_2}\},
    $$
and given a simple path $\eta$ of length $\ell$ between $0$ and $x$, write
    $$ \A_\ell(x,\eta) = \{0 \stackrel{=\ell}{\lrfill} x \text{ in }G_{p_2} \and \eta \text { is the lexicographical first $p_2$-open path between $0$ and $x$}\} \, ,$$
so that $\A_\ell(x) = \bigcup_{\eta} \A_{\ell}(x,\eta)$. Write $\B_\ell(x,\eta)$ for the event
that the edges of $\eta$ are in fact $p_1$-open (not just $p_2$).
We have that
    $$
    \A_\ell(x,\eta) \cap \B_\ell(x,\eta) \subseteq
    \{0 \stackrel{=\ell}{\lrfill} x \hbox{ {\rm in $G_{p_1}$}} \} \, ,
    $$
so,
    \be\label{lowerball.bigset}
    \bigcup _{\ell \in [r-\eps^{-1}, r], \eta} \A_\ell(x,\eta) \cap \B_\ell(x,\eta) \subseteq
    \{0 \stackrel{[r-\eps^{-1},r]}{\lrfill} x \hbox{ {\rm in $G_{p_1}$}} \} \, .
    \ee
We will show that
    \be
    \label{lowerball.toshow}
    \sum_{x} \prob \Big (0 \stackrel{[r-\eps^{-1},r]}{\lrfill} x \hbox{ {\rm in $G_{p_1}$}}
    \setminus  \bigcup _{\ell \in [r-\eps^{-1}, r], \eta} \A_\ell(x,\eta) \cap \B_\ell(x,\eta) \Big )
    = o(\eps^{-1})  \, ,
    \ee
and first complete the proof subject to \eqref{lowerball.toshow}. Since $\{\A_\ell(x,\eta) \cap \B_\ell(x,\eta)\}_{\ell, \eta}$ are disjoint events, (\ref{lowerball.bigset}) and (\ref{lowerball.toshow}) show that
    $$
    \sum_{x, \ell \in [r-\eps^{-1},r], \eta} \prob(\A_\ell(x, \eta) \cap \B_\ell(x,\eta)) \geq
    \E_{p_1} |B([r-\eps^{-1},r])| - o(\eps^{-1}) \geq c\eps^{-1} \, ,
    $$
where the last inequality used Lemma \ref{lowercritbdry} and the fact that $r \ll V^{1/3}$
by \eqref{lowerball.rcond} and $\vep\gg V^{-1/3}$. From this the required result follows since
    $$
    \prob(\B_\ell(x,\eta) \mid \A_\ell(x,\eta)) = (1+\eps)^{-\ell} \, ,
    $$
which implies that
    \eqan{
    \E_{p_2} |B(r)|
    &\geq \sum_{x,\ell\in[r-\eps^{-1},r], \eta} \prob(\A_\ell(x, \eta))\\
    &=\sum_{x,\ell\in[r-\eps^{-1},r], \eta} \prob(\A_\ell(x,\eta))
    \prob(\B_\ell(x,\eta) \mid \A_\ell(x,\eta))(1+\eps)^{\ell}\nn\\
    &\geq (1+\eps)^{r-\eps^{-1}} \sum_{x,\ell\in[r-\eps^{-1},r], \eta}
    \prob(\A_\ell(x,\eta)\cap \B_\ell(x,\eta))\nn\\
    &\geq (1+\eps)^{r-\eps^{-1}} c\eps^{-1} \, .\nn
    }

Thus, our main effort is to show (\ref{lowerball.toshow}) under the restriction
of (\ref{ballcond}) and (\ref{lowerball.rcond}). Fix $x$ and assume that the event
    \be\label{lowerball.badevent}
    \{ 0 \stackrel{[r-\eps^{-1},r]}{\lrfill} x \inn G_{p_1}\} \setminus
    \bigcup _{\ell \in [r-\eps^{-1}, r], \eta} \A_\ell(x,\eta) \cap \B_\ell(x,\eta) \, ,
    \ee
occurs. In words, this event means that either the shortest $p_2$-open path is
shorter than the shortest $p_1$-path or that they have the same length but the
lexicographically first shortest $p_2$-path contains an edge having value in
$[p_1,p_2]$. This implies that the $p_2$-path {\em shortcuts} the $p_1$-path.
Formally, given vertices $u,v$ and integers $\ell\in[r-\eps^{-1},r], k \in [0,\ell], t \in [2,\ell]$
with $k+t \leq \ell$ write $\T(u,v,x,\ell,k,t)$ for the event that there exists paths $\eta_1, \eta_2, \eta_3, \gamma$ in the graph such that
\begin{enumerate}
\item $\eta_1$ is a shortest $p_1$-open path of length $k$ connecting $0$ to $u$,
\item $\eta_2$ is a shortest $p_1$-open path of length $t$ connecting $u$ to $v$,
\item $\eta_3$ is a shortest $p_1$-open path of length $\ell-t-k$ connecting $v$ to $x$,
\item $B_x(\ell-k-t) \cap B_0(k) = \emptyset$ in $G_{p_1}$,
\item $\gamma$ is $p_2$-open path of length at most $t$ connecting $u$ to $v$ and one of the edges of $\gamma$ receives value in $[p_1,p_2]$, and
\item $\eta_1, \eta_2, \eta_3$ and $\gamma$ are disjoint paths,
\end{enumerate}
see Figure \ref{fig.excursion}. The event (\ref{lowerball.badevent}) implies that $\T(u,v,x,\ell,k,t)$ occurs for some $u,v,\ell,k,t$ satisfying the conditions above. Our treatment of the case $t \geq \mnot$ is easier than the case $t \leq \mnot$ so let us perform this first. When $t \geq \mnot$ we forget about condition (4) and the special edge with value $[p_1,p_2]$ in (5) and take a union over $\ell, k$ and $t \in [\mnot,r]$ of the event $\T(u,v,x,\ell,k,t)$. This union implies the existence of vertices $u,v$ such that the following events occur disjointly:
\begin{enumerate}
\item $0 \lrr u$ in $G_{p_1}$,
\item $u \stackrel{P[\mnot,r]}{\lrfill} v$ in $G_{p_1}$,
\item $v \lrr x$ in $G_{p_1}$,
\item $u \lrr v$ in $G_{p_2}$.
\end{enumerate}
Indeed, the witnesses to these (monotone) events are the paths $\eta_1, \eta_2, \eta_3, \gamma$. We now wish to use the BK inequality, however, as the astute reader
may have already noticed, our witnesses are not stated in an i.i.d.~ product
measure. Let us expand briefly on how we may still use the BK inequality.
We may consider our simultaneous coupling measure to be an i.i.d.~ product
measure by putting on each edge a countable infinite sequence of independent
random bits receiving $0$ with probability $1/2$ and $1$ otherwise such that
this sequence encodes the uniform $[0,1]$ random variable attached to each edge.
In this setting, a witness for an edge being $p$-open is the sequence of bits
attached to the edge and similarly for the edge being $p$-closed. Similarly,
we define this way events of the form ``$E_1$ in $G_{p_1}$ occurs disjointly
from $E_2$ in $G_{p_2}$''. With this definition of witnesses we may use
the BK inequality here to bound the probability of the union above and sum over $x$ (as in (\ref{lowerball.toshow}). This gives an upper bound of
$$ \sum_{u,v,x} \prob_{p_1} (0 \lrr u) \prob_{p_1}(u \stackrel{P[\mnot,r]}{\lrfill} v) \prob_{p_1}(v \lrr x) \prob_{p_2}(u \lrr v) \, .$$
\begin{figure}
\begin{center}
\includegraphics{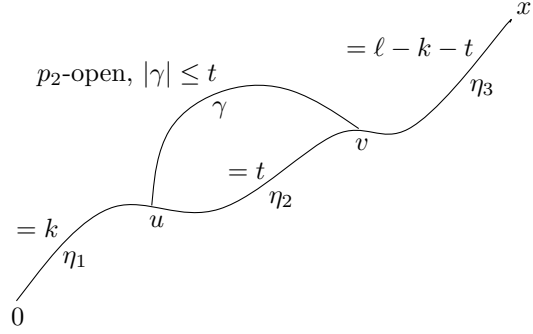}
\caption{The event $\T(u, v, x, \ell, k, t)$ --- an excursion.}
\label{fig.excursion}
\end{center}
\end{figure}
We sum over $x$ and get a factor of $r$ by Theorem \ref{KozmaNachmias}. We bound $\prob_{p_1}(u \stackrel{P[\mnot,r]}{\lrfill} v) \leq CrV^{-1}$ by Lemma \ref{uniformconnbd} and Theorem \ref{KozmaNachmias}. We then sum over $v$ and get a factor of $\E_{p_2} |B(r)|$ and on $u$ to get another factor of $r$. All together this gives an upper bound of
$$ {C r^3 \E_{p_2}|B(r)| \over V} = O(\eps^{-1} \log^{-1} (\eps^3 V)) = o(\eps^{-1}) \, ,$$
by (\ref{ballcond}).

We now treat the case in which $t\in [2,\mnot]$.  We claim that the event $\T(u,v,x,\ell,k,t)$ implies that there exist disjoint paths $\eta_2, \gamma$ between $u$ and $v$ such that $|\eta_2|=t$ and $|\gamma|\leq t$ and the intersection of the following events occurs:
\begin{enumerate}
\item[(a)] $\eta_2$ is $p_1$-open,
\item[(b)] $\gamma$ is $p_2$-open, and one of its edges receives value in $[p_1,p_2]$,
\item[(c)] $0 \stackrel{=k}{\lrfill} u$ off $\eta_2 \cup \gamma$ and $v \stackrel{=\ell-k-t}{\lrfill} x$ off $\eta_2 \cup \gamma \cup B_0(k)$ in $G_{p_1}$.
\end{enumerate}
Indeed, let $\eta_1,\eta_2, \eta_3, \gamma$ be the disjoint paths guaranteed to exists in the definition of $\T(u,v,x,\ell,k,t)$. The paths $\eta_2$ and $\gamma$ show that both (a) and (b) indeed occur (note that we have relaxed the requirement that $\eta_2$ is a shortest $p_1$-open path). Seeing that (c) occurs is more subtle. First observe that for any two vertices $z,y$ and integer $\ell \geq 0$
    $$ \{ z \stackrel{=\ell}{\lrfill} y \off A \} = \bigcup _{\beta : |\beta|=\ell, \beta \cap A = \emptyset} \Big ( \big \{ \beta \hbox{ {\rm is open}}\big \} \bigcap _{\beta' : |\beta'| < \ell, \beta' \cap A = \emptyset} \big \{ \beta' \hbox{ {\rm has a closed edge}}\big \} \Big ) \, ,$$
where $\beta, \beta'$ are simple paths in $G$ and we slightly abuse notation and write $\beta \cap A = \emptyset$
to denote that the edges of $\beta$ are disjoint from the edges touching $A$. To see that (c) holds we note that the event $\T(u,v,x,\ell,k,t)$ implies that $\eta_3$ is of length $k$ between $0$ and $u$, disjoint from $\eta_2 \cup \gamma$, is $p_1$-open and any shorter path between $0$ and $u$ has a $p_1$-closed edge in it; in particular, $0 \stackrel{=k}{\lrfill} u$ off $\eta_2 \cup \gamma$ occurs in $G_{p_1}$. Similarly, $\eta_3$ is of length $\ell-k-t$ between $v$ and $x$, is disjoint from $\eta_2 \cup \gamma \cup B_0(k)$, is $p_1$-open and any shorter path between $v$ and $x$ has a $p_1$-closed edge in it; in particular $v \stackrel{=\ell-k-t}{\lrfill} x$ off $\eta_2 \cup \gamma \cup B_0(k)$ occurs in $G_{p_1}$.

Now, the events (a), (b), (c) are independent since they are measurable with respect to disjoint sets of edges (the edges of $\eta_2$, $\gamma$ and all the rest). The probability of their intersection is hence
    $$ p_1 ^{|\eta_2|} p_2^{|\gamma|} \big [1- (p_1/p_2)^{|\gamma|} \big] \prob_{p_1} \big (0 \stackrel{=k}{\lrfill} u \off \eta_2 \cup \gamma \and v \stackrel{=\ell-k-t}{\lrfill} x \off \eta_2 \cup \gamma \cup B_0(k) \big ) \, ,$$
where the factor $[1- (p_1/p_2)^{|\gamma|}]$ is the probability that one edge of $\gamma$ has value in $[p_1,p_2]$ conditioned on all edges being $p_2$-open. We compute the probability on the right hand side as usual by conditioning on $B_0(k)$, this gives
    $$ p_1 ^{|\eta_2|} p_2^{|\gamma|} \big [1- (p_1/p_2)^{|\gamma|} \big] \sum_{A : 0 \stackrel{=k}{\lrfill} u} \prob_{p_1} ( B_0(k) = A) \prob_{p_1}(v \stackrel{=\ell-k-t}{\lrfill} x \off A \cup \eta_2 \cup \gamma ) \, .$$
We now start summing all this over $u,v,x,\ell,k,t, \eta_2, \gamma$. We start by summing over $x$ the last probability, giving as a constant factor by Theorem \ref{bdry}. The sum over $A$ gives a term of $\prob_{p_1}(0 \stackrel{=k}{\lrfill} u \off \eta_2 \cup \gamma)$ which we sum over $k\in [0,r]$ and bound this by $\prob_{p_1}(0 \lrr u)$. Furthermore, the number of possible $\eta_2$'s is at most $m(m-1)^{t} \p^t(u,v)$ and if $|\gamma|=s \leq t$, then the number of such $\gamma$'s is at most $m(m-1)^{s-1} \p^s(u,v)$. We also bound $[1-(p_1/p_2)^s] \leq C s \eps$. All this gives that
    $$ \sum_{u,v,x,\ell,k,t\in[2,\mnot]} \prob(\T(u,v,x,\ell,k,t)) \leq C \eps \sum_{\substack{u,v,\ell \\ t \in [2,\mnot], s \in [1,t]}} (m-1)^{s+t} p_1^t p_2^s s \p^t(u,v) \p^s(u,v) \prob_{p_1}(0 \lrr u) \, .$$
By condition (2) of Theorem \ref{mainthmgeneral} and the fact that $\mnot = o(\eps^{-1})$, we have that $(m-1)^{s+t} p_1^t p_2^s = 1+o(\alpham)$, so we may bound this sum by
    $$ C \eps \sum_{u,\ell} \prob_{p_1}(0 \lrr u) \sum_{v, t \in [2,\mnot], s \in [1,t]} s \p^t(u,v) \p^s(u,v) \, .$$
The sum over $\ell\in [r-\eps^{-1},r]$ gives a factor of $\eps^{-1}$, and since $G$ is transitive, the second sum over $v,t,s$ does not depend on $u$. Hence we may sum over $u$ separately using Theorem \ref{KozmaNachmias} giving a bound of
$$ C r \sum_{v,t\in[2,\mnot], s\in [1,t]} s \p^t(u,v) \p^s(u,v) \, .$$
For each $s\geq 1$ and $s_1\in \{1, \ldots, s\}$, we can bound
    $$
    \p^s(0,v)\leq \frac{\m}{\m-1} \sum_{w} \p^{s_1}(0,w)\p^{s-s_1}(w,v) \, ,
    $$
because the number of non-backtracking paths of length $s$ from $0$ to $v$ is at most the sum over $w$ the number of non-backtracking paths of length $s_1$ from $0$ to $w$ times the number of non-backtracking paths of length $s-s_1$ from $w$ to $v$ (the factor $m/(m-1)$ comes from properly normalizing these numbers). As a result,
    $$
    \sum_{v,t \in[2,\mnot], s\in [1,t]} s \p^t(0,v) \p^s(0,v)
    \leq  \frac{\m}{\m-1}
    \sum_{v,w}\sum_{t \in[2,\mnot], s_1\in [1,t], s_2\leq s_1} \p^t(0,v) \p^{s_1}(0,w)\p^{s_2}(w,v)
    \leq \frac{C\alpham}{\log{V}},
    $$
by condition (3) in Theorem \ref{mainthmgeneral} and the fact that $t+s_1+s_2\geq 3$. All together we get that
$$ \sum_{u,v,x,\ell,k,t\in[2,\mnot]} \prob(\T(u,v,x,\ell,k,t)) \leq {Cr \alpham \over \log V} = o(\eps^{-1}) \, ,$$
by our assumption (\ref{lowerball.rcond}) and since $\alpham=o(1)$. This finishes the proof of (\ref{lowerball.toshow}) and concludes the proof of the theorem.
\end{proof}

The following are easy corollaries:

\begin{corollary} \label{r0ball} Let $G$ be a graph satisfying the assumptions in Theorem \ref{mainthmgeneral} and consider percolation on it with $p = p_c(1+\eps)$. Then for any $r$ satisfying
$$r \leq \eps^{-1} \big [ \log(\eps^3 V) - 4 \log\log (\eps^3 V) \big ] \, ,$$
the following bound holds
    $$
    \E |B(r)| = \Theta(\eps^{-1} (1+\eps)^r) \, .
    $$
In particular for $r_0$ defined in (\ref{rzerocond})
    $$ \E |B(r_0)| = \Theta( \sqrt{\alpham \eps V}) \, .$$
\end{corollary}
\begin{proof}
The upper bound follows from Lemma \ref{upperball} and the lower bound from Theorem \ref{lowerball}.
\end{proof}

\begin{lemma}
\label{upperballsqaured}
Let $G$ be a graph satisfying the assumptions in Theorem \ref{mainthmgeneral}
and consider percolation on it with $p\leq p_c(1+\eps)$. Let $r$ be an integer
satisfying the assumptions of Theorem \ref{lowerball}. Then,
    $$
    \E |B(r)|^2 \leq C\eps^{-1} (\E |B(r)|)^2 \, .
    $$
\end{lemma}
\begin{proof} If $0 \lrr x$ and $0 \lrr y$, then there exists a vertex $z$ and an integer $t \leq r$ such that the event
    $$
    \{0 \stackrel{=t}{\lrfill} z\} \circ \{z \stackrel{r-t}{\lrfill} x\} \circ \{z \stackrel{r-t}{\lrfill} y\} \, .
    $$
Apply BK-Reimer and sum over $x,y$ and then $z$ to bound
    $$
    \E |B(r)|^2 \leq \sum_{t=1}^{r} \E | \partial B(t)| \E |B(r-t)| \E |B(r-t)| \, .
    $$
We apply Lemma \ref{upperball} and Theorem \ref{bdry} to bound
    $$
    \E |B(r)|^2 \leq C \eps^{-3} (1+\eps)^{2r} \, ,
    $$
and Theorem \ref{lowerball} gives the required claim.
\end{proof}

\section{An intrinsic-metric regularity theorem}
\label{sec-intrin-reg}

%Given vertices $x,y,a,b,u,u'$ and integers $r_0,j_x,j_y \geq 2\mnot$, we define the event
%$$ \{ a \lrefmpx u \} := \big \{ a \lrefm u \off B_x(j_x) \cup B_y(j_y)\big\} \cap \big \{a \hbox{ {\rm is pivotal %for $x \stackrel{j_x+r_0}{\lrfill} u$}} \big \} \, ,$$
%We similarly define
%$$ \{b \lrefmpy u'\} := \big \{ b \lrefm u' \off B_x(j_x) \cup B_y(j_y)\big \} \cap \big \{b \hbox{ {\rm is pivotal for $y \stackrel{j_y+r_0}{\lrfill} u'$}} \big \} \, .$$
%Furthermore, define
For an increasing event $E$ and a vertex $a$, we say that $a$ is \emph{pivotal for} $E$
for the event that $E$ occurs but does not occur in the modified configuration in which we
close all the edges touching $a$. We write $\Piv(E)$ for the set of pivotal vertices for
the event $E$. For vertices $a,x$, radii $r_1,j_x$ and $A \subset V$, we define
    \eqan{
    G_{r_1,j_x}(a, x; A)
    = \E\big[ |\{u\colon
    a \stackrel{P[2\mnot, r_1]}{\lrfill} u \off A\setminus \{a\}
    \and a \in \Piv(\{x \stackrel{j_x+r_1}{\lrfill} u\})\} |
    \,\,  \big | \,\, B_x(j_x)=A\big] \, .\nonumber
    }
%where the event ``$a\in \Piv(\{x \stackrel{j_x+r_0}{\lrfill} u\})$'' means that
%any open path of length at most $j_x+r_0$ between $x$ and $u$ must visit the
%vertex $a$.\\

\bde[Regenerative and fit vertices] (a) Given vertices $x,a$, radii
$r_1, j_x \geq \mnot$ and a real number $\beta>0$, we say that $a$ is
{\em $(\beta,j_x,r_1)$-regenerative} if
    \begin{enumerate}
     \item $x \lrejx a$, and
     \item $G_{r_1, j_x}\big(a,x;B_x(j_x)\big)\geq (1-\beta) \E|B(r_1)|$,
    \end{enumerate}
and note that this event is determined by the status of the edges touching
$B_x(j_x)$. We say that $a$ is {\em $(\beta,j_x,r_1)$-nonregenerative} if
$x \lrejx a$ but it is not $(\beta,j_x,r_1)$-regenerative.

%We write $N((x,y), \beta, j, r_0)$ for the random variable
%$$ N_{(x,y)}(\beta, j, r_0) = \big | \big \{ a : a \hbox{ {\rm is $(\beta, j, r_0)$-regenerative with respect to $(x,y)$}} \big \} \big | \, ,$$
%and $N_{(y,x)}(\beta, j, r_0)$ is defined similarly.

\noindent
(b) Given an additional real number $\delta>0$, we say that $x$ is
{\em $(\delta,\beta,j_x,r_1)$-fit} if
\begin{enumerate}
 \item $\partial B_x(j_x) \neq \emptyset$ holds and,
 \item the number of $(\beta,j_x,r_1)$-nonregenerative vertices is at
 most $\delta \eps^{-1}$.
\end{enumerate}
%and we say that $x,y$ are {\em $(\delta,\beta,j_x,j_y,r_0)$-unfit} if $\A(x,y,j_x,j_y)$
%holds and they are not $(\delta,\beta,j_x,j_y,r_0)$-fit.
\ede
%Let us define some parameters. Let $M$ be a large constant, depending only on $\alpha$ [*** Say what $\alpha$ is!***] and we put
%$$ r = M\eps^{-1} \, .$$
%Furthermore we choose $r_0$ with the (\ref{rzerocond}) properties
%These conditions will come up in various parts of the proof.
%Note that the event that $x,y$ are $(\delta,\beta,j_x,j_y,r_0)$-fit is an event determined by the balls $B_x(j_x)$ and $B_y(j_y)$.

\noindent It will also be convenient to combine our error terms. For this, we define
    \eqn{
    \label{omegam-def}
    \omegam = \alpham^{1/2} + \eps \mnot \, ,
    }
so that $\omegam = o(1)$. Our goal in this section is to prove that if $\partial B_x(j_x) \neq \emptyset$, then $x$ is fit with high probability. This is the {\em intrinsic metric} regularity theorem discussed in Section \ref{sec-sketch}.

\begin{theorem}[Intrinsic regularity]
\label{mostsetsaregood} Let $G$ be a graph satisfying the assumptions of Theorem \ref{mainthmgeneral}.
Let $p=p_c(1+\eps)$, let $r=\rmm=M/\eps$ where
$M=\Mm$ is defined in (\ref{Mchoice}) and $r_1\in [\eps^{-1}, r_0]$, where
$r_0$ is defined in (\ref{rzerocond}). For any $\delta, \beta \in (0,1)$ there
exist at least $(1-O(\omegam^{1/4}))r$ radii $j_x \in [r,2r]$ such that
    $$
    \prob( x \hbox{ {\rm is} } (\delta,\beta,j_x,r_1)\hbox{{\rm -fit}} )
    \geq \big ( 1- O\big(\delta^{-1} \beta^{-2} \e^{2M}\omegam^{1/4}\big) \big) \prob(\partial B_x(j_x) \neq \emptyset) \, .
    $$
\end{theorem}

We start by proving some preparatory lemmas:

\begin{lemma}
\label{lowerbound} Assume the setting of Theorem \ref{mostsetsaregood}. Then,
    \eqan{
    \sum_{j_x=r}^{2r} \sum_{a,u}
    \prob\big (x \lrejx a,\,
    a \stackrel{P[2\mnot, r_1]}{\lrfill} u \off  B_x(j_x)\setminus \{a\}\big )
    \geq \big (1- O ( \omegam) \big )
    \E|B([r,2r])|\E|B(r_1)| \, .\nonumber
    }
    \end{lemma}

\begin{proof}
We condition on $B_x(j_x)=A$ for any admissible $A$ (that is, $A$ for which the event
$x \lrejx a$ occurs and $\prob(B_x(j_x)=A) > 0$). Then,
    $$
    \prob(a \stackrel{P[2\mnot, r_1]}{\lrfill} u \off B_x(j_x)\setminus \{a\}\mid B_x(j_x)=A)
    =\prob(a \stackrel{P[2\mnot, r_1]}{\lrfill} u \off A\setminus \{a\}) \, ,
    $$
and
    $$
    \prob(a \stackrel{P[2\mnot, r_1]}{\lrfill}  u \off A\setminus \{a\})=
    \prob(a \stackrel{P[2\mnot, r_1]}{\lrfill} u) - \prob(a \stackrel{P[2\mnot, r_1]}{\lrfill} u \onlyon A\setminus \{a\}) \, .
    $$
Summing over the first term gives
    \eqan{
    \sum_{a,u,j_x\in [r,2r]}
    \prob(x \lrejx a)
    \prob(a \stackrel{P[2\mnot, r_1]}{\lrfill} u)
    &\geq
    \E|B([r,2r])| \big(\E|B(r_1)|-\E|B(2\mnot)|\big) \nonumber\\
    &=(1- O(\omegam))\E|B(r_1)| \E|B([r,2r])|\, ,
    \nonumber
    }
since $\E|B(2\mnot)|\leq C \mnot$ by Corollary \ref{corrlength}
and $\E|B(r_1)|\geq c\eps^{-1}$ by Theorem \ref{lowerball} and since $r_1 \geq \eps^{-1}$. It remains to bound the sum
    $$
    \sum_{a,u,j_x \in [r,2r]} \sum_{A} \prob(B_x(j_x)=A)
    \prob(a \stackrel{P[2\mnot, r_1]}{\lrfill} u \onlyon A\setminus \{a\}) \, .
    $$
As usual, if $a \stackrel{P[2\mnot, r_1]}{\lrfill} u \onlyon A\setminus \{a\}$ occurs, then there exists $z\in A$ such that
$\{ a \stackrel{r_1}{\lrfill} z \} \circ \{z \stackrel{r_1}{\lrfill} u\}$. BK inequality now gives
    \be\label{lowerbound.midsum}
    \sum_{a,u, j_x\in[r,2r]}\sum_{A} \prob(B_x(j_x)=A)
    \sum_{z \in A \setminus \{a\}} \prob(a \stackrel{r_1}{\lrfill} z) \prob(z \stackrel{r_1}{\lrfill} u) \, .
    \ee
We sum over $u$ and extract a factor of $\E|B(r_1)|$. We then change the
order of summation, so the sum simplifies to	
    $$
    \E|B(r_1)| \sum_{a,z\neq a, j_x\in[r,2r]}
    \prob(x \stackrel{=j_x}{\lrfill} a\, , x \stackrel{j_x}{\lrfill} z)
    \prob(a \stackrel{r_1}{\lrfill} z) \, .
    $$
We sum over $j_x$ (noting that the events $x \stackrel{=j_x}{\lrfill} a\, , x \stackrel{j_x}{\lrfill} z$ are disjoint as $j_x$ varies) and bound this sum by
    $$
    \E|B(r_1)| \sum_{a,z\neq a} \prob(x \stackrel{2r}{\lrfill} a, x \stackrel{2r}{\lrfill} z)
    \prob(a \stackrel{r_1}{\lrfill} z) \, .
    $$
As usual, if $x \stackrel{2r}{\lrfill} a$ and $x \stackrel{2r}{\lrfill} z$,
then there exists $z'$ such that the event
    $$
    \{ x \stackrel{2r}{\lrfill} z' \} \circ \{ z' \stackrel{2r}{\lrfill} z\}
    \circ \{ z' \stackrel{2r}{\lrfill} a\} \, ,
    $$
occurs. By the BK inequality we bound the above sum by
    $$
    \E|B(r_1)| \sum_{a,z', z \neq a} \prob(x \stackrel{2r}{\lrfill} z')
    \prob(z' \stackrel{2r}{\lrfill} z) \prob (z' \stackrel{2r}{\lrfill} a) \prob(a \stackrel{r_1}{\lrfill} z) \, .
    $$
We may now sum over $a$ and $z\neq a$ using Corollary \ref{extendedtriangle} and
then sum over $z'$ to get that this is bounded by
    $$
    C\E|B(r_1)| \E|B(2r)| \Big [ \omegam +
    {\E|B(r_1)| (\E|B(2r)|)^2 \over V} \Big ] \, .
    $$
This concludes our proof since the second term in the parenthesis is of
order at most $\alpham^{1/2} \e^{4M} (\eps^3 V)^{-1/2} \leq \alpham^{1/2}$ by the upper bound on $r_1$, our choice of $r$ and $M$ in \eqref{Mchoice} and Corollary \ref{r0ball}.
 %When applying the result, it is useful
%to check that the error term is $o(1)$ for the choice $M=\Mm, r=\rmm, r_1=r_0$.
%In this case, we bound $\E|B(2\rmm)| \leq C\Mm\e^{2\Mm} \eps^{-1}$
%by Corollary \ref{corrlength} and now
%our choice of $\rmm,r_0,\Mm$ in (\ref{Mchoice}) and (\ref{rzerocond}) gives that
%    $$
%    \Big [ {C \over m} +  {\E|B(r_0)| (\E|B(2r)|)^2 \over V} \Big ] = o(1) \, ,
%    $$
%so the whole sum is $o(\E|B(r_0)| \E|B(2r)|)$ and Lemma \ref{balllowerbound}
%shows that this of lower order than the right hand side of the assertion of
%the lemma.
\end{proof}

\begin{lemma}
\label{lowerboundfinal} Assume the setting of Theorem \ref{mostsetsaregood}. There exists a $C>0$,
such that
    \eqan{
    \sum_{j_x=r}^{2r} \sum_{a,u} \prob\big (x \lrejx a,\, a \stackrel{P[2\mnot, r_1]}{\lrfill} u \off B_x(j_x)
    \setminus \{a\}, a \in \Piv(\{x \stackrel{j_x+r_1}{\lrfill} u\})\big)
    \geq  \big (1- O ( \omegam) \big )
    \E|B([r,2r])|\E|B(r_1)| \, . \nonumber
    }
\end{lemma}

\begin{proof} Fix $j_x \in [r,2r]$. We rely on Lemma \ref{lowerbound}, and bound the difference
in the probabilities appearing in Lemma \ref{lowerbound} and the one above. If the event
    \be
    \label{lowerboundfinal.goodevent}
    \{x \lrejx a,\, a \stackrel{P[2\mnot, r_1]}{\lrfill} u \off B_x(j_x)
    \setminus \{a\}\}
    \ee
occurs but $a \not\in \Piv(\{x \stackrel{j_x+r_1}{\lrfill} u\})$, then there
exist $z_1,z_2$ and $t \leq j_x$ and paths $\eta_1, \eta_2, \gamma_1, \gamma_2, \gamma_3$ such that
\begin{enumerate}
\item[(a)] $\gamma_1$ is an open path of length at most $r_1$ connecting $a$ to $z_2$,
\item[(b)] $\gamma_2$ is an open path of length at most $r_1$ connecting $z_2$ to $u$,
\item[(c)] $\gamma_3$ is an open path of length at most $r_1$ connecting $z_1$ to $z_2$,
\item[(d)] $\eta_1$ is a shortest open path of length precisely $t$ connecting $x$ to $z_1$,
\item[(e)] $\eta_2$ is a shortest open path of length precisely $j_x-t$ connecting $z_1$ to $a$,
\item[(f)] $\gamma_1, \gamma_2, \gamma_3, \eta_1, \eta_2$ are disjoint.
\end{enumerate}
%
%such that the event
%    $$
%    \{ x \stackrel{=t}{\lrfill} z_1\} \circ \{z_1 \stackrel{=j_x-t}{\lrfill} a \off B_x(t)\setminus \{a\}\}
%    \circ \{a \stackrel{r_1}{\lrfill} z_2 \} \circ \{z_1 \stackrel{r_1}{\lrfill}  z_2\} \circ \{z_2 \stackrel{r_1}{\lrfill}  u\} \, ,
%    $$
%occurs.
\begin{figure}
\begin{center}
\includegraphics{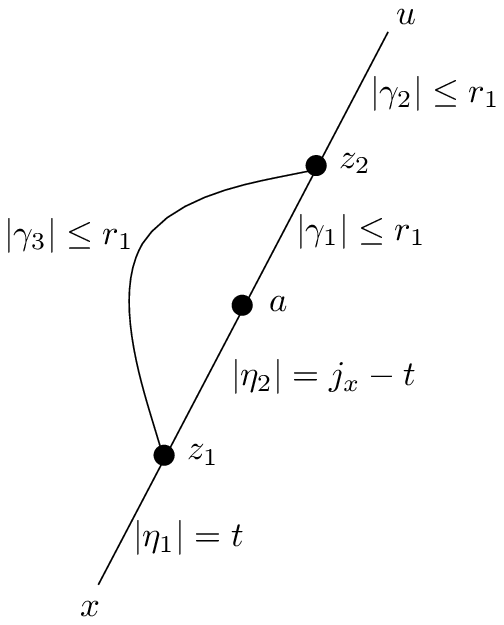}
\caption{$a$ is not pivotal for the event $x \stackrel{j_x+r_1}{\lrfill u}$.}
\label{fig.notpivotal}
\end{center}
\end{figure}
See Figure \ref{fig.notpivotal}. Indeed, assume that $a$ is not pivotal for $x \stackrel{j_x+r_1}{\lrfill} u$
and (\ref{lowerboundfinal.goodevent}) holds. Let $\eta$ be the first
lexicographical shortest open path of length $j_x$ between $x$ and $a$ and
$\gamma$ a disjoint open path of length in $[2\mnot,r_1]$ between $a$ and $u$
off $B_x(j_x)\setminus \{a\}$ which we are guaranteed to have since (\ref{lowerboundfinal.goodevent})
holds. Since $a$ is not pivotal we learn that there exists another open path $\beta$ between
$x$ and $u$ of length at most $j_x+r_1$ that does {\em not} visit $a$.
Hence, $\beta$ goes ``around'' $a$, or in formal words, there exists vertices
$z_1$ and $z_2$ on $\beta$ appearing on it in that order such that
$z_1 \in \eta$ and $z_2\in \gamma$ and the part of $\eta$ between $z_1$ and $z_2$ is disjoint from $\eta \cup \gamma$. We take $t < j_x$ to be such that $\eta(t)=z_1$ and put $\eta_1=\eta[0,t], \eta_2=\eta[t,j_x]$. We take $\gamma_3$ to be that section of $\beta$ between $z_1$ and $z_2$ and $\gamma_1, \gamma_2$ be the sections of $\gamma$ from $a$ to $z_2$ and from $z_2$ to $u$, respectively.

%Our disjoint witnesses are the following. For $\{ x \stackrel{=t}{\lrfill} z_1\}$,
%we take the open edges of $\gamma_1[0,t]$ together with the closed edges touching
%$B_x(t-1)$.
%For $\{z_1 \stackrel{=j_x-t}{\lrfill} a \off B_x(t)\setminus \{a\}\}$, we take
%the open edges of $\gamma_1[t,j_x]$ together with the closed edges touching
%$B_a(j_x-t-1)$ but not touching $B_x(t-1)$.
%The witnesses for the last three events are simply the part of $\eta$
%between $z_1$ and $z_2$, the part of $\gamma_2$ between $a$ and $z_2$,
%and the part of $\gamma_2$ between $z_2$ and $u$, respectively.

For all $j_x,t \in [r,2r]$ these events (that is, the existence of $z_1,z_2$ and the disjoint paths) are disjoint since $\eta_1$ and $\eta_2$ are required to be shortest open paths. The union of these events over $j_x,t$ implies that there exists $z_1,z_2$ such that
    $$ \{x \stackrel{2r}{\lrfill} z_1\} \circ \{z_1 \stackrel{2r}{\lrfill} a\} \circ \{a \stackrel{r_1}{\lrfill} z_2\} \circ
    \{z_1 \stackrel{r_1}{\lrfill} z_2\} \circ \{z_2 \stackrel{r_1}{\lrfill} u\} \, ,$$
since we can just take $\eta_1,\eta_2,\gamma_1,\gamma_2,\gamma_3$ as our disjoint witnesses. Using BK inequality we bound the required sum from above by
    $$ \sum_{\substack{a,u, z_1,z_2 \colon \\ z_1 \neq z_2, z_1 \neq a, z_2 \neq a}} \prob(x \stackrel{2r}{\lrfill} z_1)
    \prob(z_1 \stackrel{2r}{\lrfill} a) \prob(a \stackrel{r_1}{\lrfill} z_2)
    \prob(z_1 \stackrel{r_1}{\lrfill} z_2) \prob(z_2 \stackrel{r_1}{\lrfill} u) \, .
    $$
Summing first over $u$ extracts a factor of $\E|B(r_1)|$, and we sum over $a$ and $z_2$ using Corollary \ref{extendedtriangle}, and lastly sum over $z_1$. This gives a bound of
$$ C \E|B(r_1)| \E|B(2r)| \Big [ \omegam + {(\E|B(r_1)|)^2\E|B(2r)| \over V} \Big ] \, .$$
We apply Lemma \ref{lowerbound} to conclude the proof since the second term in the
parenthesis is of order at most $\alpham \e^{2M} \leq \alpham^{1/2}$ by the upper bound on $r_1$, our choice of $r$ and $M$ in \eqref{Mchoice} and Corollary \ref{r0ball}. Also note that $\E |B([r,2r)| \geq \E|B(2r)|/2$ by Corollary \ref{r0ball} and our choice of $r$ and $M$.
\end{proof}

\begin{lemma}
\label{converse} Assume the setting of Theorem \ref{mostsetsaregood}.
For any vertices $x,a$,
    \eqan{
    \sum_{j_x=r}^{2r} \sum_{u}
    \prob\big (x \lrejx a,\,
    a \stackrel{P[2\mnot, r_1]}{\lrfill} u \off  B_x(j_x)\setminus \{a\}\big )
    \leq (1 + O ( \omegam) ) V^{-1}
    \E|B(r_1)|\E|B([r,2r])| \, .\nonumber
    }
\end{lemma}

\begin{proof} The event $x \lrejx a,\, a \stackrel{P[2\mnot, r_1]}{\lrfill} u \off  B_x(j_x)\setminus \{a\}$ implies that
$$ \{ x \lrejx a \} \circ \{ a \stackrel{r_1}{\lrfill} u\} \, ,$$
the second witness is the open edges of an open path of length in $[2\mnot, r_1]$ off $B_x(j_x) \setminus \{a\}$ and the first witness is the lexicographically first shortest open path of length  $j_x$ between $x$ and $a$ together with all the closed edges of the graph. The BK-Reimer inequality gives that
    $$
    \prob\big (x \lrejx a,\, a \stackrel{P[2\mnot, r_1]}{\lrfill} u \off  B_x(j_x)\setminus \{a\}\big )
    \leq \prob(x \lrejx a) \prob(a \stackrel{r_1}{\lrfill} u) \, .
    $$
We sum over $u$ and $j_x\in[r,2r]$ get that the sum is bounded by
    $$ \E|B(r_1)| \prob(x \lrer a) \, .
    $$
Lemma \ref{unifconnbd2} gives that
    $$
    \prob(x \lrer a ) \leq { (1 + O (\omegam)  ) V^{-1} \E|B([r-\mnot,2r-\mnot])|} \, .
    $$
We have that
    \eqan{
    \E |B([r-\mnot,2r-\mnot])| &\leq \E|B([r,2r])|+\E|B([r-\mnot,r])|\leq (1 + O(\eps \mnot)) \E|B([r,2r])| \,\nonumber
    }
since $\E|B([r-\mnot,r])| \leq C\mnot (1+\eps)^r$ by Theorem \ref{bdry} and
Corollary \ref{corrlength} and since $\E|B([r,2r]) \geq c\eps^{-1}(1+\eps)^{2r}$ by Theorem \ref{lowerball} and Lemma \ref{upperball} (we use the assumption that $r \gg \eps^{-1}$). Hence
    \eqn{
    \label{unif-a-bd}
    \prob(x \lrer a ) \leq { (1+ O(\omegam) ) V^{-1} \E|B([r,2r])|} \, ,
    }
concluding our proof.
\end{proof}

%\noindent
%Now we are ready to complete the proof of Theorem \ref{mostsetsaregood}:\\

%First,
%denote by
%    \eqn{
%    \Errm=\frac{C}{\m}+{\E|B(r_1)| \E|B(2r)|(\E|B(r_1)|+\E|B(2r)|)\over V}+C \eps \mnot \e^M
%    }
%the sum of the error terms appearing in Lemmas \ref{lowerboundfinal} and \ref{converse}.
%We note that, whenever $p=p_c(1+\eps)$, let $M\leq \Mm$ and $r=M/\eps$, where
%$\Mm$ is defined in (\ref{Mchoice}) and $r_1\leq r_0$, where
%$r_0$ satisfies the requirements in (\ref{rzerocond}), we can bound the terms in
%$\Errm$ from above using
%Lemma \ref{upperball} and (\ref{rzerocond}) by
%    \eqan{
%    \E|B(r_1)| \E|B(2r)|(\E|B(r_1)|+\E|B(2r)|)
%    &\leq C\eps^{-1} (\eps V) (\log(\eps^3V))^{-2} \e^{4\Mm}\nonumber\\
%    &\leq C (\log\log(\eps^3V))^4/(\log(\eps^3V))^{2},\nonumber
%    }
%so that $\Errm\leq \errm$. We shall prove the claim with $\Errm$ instead of $\errm$.

\noindent{\bf Proof of Theorem \ref{mostsetsaregood}.} By combining Lemmas \ref{lowerboundfinal} and \ref{converse} we deduce that for any $x$ there exist at least $(1- O(\omegam^{1/2}))V$ vertices $a$ such that
    \eqan{
    \sum_{j_x=r}^{2r} \sum_{u} \prob\big (x \lrejx a,\,
    a \stackrel{P[2\mnot, r_1]}{\lrfill} u \off  B_x(j_x)\setminus \{a\}, a \in \Piv(\{x \stackrel{j_x+r_1}{\lrfill} u\})\big )
    = (1+ O(\omegam^{1/2}) ) V^{-1} \E|B(r_1)| \E|B([r,2r])|\, .\nonumber
    }
%Indeed, suppose the above is \emph{not} true. Then, the sum of
%the left hand side over all $a$ is at most
%    \eqan{
%    &(1-2\sqrt{\Errm})(1+\Errm)+2\sqrt{\Errm}(1-2\sqrt{\Errm})\nonumber\\
%    &=(1-2\sqrt{\Errm})(1+2\sqrt{\Errm}+\Errm)\leq 1-3\Errm
%    <1-\Errm\nonumber
%    }
%times $\E|B(r_1)| \E|B([r,2r])|$, which contradicts
%the statement in Lemma \ref{lowerboundfinal}.

Write $\widetilde{G}$ for the variable
    $$
    \widetilde{G} = \sum _{j_x=r}^{2r} G_{j_x, r_1}(a, x; B_x(j_x)) {\bf 1}_{\{x \lrejx a\}} \, .
    $$
Note that $\widetilde{G}$ is a random variable that is measurable with
respect to $B_x(2r)$ (that is, it is determined by the status of the edges touching $B_x(2r)$) and that it equals $0$ unless $x \lrer a$. Furthermore,
only one of the summands can be nonzero because the events
in the indicators are disjoint. Our previous inequality can be rewritten as
    $$
    \E \widetilde{G} = (1+ O(\omegam^{1/2}) ) V^{-1} \E|B(r_1)| \E|B([r,2r])| \, .
    $$
Hence, for at least $(1-O(\omegam^{1/2}))V$ vertices $a$,
    \be\label{mostsets.firstmom}
    \E \big [ \widetilde{G} \, \big | \, x \lrer a \big ] \geq (1-O(\omegam^{1/2}))\E|B(r_1)| \, ,
    \ee
by Lemma \ref{unifconnbd2}. This gives the conditional first moment estimate. The second moment calculation is somewhat easier. We have
    \eqan{
    \E \widetilde{G}^2 = \sum_{j_x=r}^{2r} \sum_{u_1,u_2}
    \E \Big [ \prob\big(a \stackrel{P[2\mnot, r_1]}{\lrfill} u_1 \off  B_x(j_x)\setminus \{a\}\mid B_x(j_x)\big)
    \prob\big(a \stackrel{P[2\mnot, r_1]}{\lrfill} u_2 \off  B_x(j_x)\setminus \{a\}\mid B_x(j_x)\big)
    {\bf 1}_{\{x \lrejx a\}}\Big ]\, .
    \nonumber
    }
We bound, almost surely in $B_x(j_x)$ and for $i=1,2$,
    $$
    \prob(a \stackrel{P[2\mnot, r_1]}{\lrfill} u_i \off  B_x(j_x)\setminus \{a\}\mid B_x(j_x))
    \leq
    \prob(a \stackrel{r_1}{\lrfill} u_i),
    $$
and sum over $u_1$ and $u_2$ to get that
    \begin{eqnarray*}
    \E \widetilde{G}^2 {\bf 1}_{\{x \lrer a \}} \leq \big [\E|B(r_1)|\big ]^2
    \prob(x \lrer a) \, ,
    \end{eqnarray*}
so that
    $$
    \E \big [\widetilde{G}^2 \, \big | \, x \lrer a \big ] \leq \big [\E|B(r_1)| \big ]^2\, .
    $$
Combining this with (\ref{mostsets.firstmom}), we obtain
    $$
    {\rm Var}\big(\widetilde{G}\big | \, x \lrer a\big) = O\big ([\E|B(r_1)|]^2 \omegam^{1/2}  \big ) \, .$$
By Chebychev's inequality, for any $\beta>0$,
    $$
    \prob\big ( \widetilde{G} \leq (1-\beta) \E|B(r_1)| \, \big | \,  x \lrer a \big)
    = O( \beta^{-2}\omegam^{1/2})\, .
    $$
Recall that this holds for at least $(1-O(\omegam^{1/2}))V$ vertices $a$.
Call these vertices {\em valid}. We have
    $$
    \sum_{a \hbox{\small{\rm~valid}}}
    \prob\Big (x \lrer a, \widetilde{G} \leq (1-\beta) \E|B(r_1)|\Big ) = O \big ( \E|B([r,2r])| \beta^{-2} \omegam^{1/2} \big ) \, ,
    $$
by our previous estimate. Also, since there are at most $O(\omegam^{1/2} V)$ invalid $a$'s,
we apply \eqref{unif-a-bd} to bound the sum over all $a$ by
    $$
    \sum_{a} \prob\Big (x \lrer a,  \widetilde{G} \leq (1-\beta) \E|B(r_1)|\Big) =
    O \big ( \E|B([r,2r])| \beta^{-2} \omegam^{1/2} \big )\, .
    $$
Returning to our original notation, we rewrite this as
    $$
    \sum_{j_x=r}^{2r}
    \sum_{a} \prob\Big (x \lrejx a, G_{j_x,r_1}(a, x; B_x(j_x))
    \leq (1-\beta) \E|B(r_1)|\Big) = O \big ( \E|B([r,2r])| \beta^{-2} \omegam^{1/2} \big ) \, .
    $$
Hence, there are at least $(1-O(\omegam^{1/4}))r$ radii $j_x \in [r,2r]$ such that
    \begin{eqnarray*}
    \sum_{a} \prob\Big (x \lrejx a, G_{j_x,r_1}(a, x; B_x(j_x))
    \leq (1-\beta) \E|B(r_1)|\Big) &=& O \big ( \E|B([r,2r])| r^{-1} \beta^{-2} \omegam^{1/4} \big ) \\
    &=& O(\e^{2M} \beta^{-2} \omegam^{1/4} )\, ,
    \end{eqnarray*}
where the last inequality is by Lemma \ref{upperball}.
%%% This is just false
%while, for each $j_x\in [r,2r]$,
%    $$
%    \sum_{a} \prob\Big (x \lrejx a, G_{j_x,r_1}(a, x; B_x(j_x))
%    \leq (1-\beta) \E|B(r_1)|\Big) \leq \E|\partial B(j_x)|.
%    $$
%By Lemma \ref{upperball} and Theorem \ref{lowerball}, we can
%bound, for some constant $C>0$ and every $j\in [r,2r]$,
%    $$
%    \E|\partial B(j)|\leq C\E|B([r,2r])|/r,
%    $$
%so that, for each $j_x\in [r,2r]$,
%    $$
%    \sum_{a} \prob\Big (x \lrejx a, G_{j_x,r_1}(a, x; B_x(j_x))
%    \leq (1-\beta) \E|B(r_1)|\Big) \leq C\E|B([r,2r])|/r.
%    $$
%We learn that, for some constant $C>0$ and for at
%least $(1-C\sqrt{\Errm}\beta^{-2})r$ of the $j_x\in [r,2r]$,
%    $$
%    \sum_{a} \prob\big (x \lrejx a,G_{j_x,r_1}(a,x;B_x(j_x))
%    \leq (1-\beta) \E|B(r_1)|\big )\leq 8\sqrt{\Errm}\beta^{-2}\, .
%    $$
Given such $j_x$, write $X(j_x)$ for the random variable
    $$
    X(j_x) = \Big | \Big \{ a \colon x \lrejx a, G_{j_x,r_1}(a,x;  B_x(j_x))
    \leq (1-\beta) \E|B(r_1)| \Big \}\Big | \, ,
    $$
so that $\E X(j_x)\leq C \e^{2M} \beta^{-2} \omegam^{1/4}$. The variable $X(j_x)$ equals the
number of $(\beta,j_x,r_1)$-non-regenerative vertices. By
Markov's inequality we get that for any $\delta>0$
    $$
    \prob(X(j_x) \geq \delta \eps^{-1}) = O\big (\eps \delta^{-1} \beta^{-2} \e^{2M} \omegam^{1/4} \big ) \, ,
    $$
and we conclude by Lemma \ref{inttail} that
at least $(1-O(\omegam^{1/4}))r$ radii $j_x\in[r,2r]$ satisfy
    $$
    \prob(\partial B_x(j_x) \neq \emptyset \and X(j_x) \leq \delta \eps^{-1})
    \geq  \big (1- O(\delta^{-1} \beta^{-2} \e^{2M} \omegam^{1/4}) \big ) \prob(\partial B_x(j_x) \neq \emptyset) \, ,
    $$
as required.
\qed
\medskip

\section{Large clusters are close}
\label{sec-large-clusters-close}

In this section, we prove Theorem \ref{alphapairs} which shows that many
closed edges exist between most large clusters. This section involves
all our notation from the previous sections and in order to ease the
readability we include a legend in terms of $V, \m,\mnot, \epsm, \alpham$
which are given in Theorem \ref{mainthmgeneral} in Table 1
at the end of the paper. We define $\beta, k, \ell, \zeta, \delta$ as
    \eqn{
    \label{parameter-def-1}
    \beta= (\log M)^{-2},  \qquad k = {M \over \log M}, \qquad \ell = (\log M)^{1/4},
    \qquad \zeta=(\log M)^{-1/8},
    \qquad
    \delta=\zeta/2.
    }
%We start by defining some events and parameters.
%We take
%For vertices $x,y$ and radii $j_x, j_y$ we say that $B_x(j_x)$ and $B_y(j_y)$ are %{\em far} if $\A(x,y,j_x,j_y)$ occurs and
%$$ \prob(\partial B_x(j_x) \stackrel{2r_0}{\lrfill} \partial B_y(j_y) \, \mid \, %B_x(j_x), B_y(j_y)) \leq {\eps^2 \over \log \eps^3 V} \, .$$
%    \eqn{
%    \label{parameter-def-1}
%    \beta= (\log M)^{-1},  \qquad k = {M \over \log M}, \qquad \ell = (\log M)^{1/4} \, .
%    }
For notational convenience we also denote
    \begin{eqnarray*} \{ a \lrefmpx u \} &=& \{ a \stackrel{P[2\mnot,r_0]}{\lrfill} u \} \cap \{ a \in \Piv(x \stackrel{j_x+r_0}{\lrfill} u)\} \, ,\\
    \{ b \lrefmpy u' \} &=& \{ b \stackrel{P[2\mnot,r_0]}{\lrfill} u' \} \cap \{ b \in \Piv(y \stackrel{j_y+r_0}{\lrfill} u') \} \, .
    \end{eqnarray*}

Let $S_{j_x,j_y,r_0}(x,y)$ be the random variable counting the number of edges
$(u,u')$ such that there exist vertices $a,b$ with
\begin{enumerate}
\item $\A(x,y,j_x,j_y)$,
\item $x \stackrel{=j_x}{\lrfill} a$ and $y \stackrel{=j_y}{\lrfill} b$ and
\item $a \lrefmpx u$ and
\item $b \lrefmpy u'$ off $B_x(j_x+r_0)$.
%\item $a \in \Piv(x \stackrel{j_x+r_0}{\lrfill} u)$ and $b \in \Piv(y \stackrel{j_y+r_0}{\lrfill} u')$ and
\end{enumerate}

Further define
    \begin{eqnarray*}
    \widehat{S}_{2r,2r,r_0}(x,y)
    = \big | \big \{ (u,u') \colon \{x \stackrel{2r+r_0}{\lrfill} u\}\circ\{y \stackrel{2r+r_0}{\lrfill} u'\} \, ,
    |B_u(2r+r_0)|\cdot |B_{u'}(2r+r_0)| \geq \e^{40M} \eps^{-2} (\E|B(r_0)|)^2 \big \} \big | \, .
    \end{eqnarray*}
We will use the fact that for any $j_x,j_y \in \{r, \ldots, 2r\}^2$
    \be
    \label{usefulfact}
    S_{2r+r_0}(x,y) \geq S_{j_x,j_y,r_0}(x,y) - \widehat{S}_{2r,2r,r_0}(x,y) \, ,
    \ee
where $S_{2r+r_0}(x,y)$ is the random variable defined above
Theorem \ref{alphapairs}. Finally, write $\A(x,y,j_x,j_y, r_0, \beta, k)$
for the intersection of the events
\begin{enumerate}
\item $\A(x,y,j_x,j_y)$,
\item $|B_x(j_x)| \leq \eps^{-2} (1+\eps)^{3r}$ and $|B_y(j_y)|\leq \eps^{-2}(1+\eps)^{3r}$ and $|\partial B_x(j_x)| \leq \eps^{-1} (1+\eps)^{3r}$,
\item $|\partial B_x(j_x)| \geq \e^{k/4} \eps^{-1}$ and $|\partial B_y(j_y)| \geq \e^{k/4} \eps^{-1}$,
\item $x$ is $(1,\beta,j_x,r_0)$-fit and $y$ is $(1,\beta,j_y,r_0)$-fit,
\item
    $$
    \E\Big[ S_{j_x,j_y,r_0}(x,y) {\bf 1}_{\{\partial B_x(j_x) \stackrel{2r_0}{\lrfill} \partial B_y(j_y)\}} \mid B_x(j_x), B_y(j_y)\Big] \leq { V^{-1} \m \eps^{-2} (\E|B(r_0)|)^2} \alpham^{1/2} \, .
    $$
\end{enumerate}
\noindent This event is measurable with respect to $B_x(j_x), B_y(j_y)$.
The following three statements will prove Theorem \ref{alphapairs}:

\begin{lemma}
\label{errorfromlargeu}
Assume the setting of Theorem \ref{alphapairs}. Then,
    $$
    \E \big | \big \{ x,y \colon \A(x,y,2r,2r) \and \widehat{S}_{2r,2r,r_0}(x,y)
    \geq \beta^{1/2} V^{-1} \m \eps^{-2}(\E|B(r_0)|)^2 \big \} \big | = o(\eps^2 V^2) \, .
    $$
\end{lemma}

\begin{theorem}
\label{largeclose.findradii}
Assume the setting of Theorem \ref{alphapairs}. Then there exists
radii $j_1, \ldots, j_\ell \in [r,2r]$ such that for at least $(1-o(1))V^2$ pairs $x,y$,
    $$
    \prob \Big ( \A(x,y,2r,2r) \and \bigcap _{j_x,j_y \in \{j_1,\ldots, j_\ell\}^2}
    \A(x,y,j_x,j_y,r_0,\beta,k)^c \Big ) = o(\eps^2) \, .
    $$
\end{theorem}

%%Let $\{j_1, \ldots, j_\ell\}$ be the radii guaranteed to exists by the previous theorem and
%Write $J(x), J(y)$ for the lexicographically first pair $(j_x,j_y) \in \{j_1,\ldots,j_\ell\}^2$ for which the event $\A(x,y,j_x,j_y, r_0, \beta, k)$ holds. Put $J(x)=J(y)=\infty$ if no such pair exists.

\begin{theorem}
\label{largeclose.getsedge}
Assume the setting of Theorem \ref{alphapairs} and let $x,y$ be a pair of vertices.
Then, for any radii $j_x,j_y \in \{r,\ldots, 2r\}^2$,
    $$
    \prob \big ( S_{j_x,j_y,r_0}(x,y) \leq  2\beta^{1/2} V^{-1} \m \eps^{-2}(\E|B(r_0)|)^2 \and \A(x,y,j_x,j_y,r_0,\beta,k) \big ) = O( \beta^{1/2} \eps^2 ) \, .
    $$
\end{theorem}

\noindent {\bf Proof of Theorem \ref{alphapairs} subject to Lemma \ref{errorfromlargeu}
and Theorems \ref{largeclose.findradii}--\ref{largeclose.getsedge}.}
Lemma \ref{alotofpairs} shows that
    $$
     { | \{ x,y \colon \A(x,y,2r,2r) \} | \over 4 \eps^2 V^2 } \convp 1 \, .
    $$
Thus, it suffices to prove that
    $$ { \big | \big \{ x,y \colon \A(x,y,2r,2r) \and x,y \hbox{ {\rm are not} }
    (r,r_0)\hbox{{\rm -good}} \big \} \big | \over \eps^2 V^2} \convp 0 \, .
    $$
Lemma \ref{morefitnesstwo} shows that
    $$
    { \big | \big \{ x,y : \A(x,y,2r,2r) \and |\C(x)| \leq (\eps^3 V)^{1/4} \eps^{-2}
    \hbox{ {\rm or} } |\C(y)| \leq (\eps^3 V)^{1/4} \eps^{-2} \big \} \big | \over \eps^2 V^2} \convp 0 \, .
    $$
We are left to handle requirement (3) in the definition of $(r,r_0)$-good.
Let $j_1, \ldots, j_\ell$ be the radii guaranteed to exist by Theorem \ref{largeclose.findradii}
and let $x,y$ be a pair of vertices for which the assertion of Theorem \ref{largeclose.findradii} holds.
Theorem \ref{largeclose.findradii} asserts that the number of such pairs is $(1-o(1))V^2$ so the sum of $\prob(\A(x,y,2r,2r))$ over pairs not counted is $o(\eps^2 V^2)$.
Write  $J(x), J(y)$ for the lexicographically first pair $(j_x,j_y) \in \{j_1,\ldots, j_\ell\}^2$
for which the event $\A(x,y,j_x,j_y, r_0, \beta, k)$ occurs, or put $J(x)=J(y)=\infty$ if no such $j_x,j_y$ exist.
Theorem \ref{largeclose.findradii} states that for at least $(1-o(1))V^2$ pairs $x,y$
    $$
    \prob( \A(x,y,2r,2r), J(x) = \infty, J(y) = \infty ) = o(\eps^2) \, .
    $$
%By Lemma \ref{largeclose.donotintersect} we have that
%$$ \prob(\A(x,y,2r,2r), J(x)=j_x, J(y)=j_y, B_x(j_x) \and B_y(j_y) \hbox{ {\rm are %not far} }) \leq {\eps^2 \over \log \eps^3 V} \, .$$
%Since $\ell = o(\log \eps^3 V)$ we get
%$$ \sum_{j_x, j_y \in \{j_1,\ldots,j_\ell\}^2} \prob \big (  \A(x,y,2r,2r,), J(x) = %j_x, J(y) = j_y, B_x(j_x) \and B_y(j_y) \hbox{ {\rm are far}} \big ) = (2+o(1))\eps %\, .$$
Theorem \ref{largeclose.getsedge} together with the union bound implies that for any such pair $x,y$
    $$
    \sum_{j_x,j_y\in \{j_1,\ldots,j_\ell\}^2} \prob \big ( S_{j_x,j_y,r_0}(x,y)
    \leq 2\beta^{1/2} V^{-1} \m \eps^{-2}(\E|B(r_0)|)^2, J(x)=j_x, J(y)=j_y \big ) = O(\beta^{1/2} \ell^2 \eps^2) \, ,
    $$
which is $o(\eps^2)$ by our choice of $\ell$ and $\beta$ in \eqref{parameter-def-1}.
By these last two statements we deduce that
    $$
    \E \big | \big \{ x,y \colon \A(x,y,2r,2r) \and \,\,\, \forall j_x,j_y  \quad  S_{j_x,j_y,r_0}(x,y)
    \leq 2 \beta^{1/2} V^{-1} \m \eps^{-2}(\E|B(r_0)|)^2 \big \} \big | = o(\eps^2 V^2) \, .
    $$
This together with (\ref{usefulfact}) and Lemma \ref{errorfromlargeu} implies that
    $$
    \E \big | \big \{ x,y \colon \A(x,y,2r,2r) \and S_{2r+r_0}(x,y)
    \leq \beta^{1/2} V^{-1} m \eps^{-2} (\E|B(r_0)|)^2 \big \} \big | = o(\eps^2 V^2) \, ,
    $$
concluding our proof since $\beta^{1/2} = (\log M)^{-1}$.\qed \\

%
%\begin{lemma}\label{largeclose.donotintersect} At least $(1-o(1))V^2$ pairs of vertices $x,y$ are such that for all radii $j_x,j_y \in [r,2r]$ we have that
%$$ \prob \big ( \partial B_x(j_x) \stackrel{2r_0}{\lrfill} \partial B_y(j_y) \, \big | \, \A(x,y,j_x,j_y) \big ) \leq {\eps^2 \over \log \eps^3 V} \, .$$
%%\, \big | \, x \is
%%    (\delta,\beta,j_x,r_0)\hbox{{\rm -fit}}, y \is (\delta,\beta,j_y,r_0)\hbox{{\rm -fit}} \big ) = o(1) \, .$$
%\end{lemma}

\subsection{Proof of Lemma \ref{errorfromlargeu}:
Bounding the error $\widehat{S}_{2r,2r,r_0}$}
In this section, we prove Lemma \ref{errorfromlargeu}.
We begin by providing some useful estimates.

\begin{lemma}
\label{connectearlyerror}
Assume the setting of Theorem \ref{alphapairs} and let $p=p_c(1+\eps)$.
There exists $C>0$ such that for any positive integer $n$
    \begin{enumerate}
    \item[{\rm (1)}]
    $\displaystyle \sum_{x,y,(u,u')} \prob \big( \{ u \stackrel{n}{\lrfill} x \} \circ \{ u' \stackrel{n}{\lrfill} y \} \and u \stackrel{2n}{\lrfill} u' \big ) \leq C \big [\m \eps^{-5} (1+\eps)^{4n} + \alpham V\m \eps^{-2} (1+\eps)^{2n} \big ]$.
    \item[{\rm (2)}]
    $\displaystyle \sum_{x,y,(u,u')} \prob \big( \{ u \stackrel{n}{\lrfill} x \} \circ \{ u' \stackrel{n}{\lrfill} y \} \and x \stackrel{2n}{\lrfill} y \big ) \leq C \big [\m \eps^{-5} (1+\eps)^{4n} + \alpham V\m \eps^{-2} (1+\eps)^{2n} \big ]$.
    \end{enumerate}
\end{lemma}
\begin{proof} We begin by showing (1). If $\{ u \stackrel{n}{\lrfill} x \} \circ \{ u' \stackrel{n}{\lrfill} y \}$
and $u \stackrel{2n}{\lrfill} u'$, then there exists vertices $z_1,z_2$ and integers $t_1,t_2 \leq n$ such that the event
%, and paths $\eta_1, \eta_2, \eta_3, \gamma_1, \gamma_2$ such that
%\begin{enumerate}
%\item $\eta_1$ is a shortest open path of length $t_1$ connecting $u$ to $z_1$,
%\item $\eta_2$ is a shortest open path of length $t_2$ connecting $u$ to $z_2$,
%\item $\eta_3$ is an open path of length at most $2n-t_1-t_2$ connecting $z_1$ to $z_2$,
%\item $\gamma_1, \gamma_2$ are open paths of length at most $n-t_1$ and $n-t_2$ respectively such that either $\gamma_1$ connects $x$ to $z_1$ and $\gamma_2$ connects $y$ to $z_2$, or $\gamma_1$ connects $x$ to $z_2$ and $\gamma_2$ connects $y$ to $z_1$.
%\item $\eta_1, \eta_2, \eta_3, \gamma_1, \gamma_2$ are disjoint.
%\end{enumerate}
    $$
    \{u \stackrel{=t_1}{\lrfill} z_1, u' \stackrel{=t_2}{\lrfill} z_2 , d_{G_p}(u,u') \geq t_1+t_2\}
    \circ \{z_1 \stackrel{2n-t_1-t_2}{\lrfill} z_2\} \circ \{x \stackrel{n-t_1}{\lrfill} z_1\}
    \circ \{y \stackrel{n-t_2}{\lrfill} z_2\} \, ,
    $$
or the event
    $$ \{u \stackrel{=t_1}{\lrfill} z_1, u' \stackrel{=t_2}{\lrfill} z_2 ,  d_{G_p}(u,u') \geq t_1+t_2\}
    \circ \{z_1 \stackrel{2n-t_1-t_2}{\lrfill} z_2\} \circ \{x \stackrel{n-t_1}{\lrfill} z_2\}
    \circ \{y \stackrel{n-t_2}{\lrfill} z_1\} \, ,
    $$
occur.
See Figure \ref{fig.connectearly} (a). Indeed, let $\eta$ be the shortest open path of
length at most $2n$ between $u$ and $u'$ and let $\gamma_{x,u}, \gamma_{y,u'}$ be two disjoint paths
of length at most $n$ connecting $x$ to $u$ and $y$ to $u'$, respectively. We take $z_1,z_2$
to be the first vertices of $\gamma_{x,u}$ and $\gamma_{y,u'}$ which belongs to $\eta$. There are two
possible orderings of $z_1,z_2$ on $\eta$, that is, ($u,z_1,z_2,u'$) or ($u,z_2,z_1,u'$),
which give the two possible events. Assume the ordering on $\eta$ is $(u,z_1,z_2,u')$ (the two orderings give rise to identical contributions to the sum in (1)), and put $t_1,t_2$ to be the distances on $\eta$ between $u$ and $z_1$ and between $z_2$ and $u'$, respectively and write $\eta_1, \eta_2$ to be the corresponding sections of $\eta$ and $\eta_3$ is the section of $\eta$ between $z_1$ and $z_2$. The paths $\gamma_1$ and $\gamma_2$ are the sections of $\gamma_{x,u}$ and $\gamma_{y,u'}$ from $x$ to $z_1$ and from $y$ to $z_2$, respectively. The witness for the first event is $\eta_1, \eta_2$ together with all the closed edges of $G_p$ (the closed edges determine that $\eta_1, \eta_2$ are indeed shortest open paths, and that $d_{G_p}(u,u')\geq t_1+t_2$), for the second, third and fourth events, the witnesses are just $\eta_3, \gamma_1$ and $\gamma_2$, respectively.

We now apply the BK-Reimer inequality and bound the sum in (1) by
    $$
    2\sum_{x,y,z_1,z_2,(u,u'),t_1\leq n, t_2\leq n} \prob(u \stackrel{=t_1}{\lrfill} z_1, u' \stackrel{=t_2}{\lrfill} z_2, d_{G_p}(u,u')\geq t_1+t_2) \prob(z_1 \stackrel{2n-t_1-t_2}{\lrfill} z_2)
    \prob(x \stackrel{n-t_1}{\lrfill} z_1) \prob(y \stackrel{n-t_2}{\lrfill} z_2)  \, .
    $$
We first sum over $x,y$ and get a factor of $C\eps^{-2}(1+\eps)^{2n-t_1-t_2}$ by
Lemma \ref{upperball}. The event $u \stackrel{=t_1}{\lrfill} z_1, u' \stackrel{=t_2}{\lrfill} z_2, d_{G_p}(u,u')\geq t_1+t_2$ implies that $u \stackrel{=t_1}{\lrfill} z_1$ and $u' \stackrel{=t_2}{\lrfill} z_2 \off B_u(t_1)$ hence we may bound its probability by
    $$ \sum_{A \colon u \stackrel{=t_1}{\lrfill} z_1} \prob(B_u(t_1)=A) \prob(u' \stackrel{=t_2}{\lrfill} z_2 \off A) \, ,$$
and so we get an upper bound of
    \be\label{connectearly.medstep}
    C \eps^{-2} \sum_{z_1,z_2,(u,u'),t_1\leq n, t_2\leq n} (1+\eps)^{2n-t_1-t_2} \prob(u \stackrel{=t_1}{\lrfill} z_1) \max_A \prob(u' \stackrel{=t_2}{\lrfill} z_2 \off A) \prob(z_1 \stackrel{2n-t_1-t_2}{\lrfill} z_2) \, .\ee

\begin{figure}
\begin{center}
\includegraphics{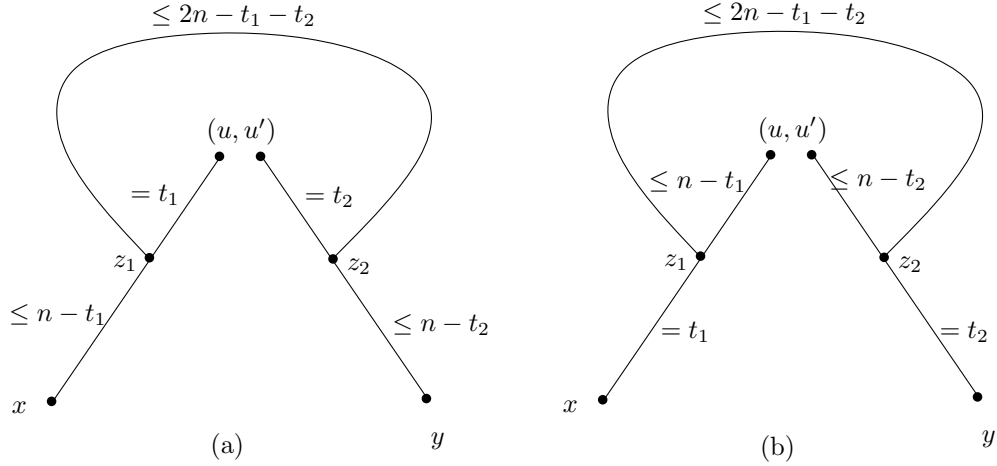}
\caption{The edge $(u,u')$ is counted in the first and
second sum of Lemma \ref{connectearlyerror}.}
\label{fig.connectearly}
\end{center}
\end{figure}

We bound this in two parts. If $t_2 \geq \mnot$, then we use
Lemma \ref{unifconnbdoff} together with Lemma \ref{upperball} to bound,
uniformly in $A$, $\prob(u' \stackrel{=t_2}{\lrfill} z_2 \off A)\leq CV^{-1}(1+\eps)^{t_2}$.
We then sum over $z_2$ and $z_1$ in that order using Lemma \ref{upperball}
and extract a $V\m$ factor from summing over $(u,u')$.
If $t_2 \leq \mnot$ and $t_1\geq \mnot$, then we use
Lemma \ref{uniformconnbd} together with Lemma \ref{upperball} to bound
$\prob(u \stackrel{=t_1}{\lrfill} z_1)\leq CV^{-1}(1+\eps)^{t_1}$.
Further, we use condition (2) in Theorem \ref{mainthmgeneral} and $\vep=o(\mnot^{-1})$
to bound, uniformly in $A$, $\prob(u' \stackrel{=t_2}{\lrfill} z_2 \off A)\leq C\p^{t_2}(u',z_2)$.
We then sum over $z_1$ and $z_2$ in that order using Lemma \ref{upperball}
and extract a $V\m$ factor from summing over $(u,u')$.
All this gives an upper bound of
    $$
    C \m \eps^{-3} (1+\eps)^{4 n} \sum_{t_1, t_2 \leq n} (1+\eps)^{-t_1-t_2}
    \leq C \m \eps^{-5} (1+\eps)^{4n} \, ,
    $$
as required. We next sum (\ref{connectearly.medstep}) over $t_1, t_2 \leq \mnot$. We first relax
$(1+\eps)^{2n-t_1-t_2} \leq (1+\eps)^{2n}$ and $\prob(z_1 \stackrel{2n-t_1-t_2}{\lrfill} z_2)
\leq \prob(z_1 \stackrel{2n}{\lrfill} z_2)$, and then sum over $t_1,t_2$ to get an upper bound of
    $$
    C \eps^{-2} (1+\eps)^{2n} \sum_{z_1,z_2, (u,u')}
    \prob(u \stackrel{\mnot}{\lrfill} z_1) \prob( u' \stackrel{\mnot}{\lrfill} z_2)
    \prob(z_1 \stackrel{2n}{\lrfill} z_2) \, .
    $$
We now sum over $z_1,z_2$ using Corollary \ref{extendedtriangle} and Lemma \ref{upperball}. We get that this is bounded by
    $$
    C V \m \eps^{-2} (1+\eps)^{2n} \Big [{\mnot^2 \eps^{-1} (1+\eps)^{2n} \over V}+\alpham\Big] \leq
    C \big [\m \eps^{-5} (1+\eps)^{4n} + \alpham V\m \eps^{-2} (1+\eps)^{2n} \big ] \, ,
    $$
since $\mnot \leq \eps^{-1}$, as required.

To bound (2) we proceed in a very similar fashion. If $\{ u \stackrel{n}{\lrfill} x \} \circ \{ u' \stackrel{n}{\lrfill} y \}$
and $x \stackrel{2n}{\lrfill} y$ then there exists vertices $z_1, z_2$ and $t_1, t_2 \leq n$
such that the event
    $$
    \{x \stackrel{=t_1}{\lrfill} z_1, y \stackrel{=t_2}{\lrfill} z_2, d_{G_p}(x,y)\geq t_1+t_2\}
    \circ \{z_1 \stackrel{2n-t_1-t_2}{\lrfill} z_2\} \circ \{u \stackrel{n-t_1}{\lrfill} z_1\}
    \circ \{u' \stackrel{n-t_2}{\lrfill} z_2\} \, ,
    $$
or the event
    $$
    \{x \stackrel{=t_1}{\lrfill} z_2, y \stackrel{=t_2}{\lrfill} z_1, d_{G_p}(x,y)\geq t_1+t_2\}
    \circ \{z_1 \stackrel{2n-t_1-t_2}{\lrfill} z_2\} \circ \{u \stackrel{n-t_1}{\lrfill} z_1\}
    \circ \{u' \stackrel{n-t_2}{\lrfill} z_2\} \, ,
    $$
occur, by the same reasoning as before, see Figure \ref{fig.connectearly} (b). Let us handle the first case only (the second leads to an identical contribution). We appeal to the BK-Reimer inequality and as before we condition on $B_x(t_1)$ and bound
$$ \prob(x \stackrel{=t_1}{\lrfill} z_1, y \stackrel{=t_2}{\lrfill} z_2, d_{G_p}(x,y)\geq t_1+t_2) \leq \sum_{A:x \stackrel{=t_1}{\lrfill} z_1} \prob(B_x(t_1)=A) \prob(y \stackrel{=t_2}{\lrfill} z_2 \off A) \, .$$
We sum over $y$ then $x$ using Lemma \ref{upperball} giving a bound of
    $$
    C \sum_{z_1,z_2,(u,u'), t_1,t_2 \leq n} (1+\eps)^{t_1+t_2} \prob(u \stackrel{n-t_1}{\lrfill} z_1)
    \prob(z_1 \stackrel{2n-t_1-t_2}{\lrfill} z_2) \prob(u' \stackrel{n-t_2}{\lrfill} z_2) \, .
    $$
An appeal to Corollary \ref{extendedtriangle} and Lemma \ref{upperball}
to sum over $z_1,z_2$ gives a bound of
    $$
    C V \m \sum_{t_1,t_2 \leq n} (1+\eps)^{t_1 + t_2}
    \Big [{\eps^{-3} (1+\eps)^{4n-2(t_1+t_2)} \over V}+\alpham\Big ]
    \leq C \big [\m \eps^{-5} (1+\eps)^{4n} + \alpham V\m \eps^{-2} (1+\eps)^{2n} \big ] \, ,
    $$
where the last inequality is a direct calculation. \end{proof}

\noindent{\bf Proof of Lemma \ref{errorfromlargeu}}. For convenience put $n=2r+r_0$.
By Markov's inequality, the expectation we need to bound is at most
    \be
    \label{errorfromlargeu.midstep}
    4\beta^{-1/2}Vm^{-1}\eps^2 (\E|B(r_0)|)^{-2} \sum_{x,y,(u,u')}
    \prob \big ( \{x \stackrel{n}{\lrfill} u\} \circ \{y \stackrel{n}{\lrfill} u'\} \, ,
    |B_u(n)| \geq \e^{20M} \eps^{-1} \E|B(r_0)| \big ) \, .
    \ee
We split the sum into
    $$
    S_1 = \sum _{x,y,(u,u')} \prob \big ( \{x \stackrel{n}{\lrfill} u\}
    \circ \{y \stackrel{n}{\lrfill} u'\} \, ,
    |B_u(n)| \geq \e^{20M} \eps^{-1} \E|B(r_0)| \and
    B_u(n) \cap B_{u'}(n) = \emptyset \big ) \, ,
    $$
and
    $$
    S_2 = \sum _{x,y,(u,u')} \prob \big ( \{x \stackrel{n}{\lrfill} u\}
    \circ \{y \stackrel{n}{\lrfill} u'\} \, , u \stackrel{2n}{\lrfill} u' \big ) \, .
    $$
We bound $S_1$ using the BK inequality
    $$
    S_1 \leq \sum_{x,y,(u,u')} \prob \big (x \stackrel{n}{\lrfill} u \, ,
    |B_u(n)|\geq \e^{20M} \eps^{-1} \E|B(r_0)| \big ) \prob(y \stackrel{n}{\lrfill} u') \, .
    $$
Summing over $y$ and $x$ and then over $(u,u')$ gives that this is at most
    $$ V\m \E |B(n)| \cdot \E|B(n)| {\bf 1}_{\{|B(n)| \geq
    \e^{20M} \eps^{-1} \E|B(r_0)| \}} | \, .
    $$
We use the Cauchy-Schwartz inequality to bound
    $$
    \E|B(n)| {\bf 1}_{\{ |B(n)| \geq \e^{20M} \eps^{-1} \E|B(r_0)| \}} |
    \leq \big [ \E |B(n)|^2 \prob(|B(n)| \geq \e^{20M} \eps^{-1} \E|B(r_0)|) \big ]^{1/2} \, .
    $$
We bound this using Lemma \ref{upperballsqaured} and the Markov inequality by
    $$
    \E|B(n)| {\bf 1}_{\{ |B(n)| \geq \e^{20M} \eps^{-1} \E|B(r_0)| \}} |
    \leq C \e^{-10M} (\E|B(n)|)^{3/2} (\E|B(r_0)|)^{-1/2} \, ,
    $$
and conclude that
    $$
    S_1 \leq C \e^{-10M} V \m (\E|B(n)|)^{5/2} (\E|B(r_0)|)^{-1/2}  \, .
    $$
We bound $S_2$ using part (1) of Lemma \ref{connectearlyerror} by
    $$
    S_2 \leq  C \big [\m \eps^{-5} (1+\eps)^{4n} + \alpham V\m \eps^{-2} (1+\eps)^{2n} \big ] \, .
    $$
We put these two back into (\ref{errorfromlargeu.midstep}) and get that we can bound this sum by
    $$
    {C V^2 \eps^2 (\E |B(n)|)^{5/2} \over \beta^{1/2} \e^{10M} (\E|B(r_0)|)^{5/2}} +
    {C V\eps^{-3} (1+\eps)^{4n} \over \beta^{1/2}  (\E |B(r_0)|)^2 } + {C \m \alpham V^2 (1+\eps)^{2n} \over  \beta^{1/2}  \m (\E|B(r_0)|)^2}
    = o(\eps^2 V^2) \, ,
    $$
by our choice of $r_0$ in \eqref{rzerocond}, $n=r_0+2r$, $r=M/\vep$, $\beta=(\log M)^{-2}$
and using Corollary \ref{r0ball}.
\qed

%
%We begin with the proof of Lemma \ref{largeclose.donotintersect} since it is the simplest. \\
%
%\noindent{\bf Proof of Lemma \ref{largeclose.donotintersect}.} We have that
%$$ \big \{ \partial B_x(j_x) \stackrel{2r_0}{\lrfill} \partial B_y(j_y) \and \A(x,y,j_x,j_y) \big \} \subset \{ x \stackrel{2r_0+4r}{\lrfill} y \} \, .$$
%Also, if $x \stackrel{2r_0+4r}{\lrfill} y$ then either $x \stackrel{r_0}{\lrfill} y$, or $x \stackrel{[r_0,2r_0+4r]}{\lrfill} y$. The latter event implies that there exists $z$ such that
%$$ \{ x \stackrel{=r_0}{\lrfill} z\} \circ \{z \stackrel{r_0+4r}{\lrfill} y \} \, .$$
%And so by summing over $z$ and $y$ using BK-Reimer inequality that
%$$ \E|B(2r_0+4r)| \leq \E|B(r_0)| + \E|\partial B_(r_0)| \E|B(r_0+4r)| \, .$$
%We use Theorem \ref{bdry}, Corollary \ref{corrlength} and Lemma \ref{upperball} to bound
%$$ \E|B(2r_0+4r)| \leq C\eps^{-1} (1+\eps)^{r_0} + \eps^{-1} (1+\eps)^{2r_0+4r} \leq C \eps (1+\eps)^4r (\E|B(r_0)|)^2 \, ,$$
%where the last inequality is via Theorem \ref{lowerball}. Thus
%$$ \sum_{y} \prob(\partial B_x(j_x) \stackrel{2r_0}{\lrfill} \partial B_y(j_y) \and \A(x,y,j_x,j_y)) \leq C \eps (1+\eps)^4r (\E|B(r_0)|)^2 \leq {\eps^2 V \over \log^2(\eps^3 V)} \, ,$$
%where we used the definition of $r_0$ in (\ref{rzerocond}). This implies that for most $x,y$ we have
%$$ \prob(\partial B_x(j_x) \stackrel{2r_0}{\lrfill} \partial B_y(j_y) \and \A(x,y,j_x,j_y)) = o(\eps^2) \, ,$$
%and also for most $x,y$ we have that $\prob(\A(x,y,j_x,j_y)) \geq c\eps$ by Lemma \ref{xysurvive}. This concludes the proof of the lemma.\qed \\

\subsection{Proof of Theorem \ref{largeclose.findradii}: Finding good radii}
We proceed towards the proof of Theorem \ref{largeclose.findradii}.
Recall the choice of parameters in \eqref{parameter-def-1}.

\begin{lemma}
\label{nottoosmallbdry}
For any radius $r \geq \eps^{-1}$ and any $\zeta>0$,
    $$
    \prob \Big ( |\partial B(r)|>0 \and \exists j \in [\eps^{-1}, r - \eps^{-1}]
    \with |\partial B(j)|\leq \zeta \eps^{-1} \Big ) \leq O(\zeta \eps) \, .
    $$
\end{lemma}

\begin{proof} Assume that the event holds, and let $J$
be the first radius $j$ with $j \in [\eps^{-1},r-\eps^{-1}]$ which has
$|\partial B(j)|\leq \zeta \eps^{-1}$. Conditioned on $J$ and $B(J)$, for
$|\partial B(r)|>0$ to occur, one of the vertices on the boundary of $B(J)$
needs to reach level $r$. Since $r - j \geq \eps^{-1}$, Corollary \ref{corrlength}
and the union bound gives that this probability is at most $C\zeta$. This together
with the fact that the probability of $|\partial B(j)|>0$ is at most $C\eps$, by Corollary \ref{corrlength}, concludes the proof.
\end{proof}

In the lemma below, we write $\pa(\cdot)=\prob( \cdot \off A \mid B_x(j_x)=A)$
and let $\ea$ be the corresponding expectation.

\begin{lemma}
\label{largeclose.findlarge}
There exists $c>0$ such that for any radius $j_x\in [r,2r]$ the following statement holds.
Let the set $A$ be such that $x$ is $(\delta, \beta, j_x, k\eps^{-1})$-fit
and $|\partial B_x(j_x)| \geq \zeta \eps^{-1}$ when $B_x(j_x)=A$. Then,
    $$
    \pa \big ( |\partial B_x(j_x+k\eps^{-1}/2)| \geq \eps^{-1} \e^{k/4} \big ) \geq c\zeta \, .
    $$
\end{lemma}
\begin{proof} We perform a second moment argument on $|B_x([j_x+k\eps^{-1}/2,j_x+k\eps^{-1}]|$
rather than on the required random variable. Since $x$ is $(\delta, \beta, j_x, k\eps^{-1})$-fit
    $$
    \ea |B_x([j_x,j_x+k\eps^{-1}])| \geq (|\partial A| - \delta\eps^{-1}) (1-\beta) \E|B(k\eps^{-1})| \, .
    $$
Furthermore,
    $$ \ea |B_x([j_x,j_x+k\eps^{-1}/2])| \leq |\partial A| \E|B(k\eps^{-1}/2)| \, ,
    $$
by monotonicity. Since $|\partial A| \geq 2\delta \eps^{-1}$ by our choice of $\zeta$ and $\delta$, and $\beta=o(1)$
(recall \eqref{parameter-def-1}), Corollary \ref{r0ball} now gives us a lower bound on the first moment
    $$
    \ea |B_x([j_x+k\eps^{-1}/2,j_x+k\eps^{-1}])| \geq {1 \over 4} |\partial A| \E|B(k\eps^{-1})| \, ,
    $$
To calculate the second moment, if $u,v$ are counted in $|B([j_x,j_x+k\eps^{-1}])|$, then
either there exists two vertices $a_1, a_2$ in $\partial A$ such that
    $$
    \{ a_1 \stackrel{k\eps^{-1}}{\lrfill} u \off A\} \circ \{ a_2 \stackrel{k\eps^{-1}}{\lrfill} v \off A\} \, ,
    $$
or there exists $a \in \partial A$, a vertex $z$ and $t \leq k\eps^{-1}$ such that
    $$
    \{a \stackrel{=t}{\lrfill} z \off A\} \circ \{z \stackrel{k\eps^{-1}-t}{\lrfill} u \off A\}
    \circ \{z \stackrel{k\eps^{-1}-t}{\lrfill} v \off A\} \, .
    $$
We apply the BK-Reimer inequality and sum over $u,v$. We get
    $$
    \ea |B_x([j_x,j_x+k\eps^{-1}])|^2 \leq |\partial A|^2 (\E|B(k\eps^{-1})|)^2 +
    \sum_{a\in\partial A,z,t\leq k\eps^{-1}} \pa(a \stackrel{=t}{\lrfill} z \off A)
    (\E|B(k\eps^{-1}-t)|)^2 \, .
    $$
We first sum over $z$ using Lemma \ref{upperball}, then appeal again to
Corollary \ref{r0ball} to get that
%Fixed this!
%\todo\{It would help here to formulate Lemma \ref{upperball} in terms of $\pa$...}
    $$
    \ea |B_x([j_x,j_x+k\eps^{-1}])|^2 \leq C(\E|B(k\eps^{-1})|)^2 \big [ |\partial A|^2 + |\partial A| \eps^{-1} \big ] \, .
    $$
By \eqref{X>a-prob-bd},
    $$
    \pa \big (|B_x([j_x+k\eps^{-1}/2,j_x+k\eps^{-1}])| \geq \tfrac{1}{8} |\partial A| \E|B(k\eps^{-1})| \big ) \geq { c|\partial A|^2 \over  |\partial A|^2 + |\partial A| \eps^{-1} } \geq c\zeta \, ,
    $$
where the last inequality is since $|\partial A| \geq \zeta \eps^{-1}$.
By Theorem \ref{lowerball}, we can write this as
    \be\label{findlarge.secmoment}
    \pa \big (|B_x([j_x+k\eps^{-1}/2,j_x+k\eps^{-1}])| \geq c \zeta \eps^{-2} \e^k \big ) \geq c\zeta \, ,
    \ee
for some constant $c>0$. Now, if $|B_x([j_x+k\eps^{-1}/2,j_x+k\eps^{-1}])| \geq c \zeta \eps^{-2} \e^k$
and $|\partial B_x(j_x+k\eps^{-1}/2)| \leq \eps^{-1} \e^{k/4}$ occurs, then
    $$
    |\partial B_x(j_x+k\eps^{-1}/2)| \leq \eps^{-1} \e^{k/4}\qquad \and
    \qquad \sum_{v \in \partial B_x(j_x+k\eps^{-1}/2)} |B_v(k\eps^{-1}/2; A)| \geq c \zeta \eps^{-2} \e^k \, ,
    $$
must both occur. By the Markov inequality and Lemma \ref{upperball}, the probability of this event is at most
    $$
    {\eps^{-2} \e^{3k/4} \over c \zeta \eps^{-2} \e^k} = O(\zeta^{-1} \e^{-k/4}) = o(\zeta) \, ,
    $$
by our choice of $\zeta$ and $k$ in \eqref{parameter-def-1}.
Putting this together with (\ref{findlarge.secmoment})
yields the assertion of the lemma.
\end{proof}

\begin{lemma}[Finding good radii]
\label{findradii} There exists radii $k_1, \ldots, k_{\ell}$ in $[r,2r]$ such that
    $$
    k_{i+1} - k_i \geq k\eps^{-1} \, ,
    $$
for all $i=1,\ldots, \ell$ and
\begin{eqnarray*}
\prob(x \is (\delta,\beta,k_i,k\eps^{-1})\hbox{{\rm -fit}})
&=& ( 1+O(\omegam^{1/5}) )\prob(\partial B_x(k_i) \neq \emptyset) \, ,\\
\prob(x \is (1,\beta,k_i,r_0) \hbox{{\rm -fit}})
&=& ( 1+O(\omegam^{1/5}) ) \prob(\partial B_x(k_i) \neq \emptyset) \, , \\
\prob(x \is (\delta,\beta,k_i+k\eps^{-1}/2, r_0)\hbox{{\rm -fit}})
&=& ( 1+O(\omegam^{1/5}) ) \prob(\partial B_x(k_i+k\eps^{-1}/2) \neq \emptyset) \, .
\end{eqnarray*}
\end{lemma}

\begin{proof} This is the only place where we use Theorem \ref{mostsetsaregood}.
Indeed, say a radius $j\in[r,2r]$ is {\em good} if it satisfies the three assertions
of the proposition with $k_i$ replaced by $j$. Three appeals to
Theorem \ref{mostsetsaregood} give that at least $(1-o(1))r$ radii $j\in[r,2r]$
are good by our choice of $\delta$ and $\beta$. Now, since $\ell k = o(M)$ and
$r=M\eps^{-1}$ it is immediate that there exist $\ell$ good radii which are $k\eps^{-1}$
separated from each other.
\end{proof}

\begin{lemma}
\label{useconnectearlyerror}
For at least $(1-o(1))V^2$ pairs $x,y$ and for any $j_x,j_y \in [r,2r]$,
    $$
    \E \Big[S_{j_x,j_y,r_0} {\bf 1}_{\{ x \stackrel{2r_0+4r}{\lrlong} y \}}\Big]\leq
     V^{-1} \m (\E|B(r_0)|)^2 \alpham^{3/4} \, .
    $$
\end{lemma}
\begin{proof} Part (2) of Lemma \ref{connectearlyerror} with $n=r_0+2r$ and a straightforward
calculation with Theorem \ref{lowerball} and our choice of parameters shows that
    $$
    \sum_{x,y} \E \big [ S_{j_x,j_y,r_0} {\bf 1}_{\{ x \stackrel{2r_0+j_x+j_y}{\lrlong} y \}} \big ] \leq CV\m (\E|B(r_0)|)^2
     [ \alpham \e^{8M} + \alpham\e^{4M} ] \, ,
    $$
which gives the result since $C \alpham \e^{8M} \leq \alpham^{3/4}$ by our choice of $M$ in (\ref{Mchoice}).
\end{proof}

\noindent{\bf Proof of Theorem \ref{largeclose.findradii}.}
Recall the requirements (1)-(5) in the definition of $\A(x,y,j_x,j_y, r_0,\beta,k)$.
We apply Lemma \ref{findradii} and let $k_1, \ldots, k_\ell$ be the
corresponding radii. We prove the theorem with radii $\{j_1, \ldots, j_\ell\}$ defined by
    $$
    j_i = k_i + k\eps^{-1}/2 \, ,
    $$
for $i=1,\ldots, \ell$ and assume $x,y$ are such that the assertion of
Lemma \ref{useconnectearlyerror} holds. We will prove that for these pairs $x,y$
    \eqn{
    \label{aim-i}
    \prob\Big(\A(x,y,2r,2r)\and \bigcap_{j_x,j_y\in \{j_1, \ldots, j_\ell\}}\{\mbox{(q) does not hold
    for $j_x,j_y$}\}\Big)
    =o(\vep^2) \, ,
    }
for $q \in \{1,2,3,4,5\}$. We do this in the order (1), (2), (4), (5) and (3).
Since $\A(x,y,2r,2r)\subseteq \A(x,y,j_x,j_y)$
when $j_x,j_y\leq 2r$, \eqref{aim-i} holds trivially for $q=1$
and all $x,y$, $j_x,j_y\leq 2r$.

For any $j_x \in \{j_1,\ldots, j_\ell\}$,
    $$
    \prob \big (\A(x,y,j_x,j_y) \and |B_x(j_x)| \geq \eps^{-2} (1+\eps)^{3r} \big )
    \leq C\eps^2 (1+\eps)^{-r} = O(\e^{-M} \eps^2) \, ,
    $$
by the Markov inequality, Lemma \ref{upperball}, the BK-Reimer inequality
and Corollary \ref{corrlength}. This implies that
    $$
    \prob \big ( \A(x,y,2r,2r) \and \exists j_x \in \{j_1,\ldots, j_\ell\} \,\,
    \mbox{ such that }|B_x(j_x)| \geq \eps^{-2} (1+\eps)^{3r} \big ) = o(\eps^2) \, ,
    $$
since $\ell = o(\e^M)$. Similarly,
    $$
    \prob \big (\A(x,y,j_x,j_y) \and |\partial B_x(j_x)| \geq \eps^{-1} (1+\eps)^{3r} \big )
    \leq C\eps^2 (1+\eps)^{-r} = O(\e^{-M} \eps^2) \, ,
    $$
leading to the same bound. This proves \eqref{aim-i} for $q=2$.

Next, we wish to show that for any $j_x \in \{j_1,\ldots, j_\ell\}$,
    \be
    \label{findradii.midstep}
    \prob \big (\A(x,y,j_x,j_y) \and x \hbox{ {\rm is not} }
    (1, \beta, j_x, r_0)\hbox{{\rm -fit}} \big )
    = O(\eps^2\omegam^{1/5}) \, .
    \ee
It is tempting to use the BK-Reimer inequality here, however, we cannot
claim that the event in \eqref{findradii.midstep} implies that $\partial B_y(j_y) \neq \emptyset$
occurs disjointly from the event $x \hbox{ {\rm is not} } (1, \beta, j_x, r_0)\hbox{{\rm -fit}}$,
since they are both non-monotone events and the corresponding witnesses may share closed edges.
Instead, we condition $B_x(j_x)=A$ and get that
    $$
    \prob \big (\A(x,y,j_x,j_y) \and x \hbox{ {\rm is not} } (1, \beta, j_x, r_0)\hbox{{\rm -fit}} \big ) = \sum_{A \colon x \hbox{ {\rm is not} } (1, \beta, j_x, r_0)\hbox{{\rm -fit}}} \prob(B_x(j_x)=A) \prob( \partial B_y(j_y) \neq \emptyset \off A) \, ,
    $$
since $(1, \beta, j_x, r_0)\hbox{{\rm -fit}}$ is determined by the status of the edges
touching $B_x(j_x)$. We use Corollary \ref{corrlength} to bound
$\prob( \partial B_y(j_y) \neq \emptyset \off A) = O(\eps)$ and
    \eqan{
    \prob(\partial B_x(j_x)\neq \varnothing \and x \hbox{ {\rm is not} }
    (1, \beta, j_x, r_0)\hbox{{\rm -fit}})
    &=\prob(\partial B_x(j_x)\neq \varnothing)-
    \prob(\partial B_x(j_x)\neq \varnothing \and x \hbox{ {\rm is } }
    (1, \beta, j_x, r_0)\hbox{{\rm -fit}})\nn\\
    &\leq \prob(\partial B_x(j_x)\neq \varnothing)O(\omegam^{1/5})\nn,
    }
by our choice of radii in Lemma \ref{findradii}, so Corollary \ref{corrlength} gives (\ref{findradii.midstep}).
Therefore,
    $$
    \prob \big ( \A(x,y,2r,2r) \and \exists j_x \in \{j_1,\ldots, j_\ell\}
    \text{ such that } x \hbox{ {\rm is not} }
    (1, \beta, j_x, r_0)\hbox{{\rm -fit}} \big )
    =O(\ell\eps^2\omegam^{1/5})= o(\eps^2) \, ,
    $$
by our choice of $\ell$ in \eqref{parameter-def-1}, of $M$ in \eqref{Mchoice},
and of $\omegam$ in \eqref{omegam-def}. This proves \eqref{aim-i} for $q=4$.

Similarly, by Lemma \ref{useconnectearlyerror} and Markov's inequality, for any $j_x,j_y$,
    $$
    \prob \big ( \A(x,y,j_x,j_y) \and \hbox{{\rm (5) does not hold}}\big )
    \leq C\eps^2 \alpham^{1/4} \, .
    $$
The union bound implies that
    $$
    \prob \big ( \A(x,y,2r,2r) \and \exists j_x,j_y \in \{j_1,\ldots, j_\ell\} \,\,
    \hbox{{\rm (5) does not hold}} \big ) = O(\ell^2 \alpham^{1/4} \eps^2) = o(\eps^2) \, ,
    $$
by our choice of $\ell$. Therefore, \eqref{aim-i} holds for $q=5$.

Thus, it remains to prove \eqref{aim-i} for $q=3$.
This is the difficult requirement in which we only prove
that one of the radii in $\{j_1,\ldots, j_\ell\}$ satisfies
it (in fact, all radii satisfy it, but that is harder to prove,
and we refrain from doing so). In the same way as in the proof
of \eqref{aim-i} for $q=4$, using Corollary \ref{corrlength},
it is enough to show that
    \be
    \label{largeclose.findradii.toprove}
    \prob \big ( \partial B_x(2r) \neq \emptyset \and |\partial B_x(j_x)|
    \leq \e^{k/4}\eps^{-1} \quad \forall \, j_x \in \{j_1,\ldots, j_\ell\} \big )
    = o(\eps) \, .
    \ee
For $i\in\{1, \ldots, \ell\}$, we write $\A_i$ for the event that
$x$ is $(\delta,\beta,k_i,k\eps^{-1})$-fit and $\B_i$ for the event that
$|\partial B_x(j_i)| \leq \eps^{-1} \e^{k/4}$ and $\D_i$ for the event
    $$
    \D_i = \big \{|\partial B_x(k_t)| \geq \zeta \eps^{-1} \quad \forall \, t \in \{1,\ldots, i\} \big \},
    $$
so that
    $$
    \prob \big ( \partial B_x(2r) \neq \emptyset \and |\partial B_x(j_x)|
    \leq \prob(\D_\ell \cap \B_1 \cap \cdots \cap \B_{\ell} )
    +\prob(\{\partial B_x(2r) \neq \emptyset\} \cap \D_{\ell}^c).
    $$
Then we can split
    $$
    \prob(\D_\ell \cap \B_1 \cap \cdots \cap \B_{\ell} )
    \leq \prob(\D_\ell \cap \A_{\ell}^c)
    + \prob ( \D_\ell \cap \B_1 \cap \cdots \cap \B_{\ell} \cap \A_{\ell}) \, .
    $$
By our choice of $k_i$ in Lemma \ref{findradii} and Corollary \ref{corrlength}
we have that $\prob(\D_\ell\cap \A_\ell^c) \leq \eps \omegam^{1/5}$, so that
    \begin{eqnarray*}
    \prob( \D_\ell \cap \B_1 \cap \cdots \cap \B_{\ell} )
    \leq {\eps \omegam^{1/5}} +
    \prob ( \B_{\ell} \mid \D_\ell \cap \B_1\cap  \cdots \cap \B_{\ell-1} \cap \A_{\ell} )
    \prob(\D_{\ell-1} \cap \B_1 \cap \cdots \cap \B_{\ell-1}) \, .
    \end{eqnarray*}
Thus, by Lemma \ref{largeclose.findlarge},
    \begin{eqnarray*}
    \prob( \D_\ell \cap \B_1 \cap \cdots \cap \B_{\ell} )
    \leq {\eps \omegam^{1/5}} + (1-c\zeta)
    \prob( \D_{\ell-1}\cap \B_1 \cap \cdots \cap \B_{\ell-1}) \, ,
    \end{eqnarray*}
By iterating this we obtain
    $$
    \prob( \D_\ell \cap\B_1 \cap \cdots \cap \B_{\ell} )
    \leq { \eps \ell \omegam^{1/5}}  + C \eps (1-c\zeta)^\ell
    = o(\eps) \, ,
    $$
since $\ell \omegam^{1/5}=o(1)$ and
$\zeta^{-1} = o(\ell)$ (recall \eqref{parameter-def-1}),
and $\prob(\D_1) \leq C\eps$ by Corollary \ref{corrlength}.
Lastly, Lemma \ref{nottoosmallbdry} shows that
    $$
    \prob(\{\partial B_x(2r)\neq \emptyset\}\cap \D_{\ell}^c) = o(\eps) \, ,
    $$
showing (\ref{largeclose.findradii.toprove}) and thus concluding the proof of \eqref{aim-i}
for $q=3$ and the proof of Theorem \ref{largeclose.findradii}.
\qed

\subsection{Proof of Theorem \ref{largeclose.getsedge}:
Conditional second moment}
We now set the stage for the proof of Theorem \ref{largeclose.getsedge}.
We perform this by a conditional second moment argument on
$S_{j_x,j_y,r_0}(x,y)$. We will be conditioning on
$B_x(j_x)=A$ and $B_y(j_y)=B$ where $A$ and $B$ are
such that the event $\A(x,y,j_x,j_y,r_0,\beta,k)$
holds. We abuse notation, as before, and treat $A,B$ as sets of vertices
but our conditioning is on the status of all edges touching $B_x(j_x-1)$
and $B_y(j_y-1)$. Thus, while $A$ and $B$ are disjoint sets of vertices,
they may be sharing closed edges. With this in mind, we generalize the
notation just before Lemma \ref{largeclose.findlarge}, and write $\pa$,
$\pb$ and $\pab$ for the measures
    \begin{eqnarray*}
    \pa(\cdot) &=& \prob( \cdot \off A \mid B_x(j_x)=A) \, , \\
    \pb(\cdot) &=& \prob( \cdot \off B \mid B_y(j_y)=B) \, , \\
    \pab(\cdot) &=& \prob( \cdot \off A \cup B \mid B_x(j_x)=A, B_y(j_y)=B) \, .
    \end{eqnarray*}
We start by proving five preparatory lemmas.

\begin{lemma}
\label{sedgefirstmommain}
Assume that $A,B$ are such that $x,y$ are $(1,\beta,j_x,r_0)$-fit and
$(1,\beta,j_y,r_0)$-fit, respectively. Then,
    $$
    \sum_{a,b\in \partial A \times \partial B} \sum_{(u,u')}
    \pa(a \lrefmpx u)\pb(b \lrefmpy u') \geq (1-8\beta^{1/2}) V^{-1}
    (\E|B(r_0)|)^2 m (|\partial A| - \eps^{-1})(|\partial B| - \eps^{-1}) \, .
    $$
\end{lemma}

\begin{proof} Let $a \in \partial A$ be a $(\beta,j_x,r_0)$-regenerative vertex.
Then, by definition,
    $$
    \sum_{u} \pa(a \lrefmpx u) \geq (1-\beta) \E |B(r_0)| \, .
    $$
Denote by $U$ the set of vertices
    $$
    U = \big\{ u \colon \pa(a \lrefmpx u) \leq (1- \beta^{1/2}) V^{-1} \E |B(r_0)|\big\} \, ,
    $$
and recall that Lemma \ref{unifconnbdoff} guarantees that
    $$
    \pa\big(a \lrefmpx u\big) \leq {(1+o(\beta)) \E |B(r_0)| \over V} \, ,
    $$
by our choice of $\beta$ in \eqref{parameter-def-1}, so that
    $$
    (1-\beta) \E|B(r_0)| \leq \sum_{u} \pa(a \lrefmpx u)
    \leq |U| (1- \beta^{1/2}) V^{-1} \E |B(r_0)| + (V-|U|) (1+o(\beta)) V^{-1} \E |B(r_0)| \, ,
    $$
and we deduce that $|U| \leq 2\beta^{1/2} V$. In other words,
for at least $(1-2\beta^{1/2})V$ vertices $u$,
    $$
    \pa\big(a \lrefmpx u\big) \geq {(1-\beta^{1/2}) \E |B(r_0)| \over V} \, .
    $$
Similarly, for any $b\in \partial B$ which is $(\beta,j_y,r_0)$-regenerative,
there exist at least $(1-2\beta^{1/2})V$ vertices $u$ such that
    $$
    \pb\big(b \lrefmpy u\big) \geq {(1-\beta^{1/2}) \E |B(r_0)| \over V} \, .
    $$
Thus, for such $a$ and $b$, at least $(1-4\beta^{1/2})V$
vertices $u$ satisfy both inequalities. Write $D$ for this set of
vertices so that $|D^c| \leq 4\beta^{1/2} V$. Since the degree of each
vertex is $\m$, the number of edges having at least one side
in $D^c$ is at most $4\beta^{1/2} V \m$. Thus, at least $(1-8\beta^{1/2})V\m$
directed edges $(u,u')$ are such that $u$ and $u'$ both satisfy the
above inequalities. Hence
    $$
    \sum_{(u,u')} \pa\big(a \lrefmpx u\big) \pb\big(b \lrefmpy u'\big)
    \geq {(1-8\beta^{1/2}) (\E|B(r_0)|)^2 \m \over V} \, .
    $$
Since $x$ is $(1,\beta,j_x,r_0)$-fit and $y$ is $(1,\beta,j_y,r_0)$-fit the number of such pairs $a,b$ is at least
    $$
    (|\partial A| - \eps^{-1})(|\partial B| - \eps^{-1}) \, ,
    $$
and the lemma follows.
\end{proof}

\begin{lemma} \label{sedgefirstmomerror1}
The following bounds hold:
    \eqan{
    &\sum_{\substack{(a,b) \in \partial A \times \partial B \\ (u,u'), z_1, t_1 \in [\mnot,r_0]} }
     \pab\big(a \lrefmp u, a \lref z_1\big) \pb(b \stackrel{=t_1}{\lrfill} z_1)
     \pb(z_1 \stackrel{r_0-t_1}{\lrfill} u')
     \leq \alpham^{1/2} |\partial A||\partial B|V^{-1} \m (\E|B(r_0)|)^2 \, ,\nn
    }
and
    \eqan{
    &\sum_{\substack{(a,b) \in \partial A \times \partial B \\ (u,u'), z_1, t_1 \in [\mnot,r_0]} }
     \pab(a \lrefmp u, a \lref z_1) \pb(b \stackrel{r_0-t_1}{\lrfill} z_1) \pb(z_1 \stackrel{=t_1}{\lrfill} u')
    \leq \alpham^{1/2} |\partial A||\partial B|V^{-1} \m (\E|B(r_0)|)^2 \, .\nn
    }
\end{lemma}
\begin{proof} The proof of the second assertion is identical to the first, so we only prove the first. If $a \lrefmp u$ and $a \lref z_1$, then there exists $z_2$ and $t_2 \in [\mnot,r_0]$ such that
    $$
    \{a \stackrel{=t_2}{\lrfill} z_2\} \circ \{z_2 \stackrel{r_0-t_2}{\lrfill} u\} \circ \{z_2 \stackrel{r_0-t_2}{\lrfill} z_1\}
    $$
or there exists $z_2$ such that
    $$
    \{a \stackrel{\mnot}{\lrfill} z_2\} \circ \{z_2 \lrefmonep u\} \circ \{z_2\lrefmonep z_1\} \, .
    $$
To see this, let $\eta$ be the lexicographically first shortest open path between $a$ and $z_1$ so
that $|\eta|\leq r_0$ and let $\gamma$ be an open path between $a$ and $u$
such that $|\gamma| \in [2\mnot,r_0]$. Let $z_2$ be the last vertex (according
to the ordering induced by $\gamma$) on $\gamma$ belonging to $\eta$ (that is,
the part of $\gamma$ after $z_2$ is disjoint from $\eta$). Let $t_2$ be the
distance between $a$ and $z_2$ along $\eta$. If $t_2 \geq \mnot$, then the
first event occurs: the first witness is the first $t_2$ open edges of $\eta$
together with all the closed edges in the graph, the second
witness is the set of open edges of $\gamma$ between $z_2$ to $u$ (note that
there are no more than $r_0-t_2$ edges since the part of $\gamma$ between $a$
to $z_2$ is of length at least $t_2$) and the third witness is the set of
open edges of $\eta$ between $z_2$ and $u$.  If $t_2 \leq \mnot$ occurs, then
the second event occurs by a similar reasoning.

This leads to two different calculations, we use the BK-Reimer inequality and
get that the required sum is at most $S^{\sss(a)} + S^{\sss(b)}$, where
    $$
    S^{\sss (a)} = \sum_{\substack{(a,b) \in \partial A \times \partial B,(u,u'), \\ z_1,z_2,t_1,t_2 \in [\mnot,r_0]} } \pab(a \stackrel{=t_2}{\lrfill} z_2)\pab(z_2 \stackrel{r_0-t_2}{\lrfill} u) \pab(z_2 \stackrel{r_0-t_2}{\lrfill}z_1) \pb(b \stackrel{=t_1}{\lrfill} z_1) \pb(z_1 \stackrel{r_0-t_1}{\lrfill} u') \, ,$$
and
    $$
    S^{\sss (b)} = \sum_{\substack{(a,b) \in \partial A \times \partial B, (u,u'), \\ z_1,z_2,t_1 \in [\mnot,r_0]} } \pab(a \stackrel{\mnot}{\lrfill} z_2) \pab(z_2 \lrefmonep u)\pab(z_2\lrefmonep z_1)  \pb(b \stackrel{=t_1}{\lrfill} z_1) \pb(z_1 \stackrel{r_0-t_1}{\lrfill} u') \, .
    $$
We use Lemma \ref{unifconnbdoff} together with Lemma \ref{upperball} to bound the terms $\pab(a \stackrel{=t_2}{\lrfill} z_2)$ and $\pb(b \stackrel{=t_1}{\lrfill} z_1)$ in $S^{\sss(a)}$ by $CV^{-1}(1+\eps)^{t_2}$ and $CV^{-1}(1+\eps)^{t_1}$, respectively. This gives
    $$
    S^{\sss(a)} \leq CV^{-2} \sum_{\substack{(a,b) \in \partial A \times \partial B,(u,u'), \\ t_1,t_2 \in [\mnot,r_0]} } (1+\eps)^{t_1+t_2} \sum_{z_1,z_2} \pab(z_2 \stackrel{r_0-t_2}{\lrfill} u)\pab(z_2 \stackrel{r_0-t_2}{\lrfill}z_1) \pab(z_1 \stackrel{r_0-t_1}{\lrfill} u') \, .
    $$
We sum over $z_1,z_2$ using Corollary \ref{extendedtriangle} together with Lemma \ref{upperball} to get that
    \begin{eqnarray*}
    S^{\sss(a)} &\leq& C V^{-2} \sum_{\substack{(a,b) \in \partial A \times \partial B,(u,u'), \\ t_1,t_2 \in [\mnot,r_0]} } (1+\eps)^{t_1+t_2} \Big [ \alpham + {\eps^{-3}(1+\eps)^{3r_0-t_1-2t_2} \over V} \Big ] \\ &\leq& C \alpham\m |\partial A||\partial B| V^{-1} \eps^{-2}(1+\eps)^{2r_0} + C|\partial A||\partial B|V^{-2} m \eps^{-4}(1+\eps)^{3r_0} r_0  \, ,
    \end{eqnarray*}
where the last inequality is an immediate calculation. By Theorem \ref{lowerball} the first term is at most
    $$
    C\alpham |\partial A||\partial B|V^{-1} \m (\E|B(r_0)|)^2 \, ,
    $$
and by our choice of $r_0$ in (\ref{rzerocond}), the second term is at most
    $$
    {\alpham^{1/2}\log(\eps^3 V) \over \sqrt{\eps^3 V}} |\partial A||\partial B|V^{-1}\m(\E|B(r_0)|)^2 \, .
    $$
This gives an upper bound on $S^{\sss(a)}$ fitting the error in the assertion of the lemma. To estimate $S^{\sss(b)}$ we use Lemma \ref{uniformconnbd} to bound $\pab(z_2 \lrefmonep u)$ and $\pab(z_2\lrefmonep z_1)$. This gives
    $$
    S^{\sss(b)} \leq C {(\E|B(r_0)|)^2 \over V^2} \sum_{\substack{(a,b) \in \partial A \times \partial B, (u,u'), \\ z_1,z_2,t_1 \in [\mnot,r_0]} }\pab(a \stackrel{\mnot}{\lrfill} z_2) \pb(b \stackrel{=t_1}{\lrfill} z_1) \pb(z_1 \stackrel{r_0-t_1}{\lrfill} u') \, .$$
We now sum over $(u,u')$, $z_1, z_2$ and $t_1$ using Lemma \ref{upperball}. We get that
    \begin{eqnarray*}
    S^{\sss(b)} &\leq& C |\partial A||\partial B| V^{-2} m {(\E|B(r_0)|)^2 } r_0 \mnot \eps^{-1} (1+\eps)^{r_0} \leq {C \mnot \vep \alpham^{1/2} \over \sqrt{\eps^3 V}} C |\partial A||\partial B| V^{-1} \m {(\E|B(r_0)|)^2 } \, ,
    \end{eqnarray*}
by Theorem \ref{lowerball} and (\ref{rzerocond}), concluding our proof since $\mnot =o(\eps^{-1})$.
\end{proof}

\begin{lemma} \label{sedgefirstmomerror3}
The following bounds hold:
    \eqan{
    &\sum_{\substack{(a,b) \in \partial A \times \partial B, a' \in \partial A \\ (u,u'), z_1, t_1 \in [\mnot,r_0]} }
     \!\!\!\!\!\!\!\! \!\!\!\!  \pab (a \lrefmp u) \pab(a' \lref z_1) \pb(b \stackrel{=t_1}{\lrfill} z_1)
     \pb(z_1 \stackrel{r_0-t_1}{\lrfill} u')
     \leq \alpham^{1/2} \eps |\partial A|^2|\partial B| V^{-1} m (\E|B(r_0)|)^2 \, ,\nn
    }
and
    \eqan{
    &\sum_{\substack{(a,b) \in \partial A \times \partial B, a' \in \partial A \\ (u,u'), z_1, t_1 \in [\mnot,r_0]} }
     \!\!\!\! \!\!\!\! \!\!\!\!  \pab(a \lrefmp u) \pab(a' \lref z_1) \pb(b \stackrel{r_0-t_1}{\lrfill} z_1) \pb(z_1 \stackrel{=t_1}{\lrfill} u')
    \leq  \alpham^{1/2} \eps |\partial A|^2 |\partial B| V^{-1} m (\E|B(r_0)|)^2 \, . \nn
    }
\end{lemma}
\begin{proof} The proof of the second assertion is identical to the first, so we only prove the first. We use Lemma \ref{unifconnbdoff} to bound $\pab (a \lrefmp u) \leq CV^{-1} \E|B(r_0)|$ and $\pb(b \stackrel{=t_1}{\lrfill} z_1) \leq C V^{-1} (1+\eps)^{t_1}$. We then sum $\pb(z_1 \stackrel{r_0-t_1}{\lrfill} u')$ over $(u,u')$ to obtain a factor of $C m \eps^{-1} (1+\eps)^{r_0-t_1}$ by Lemma \ref{upperball}. We now sum $\pab(a' \lref z_1)$ over $z_1$ and get another $\E|B(r_0)|$ factor and now sum all this over $a,b,a',t_1$. This gives a contribution of
$$ C|\partial A|^2 |\partial B| V^{-2} m r_0 (\E|B(r_0)|)^3 \, ,$$
by Theorem \ref{lowerball}. This is at most
$$ {C\eps |\partial A| \alpham^{1/2} \log(\eps^3 V) \over \sqrt{\eps^3 V}} |\partial A||\partial B| V^{-1} m (\E|B(r_0)|)^2 \, ,$$
by our choice of $r_0$ in \eqref{rzerocond} and Lemma \ref{upperball}, concluding our proof.
\end{proof}

\begin{lemma}\label{sedgefirstmomerror2}
The following bounds hold:
     $$
    \sum_{\substack{a ,b\in \partial A \times \partial B \\ (u,u'), z_1 \in B, t_1 \in [\mnot,r_0]}} \pa(a \stackrel{=t_1}{\lrfill} z_1)\pa(z_1 \stackrel{r_0-t_1}{\lrfill} u) \pb(b \lrefmp u') \leq C |\partial A||\partial B||B| V^{-2} \m r_0 (\E|B(r_0)|)^2 \, ,
    $$
and
    $$
    \sum_{\substack{a ,b\in \partial A \times \partial B \\ (u,u'), z_1 \in B, t_1 \in [\mnot,r_0]}} \pa(a \stackrel{r_0-t_1}{\lrfill} z_1)\pa(z_1 \stackrel{=t_1}{\lrfill} u) \pb(b \lrefmp u') \leq C |\partial A||\partial B||B| V^{-2} \m r_0 (\E|B(r_0)|)^2 \, .
    $$
\end{lemma}
\begin{proof} The proof of the second assertion is identical to the first so we only prove the first. We use Lemma \ref{uniformconnbd} to bound
    $$
    \pb(b \lrefmp u') \leq (1+o(1))V^{-1} \E|B(r_0)| \, ,
    $$
and, as before, we use Lemma \ref{moreunifconnbd} together with Lemma \ref{upperball} to bound
    $$
    \pa(a \stackrel{=t_1}{\lrfill} z_1) \leq CV^{-1}(1+\eps)^{t_1} \, ,
    $$
sum over $u'$ such that $(u,u')\in E(G)$, and finally use Lemma \ref{upperball} to bound
    $$
    \sum_{u} \pa(z_1 \stackrel{r_0-t_1}{\lrfill} u) \leq C\eps^{-1}(1+\eps)^{r_0-t_1} \, .
    $$
Altogether, after summing over $a\in \partial A, b\in \partial B, z_1\in B, t_1\leq r_0$,
this gives the bound of
    $$
    C|\partial A||\partial B||B| V^{-2} \m r_0 \vep^{-1} (1+\vep)^{r_0}\E|B(r_0)|
    \leq C |\partial A||\partial B||B| V^{-2} \m r_0 (\E|B(r_0)|)^2 \, ,
    $$
where we have used Theorem \ref{lowerball}.
\end{proof}

\begin{lemma}
\label{sedgefirstmom} For any positive $\delta>0$ and $\beta>0$,
    $$
    \eab S_{j_x,j_y,r_0}(x,y) \geq (1-8\beta^{1/2}) V^{-1} \m (\E|B(r_0)|)^2(|\partial A| - \eps^{-1})(|\partial B| - \eps^{-1}) - \hbox{{\rm Err}} \, ,
    $$
where
    $$
    \hbox{{\rm Err}} \leq C |\partial A||\partial B| V^{-1} \m (\E|B(r_0)|)^2 \Big[
    r_0(|A|+|B|)V^{-1}+ \alpham^{1/2}(1+ \eps |\partial A|)  \Big]\, .
    $$
\end{lemma}
\begin{proof} We have that
    \be\label{firstmom.expand}
    \eab S_{j_x,j_y,r_0}(x,y) = \sum_{(a,b) \in \partial A \times \partial B}
    \sum_{(u,u')} \pab \big (\{a \lrefmpx u \} \and \{b \lrefmpy u' \off B_x(j_x+r_0)\} \big ) \, ,
    \ee
because the additional requirement that $a$ and $b$ are pivotals in the
definitions of $\{a \lrefmpx u\}$ and $\{b \lrefmpy u'\}$ implies that they are unique in
$\partial A \times \partial B$, so no pair $(a,b)$ is overcounted in the sum. We define $B_{\partial A}(r_0; A \cup B) = \cup_{a' \in \partial A} B_{a'}(r_0; A \cup B)$.
We condition on $B_{\partial A}(r_0; A \cup B)=H$ for an admissible $H$ (that is, any
$H$ that has positive probability and $a \lrefmpx u$ off $B$ occurs in it). Each summand in (\ref{firstmom.expand}) equals
    $$
    \sum_{H} \pab\big( B_{\partial A}(r_0; A\cup B)=H\big) \pab\Big(b \lrefmpy u' \off H \mid B_{\partial A}(r_0; A\cup B)=H\Big) \, ,
    $$
%\todo{Check this equality: can pivotality be determined from $H$?}
and we have
    $$
    \pab\Big(b \lrefmpy u' \off H \mid B_{\partial A}(r_0; A\cup B)=H\Big) = \pb\big(b \lrefmpy u' \off A \cup H\big) \, ,
    $$
because in both sides the status of the edges touching $A \cup H$ cannot change the occurrence of the event.
This gives that
    $$
    \eab S_{j_x,j_y,r_0}(x,y) = \sum_{(a,b) \in \partial A \times \partial B}
    \sum_{(u,u')} \sum_{H} \pab(B_{\partial A}(r_0; A \cup B)=H)
    \pb (b \lrefmpy u' \off A \cup H)\, .
    $$
Now, by Claim \ref{offmethod},
    \begin{eqnarray*}
    \pb (b \lrefmpy u' \off A\cup H)
    &\geq& \pb(b \lrefmpy u') - \pb(b \lrefmp u' \onlyon A\cup H) \, ,
    \end{eqnarray*}
where in the last term we have dropped the requirement that $y$ is
pivotal (which only increases the probability).
Hence by summing on $H$ we get,
    $$
    \eab S_{j_x,j_y,r_0}(x,y) \geq \sum_{(a,b)\in \partial A \times \partial B} \sum_{(u,u')}
    \pab(a \lrefmpx u) \pb(b \lrefmpy u') - S_2 \, ,
    $$
where
    \be\label{defines2}
    S_2 = \sum_{(a,b) \in \partial A \times \partial B} \sum_{(u,u')} \sum_{H}
    \pab(B_{\partial A}(r_0; A \cup B)=H) \pb(b \lrefmp u' \onlyon A \cup H) \, .
    \ee
As before,
    $$
    \pab(a \lrefmpx u) = \pa(a \lrefmpx u \off B) \, ,
    $$
since the status of the edges touching $A \cup B$ in both sides does not matter.
Claim \ref{offmethod} again gives that
    \begin{eqnarray*}
    \pa(a \lrefmpx u \off B)
    &\geq& \pa(a \lrefmpx u) - \pa(a \lrefmp u \onlyon B) \, ,
    \end{eqnarray*}
and so we may further expand
    $$
    \eab S_{j_x,j_y,r_0}(x,y) \geq S_1 - S_2 - S_3 \, ,
    $$
with
    $$
    S_1 = \sum_{(a,b)\in \partial A \times \partial B} \sum_{(u,u')} \pa(a \lrefmpx u)\pb(b \lrefmpy u') \, ,
    $$
    $$
    S_3 = \sum_{(a,b)\in \partial A \times \partial B} \sum_{(u,u')} \pa(a \lrefmp u \onlyon B) \pb(b \lrefmp u') \, , $$
and $S_2$ is defined in (\ref{defines2}). Lemma \ref{sedgefirstmommain} gives the required lower bound on $S_1$ which yields the positive contribution in the assertion of this lemma. We now bound $S_2$ and $S_3$ from above, starting with $S_3$. If $a \lrefmp u \onlyon B$, then either $a \stackrel{2\mnot}{\lrfill} u$ or there exists $z_1 \in B$ and $t_1 \in [\mnot, r_0]$ such that
    $$
    \{a \stackrel{=t_1}{\lrfill} z_1\} \circ \{z_1 \stackrel{r_0-t_1}{\lrfill} u\} \quad \hbox{{\rm or}} \quad \{ a \stackrel{r_0-t_1}{\lrfill} z_1\} \circ \{z_1 \stackrel{=t_1}{\lrfill} u\} \, ,
    $$
Indeed, let $\gamma$ be the lexicographically first shortest path between $a$ and $u$. If $|\gamma| \leq 2\mnot$, then $a \stackrel{2\mnot}{\lrfill} u$, otherwise $|\gamma| \in [2\mnot,r_0]$ and we take $z_1$ to be the first vertex in $B$ visited by $\gamma$ and $t$ is such that
$\gamma(t)=z_1$. If $t \geq \mnot$, then we put $t_1=t$ and otherwise we put
$t_1 = |\gamma|-t$. In any case $t_1 \in [\mnot, r_0]$. When $t\geq \mnot$, the witness for $a \stackrel{=t_1}{\lrfill} z_1$
is the set of open edges of the path $\gamma[0,t]$ together with all the closed edges of the graph and the witness for $z_1 \stackrel{r_0-t_1}{\lrfill} u$
are the open edges of $\gamma[t,|\gamma|]$. The case $t \leq \mnot$ is done similarly.
We get that
    \begin{eqnarray*}
    S_3 &\leq& \sum_{\substack{ (a,b) \in \partial A \times \partial B \\ (u,u')}} \pa(a \stackrel{2\mnot}{\lrfill} u) \pb(b \lrefmp u') \\ &\ & \quad + \sum_{\substack{(a,b)\in \partial A \times \partial B \\ (u,u'), z_1 \in B, t_1 \in [\mnot,r_0]}} \pa(a \stackrel{=t_1}{\lrfill} z_1)\pa(z_1 \stackrel{r_0-t_1}{\lrfill} u) \pb(b \lrefmp u') \\
    &\ & \quad + \sum_{\substack{(a,b)\in \partial A \times \partial B \\ (u,u'), z_1 \in B, t_1 \in [\mnot,r_0]}} \pa(a \stackrel{r_0-t_1}{\lrfill} z_1)\pa(z_1 \stackrel{=t_1}{\lrfill} u) \pb(b \lrefmp u') \, .
    \end{eqnarray*}
For the first term we bound $\pb(b \lrefmp u')\leq CV^{-1} \E|B(r_0)|$ by Lemma \ref{uniformconnbd} and sum over everything to get a contribution bounded by
\begin{eqnarray} \label{annoyingterm} C |\partial A| |\partial B| m V^{-1} \E|B(r_0)| \mnot &\leq& C |\partial A||\partial B| V^{-1} m (\E|B(r_0)|)^2 \big [ \mnot (\E|B(r_0)|)^{-1} \big ]
\nn \\  &\leq& C \alpham^{1/2} |\partial A||\partial B| V^{-1} m (\E|B(r_0)|)^2 \, ,\end{eqnarray}
by our choice of $r_0$ in (\ref{rzerocond}), our assumptions $\alpham\geq (\eps^3 V)^{-1/2}$ in (\ref{alphamcond}) and $\mnot=o(\eps^{-1})$ and Corollary \ref{r0ball}. This fits in the second term of Err in the assertion of the lemma. We bound the second and third terms using Lemma \ref{sedgefirstmomerror2} giving an upper bound of
    $$
        C |\partial A||\partial B||B| V^{-2} \m r_0 (\E|B(r_0)|)^2 \, ,
    $$
which fits in the first term of Err in the assertion of the lemma.

We proceed to bound $S_2$ in \eqref{defines2} from above. As before, if $b \lrefmp u' \onlyon H\cup A$, then either $b \stackrel{2\mnot}{\lrfill} u'$ or there exists $z_1\in H\cup A$ and $t_1\in[\mnot,r_0]$ such that
    \be\label{disjointz1}
    \{b \stackrel{=t_1}{\lrfill} z_1\} \circ \{z_1 \stackrel{r_0-t_1}{\lrfill} u'\} \quad \hbox{{\rm or}} \quad \{ b \stackrel{r_0-t_1}{\lrfill} z_1\} \circ \{z_1 \stackrel{=t_1}{\lrfill} u'\} \, .
    \ee
The case $b \stackrel{2\mnot}{\lrfill} u'$ is handled as before and gives a contribution of $C |\partial A| |\partial B| m V^{-1} \E|B(r_0)| \mnot$ which by (\ref{annoyingterm}) again fits the second term of Err. To handle the other cases, let us first sum the contribution to $S_2$ due to \eqref{disjointz1} over $z_1 \in H$. We use the BK-Reimer inequality and change order of summation to bound this contribution to $S_2$ by
    \begin{eqnarray*}
    && \sum_{\substack{(a,b) \in \partial A \times \partial B \\ (u,u'),z_1,t_1 \in [\mnot,r_0]} }
     \pab(a \lrefmp u,  \exists \, a' \in \partial A \hbox{ such that } a' \lref z_1) \pb(b \stackrel{=t_1}{\lrfill} z_1) \pb(z_1 \stackrel{r_0-t_1}{\lrfill} u') \\
     &&\quad+ \sum_{\substack{(a,b) \in \partial A \times \partial B \\ (u,u'),z_1,t_1 \leq r_0} }
     \pab(a \lrefmp u, \exists \, a' \in \partial A \hbox{ such that } a' \lref z_1) \pb(b \stackrel{r_0-t_1}{\lrfill} z_1) \pb(z_1 \stackrel{=t_1}{\lrfill} u') \, .
    \end{eqnarray*}
Now, if $a \lrefmp u$ and there exists $a' \in \partial A$ with $a' \lref z_1$, then either $a \lrefmp u, a \lref z_1$ or there exists $a' \in \partial A$ such that $\{a \lrefmp u\} \circ \{a' \lref z_1\}$. Hence we may bound this from above by $(I) + (II)$ where
    \begin{eqnarray*} (I) &=& \sum_{\substack{(a,b) \in \partial A \times \partial B \\ (u,u'),z_1,t_1 \in [\mnot,r_0]} }
     \pab(a \lrefmp u, a \lref z_1) \pb(b \stackrel{=t_1}{\lrfill} z_1) \pb(z_1 \stackrel{r_0-t_1}{\lrfill} u') \\
     &&\quad+ \sum_{\substack{(a,b) \in \partial A \times \partial B \\ (u,u'),z_1,t_1 \leq r_0} }
     \pab(a \lrefmp u, a \lref z_1) \pb(b \stackrel{r_0-t_1}{\lrfill} z_1) \pb(z_1 \stackrel{=t_1}{\lrfill} u') \, ,
     \end{eqnarray*}
and
    \begin{eqnarray*} (II) &=& \sum_{\substack{(a,b) \in \partial A \times \partial B, a' \in \partial A \\ (u,u'),z_1,t_1 \in [\mnot,r_0]} }
     \pab(a \lrefmp u) \pab(a' \lref z_1) \pb(b \stackrel{=t_1}{\lrfill} z_1) \pb(z_1 \stackrel{r_0-t_1}{\lrfill} u') \\
     &&\quad+ \sum_{\substack{(a,b) \in \partial A \times \partial B, a' \in \partial A \\ (u,u'),z_1,t_1 \leq r_0} }
     \pab(a \lrefmp u) \pab(a' \lref z_1) \pb(b \stackrel{r_0-t_1}{\lrfill} z_1) \pb(z_1 \stackrel{=t_1}{\lrfill} u') \, .
    \end{eqnarray*}
Lemma \ref{sedgefirstmomerror1} readily gives that $(I)   \leq \alpham^{1/2} |\partial A||\partial B| V^{-1} \m (\E|B(r_0)|)^2$ which fits into the second term of Err. Lemma \ref{sedgefirstmomerror3} gives that
$ (II) \leq  \alpham^{1/2} \eps |\partial A|^2|\partial B| V^{-1} m (\E|B(r_0)|)^2$
which fits in the second term of Err. We sum the contribution to $S_2$ due to \eqref{disjointz1} over $z_1 \in A$ and bound it from above by
    \eqan{
    &\sum_{\substack{ (a,b) \in \partial A \times \partial B \\ (u,u'), z_1\in A, t_1 \in [\mnot,r_0]}}
    \pab(a \lrefmp u) \pb(b \stackrel{=t_1}{\lrfill} z_1)\pb(z_1 \stackrel{r_0-t_1}{\lrfill} u')
    \leq C
    |\partial A||\partial B||A| V^{-2} \m r_0 (\E|B(r_0)|)^2 \, ,\nn
    }
by an appeal to Lemma \ref{sedgefirstmomerror2}. This fits in the first term of Err and concludes our proof.
\end{proof}

%
%
%Then there exist vertices $a\in\partial A$, $b\in \partial B$ and $z_1,z_2$ and $t_1,t_2 \in [\mnot,r_0]$ such that the event
%    $$
%    \{a \stackrel{=t_1}{\lrfill} z_1\} \circ \{z_1 \stackrel{r_0-t_1}{\lrfill} u\}
%    \circ \{z_1 \stackrel{2r_0 -t_1-t_2}{\lrfill} z_2\} \circ \{z_2 \stackrel{r_0-t_2}{\lrfill} u'\}
%    \circ \{b \stackrel{=t_2}{\lrfill} z_2 \off B_x(t_1)\} \, ,
%    $$
%occurs, see Figure \ref{fig.connectearly}, or
%$$ \{a \stackrel{\mnot}{\lrfill}{z_1} \} \circ \{z_1 \stackrel{P[\mnot,r_0]}{\lrfill} u\}
%    \circ \{z_1 \stackrel{2r_0-t_2}{\lrfill} z_2\} \circ \{z_2 \stackrel{r_0-t_2}{\lrfill} u'\}
%    \circ \{b \stackrel{=t_2}{\lrfill} z_2 \} \, ,
%$$
%or
%$$ \{a \stackrel{\mnot}{\lrfill}{z_1} \} \circ \{z_1 \stackrel{P[\mnot,r_0]}{\lrfill} u\}
%    \circ \{z_1 \stackrel{2r_0}{\lrfill} z_2\} \circ \{z_2 \stackrel{P[\mnot,r_0]}{\lrfill} u'\}
%    \circ \{b \stackrel{\mnot}{\lrfill} z_2 \} \, .
%$$
%We use [fill in refs here, using triangle] and bound the first by
%$$ |\partial A| |\partial B| V^{-1} m \sum_{t_1,t_2} (1+\eps)^{t_1+t_2} \Big [ {C \over m} + {\eps^{-3}(1+\eps)^{4r_0-2(t_1+t_2)} \over V} \Big ] \, ,$$
%which is bounded above by
%$$ {C \over \log \eps^3 V} |\partial A||\partial B| V^{-1} m (\E|B(r_0)|)^2 \, .$$
%We bound the second by [fill refs here]. Both two sums are smaller.
%\end{proof}

\begin{lemma} \label{sedgesecondmom}
The following bound holds:
    $$
    \eab S_{j_x,j_y,r_0}(x,y)^2  {\bf 1}_{\{ \partial A \stackrel{2r_0}{\not \lr} \partial B \}} \leq Q_1 + Q_2 + Q_3 \, ,$$
    where
    \begin{eqnarray*}
    Q_1 &=& \big (1 +O(\alpham+ \eps \mnot)\big ) V^{-2}\m^2 (\E|B(r_0)|)^4 |\partial A|^2 |\partial B|^2 \, ,\\
    Q_2 &=& C V^{-2}\m^2 \eps^{-1} (\E|B(r_0)|)^4 |\partial A| |\partial B|(|\partial A|  +|\partial B|) \, , \\
    Q_3 &=& CV^{-2}\m^2\eps^{-2}(\E|B(r_0)|)^4 |\partial A||\partial B| \, .
    \end{eqnarray*}
\end{lemma}

\begin{proof} Assume that $(u_1,u_1')$ and $(u_2,u_2')$ are two edges and let $a,a_1,a_2$ be vertices in $\partial A$ and $b,b_1,b_2$ vertices in $\partial B$. Define
    \eqn{
    \T(u_1,u_2,a_1,a_2)=\{a_1 \lrefmp u_1\}\circ \{a_2 \lrefmp u_2\},
    \quad
    \T(u_1,u_2,a)=\{ a \lrefmp u_1 \} \cap \{a \lrefmp u_2\} \, .
    }
%%%% \todo{Don't you want the paths in $\T(u_1,u_2,a)$ to be long?} --- Yes, you are right.
We define $\T(u_1',u_2',b_1,b_2)$ and $\T(u_1',u_2',b)$ in a similar fashion.

Now, if $(u_1,u_1')$ and $(u_2,u_2')$ are counted in $S_{j_x,j_y,r_0}(x,y)^2  {\bf 1}_{\{ \partial A \stackrel{2r_0}{\not \lr} \partial B \}}$, then one of the following events must occur off $A\cup B$:
\begin{enumerate}
\item There exists $a_1,a_2,b_1,b_2$ such that $\T(u_1,u_2,a_1,a_2) \circ \T(u_1',u_2',b_1,b_2)$ occurs,
\item There exists $a_1,a_2,b$ such that $\T(u_1,u_2,a_1,a_2) \circ \T(u_1',u_2',b)$ occurs, or the symmetric case $\T(u_1,u_2,a) \circ \T(u_1',u_2',b_1,b_2)$.
\item There exists $a,b$ such that $\T(u_1,u_2,a) \circ \T(u_1',u_2',b)$ occurs.
\end{enumerate}
See Figure \ref{fig.secondmom}. Observe that the disjoint occurrence of the events
is implied since $\partial A \stackrel{2r_0}{\not \lr} \partial B$.
We now sum the probability of these events over $(u_1,u_1'),(u_2,u_2')$ and this
gives us three terms which we will bound by $Q_1, Q_2$ and $Q_3$, respectively. By Lemma \ref{uniformconnbd} and the BK inequality,
    \be
    \label{sedgesecondmom.bound1}
    \pab(\T(u_1,u_2,a_1,a_2)) \leq {\big (1 +O(\alpham+ \eps \mnot)\big )
    (\E|B(r_0)|)^2 \over V^2} \, ,
    \ee
whence
    $$
    \sum_{\substack{a_1,a_2,b_1,b_2,\\ (u_1,u_1'), (u_2,u_2')}} \pab(\T(u_1,u_2,a_1,a_2) \circ \T(u_1',u_2',b_1,b_2))
    \leq \big (1 +O(\alpham+\eps \mnot)\big ) V^{-2} \m^2 (\E|B(r_0)|)^4 |\partial A|^2 |\partial B|^2 \, ,
    $$
which equals $Q_1$. To bound
the probability of (2), if  $\T(u_1',u_2',b)$ occurs, then, as before, there exists a vertex
$z_1$ and $t_1 \in [\mnot,r_0]$ such that
    $$
    \{b \stackrel{=t_1}{\lrfill} z_1\} \circ \{z_1 \stackrel{r_0-t_1}{\lrfill} u_1'\} \circ \{z_1 \stackrel{r_0-t_1}{\lrfill} u_2'\} \, ,
    $$
or there exists $z_1$ such that
    $$
    \{b \stackrel{\mnot}{\lrfill} z_1\} \circ \{z_1 \lrefmonep u_1'\} \circ \{z_1 \lrefmonep u_2'\} \, .
    $$
\begin{figure}
\begin{center}
\includegraphics{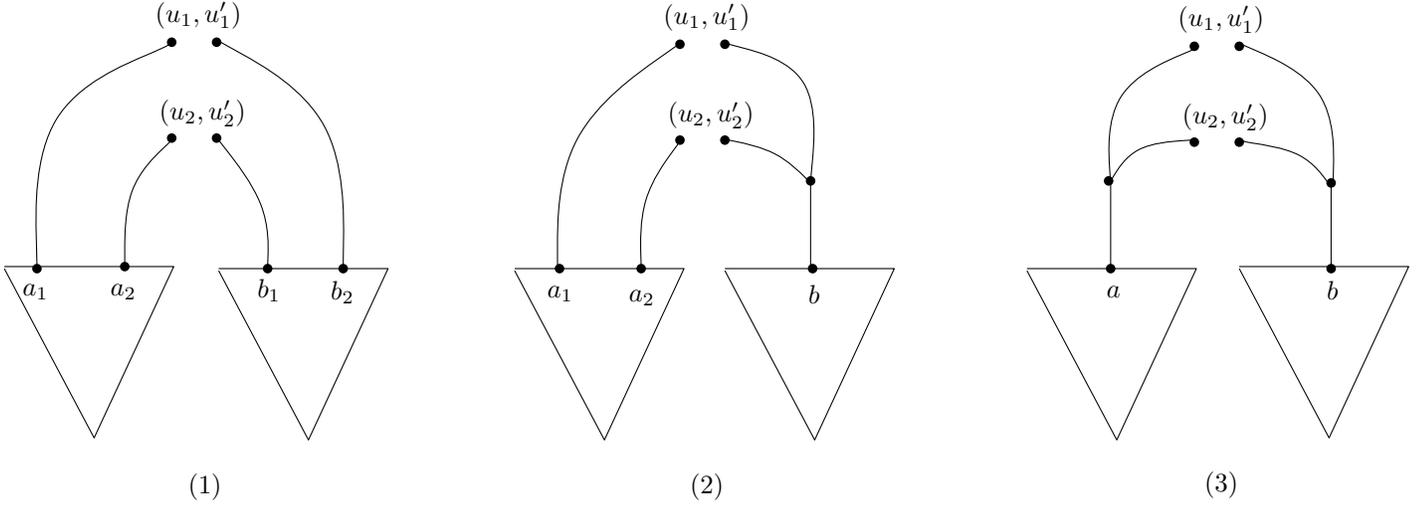}
\caption{The three contributions to the second moment of $S_{j_x,j_y,r_0}(x,y)$. The main contribution comes from (1).}
\label{fig.secondmom}
\end{center}
\end{figure}
Hence, the BK-Reimer inequality gives that
    \begin{eqnarray*} \sum_{u_1',u_2'} \pab(\T(u_1',u_2',b)) &\leq& \sum_{\substack{u_1',u_2', z_1,\\ t_1 \in [\mnot,r_0]}} \pab(b \stackrel{=t_1}{\lrfill} z_1) \pab(z_1 \stackrel{r_0-t_1}{\lrfill} u_1') \pab(z_1 \stackrel{r_0-t_1}{\lrfill} u_2') \\ &&\qquad +\sum_{u_1',u_2',z_1} \pab(b \stackrel{\mnot}{\lrfill} z_1)\pab(z_1 \lrefmonep u_1')\pab(z_1 \lrefmonep u_2') \, .
    \end{eqnarray*}
We estimate the first sum by summing on $u_2', u_1'$ then on $z_1,t_1$ using Lemma \ref{upperball} to get a bound of
    $$
    C \eps^{-3}(1+\eps)^{2r_0} \leq C \eps^{-1} (\E |B(r_0)|)^2 \, ,
    $$
by Theorem \ref{lowerball}, and the second sum is bounded by $C\mnot (\E |B(r_0)|)^2$
which of lower order since $\eps \mnot = o(1)$ by (\ref{epscond}). We use the BK
inequality
and (\ref{sedgesecondmom.bound1}) to bound the contribution due to the first event in (2) by
    $$
    \sum_{\substack{a_1,a_2,b, (u_1,u_1'), \\ (u_2,u_2')}} \pab(\T(u_1,u_2,a_1,a_2)) \pab(\T(u_1',u_2',b)) \leq C V^{-2} m^2 \eps^{-1} (\E|B(r_0)|)^4 |\partial A|^2 |\partial B| \, .
    $$
The symmetric in (2) obeys the same bound with the roles of $|\partial A|$ and $|\partial B|$ reversed. This contribution equals $Q_2$. To bound the contribution due to (3),
we note that
    $$
    \sum_{a,b, (u_1,u_1'), (u_2,u_2')} \pab(\T(u_1,u_2,a)) \pab(\T(u_1',u_2',b)) \, ,
    $$
is bounded using the BK-Reimer inequality by the three sums
    \begin{align*}
    \sum_{\substack{a,b, (u_1,u_1'), (u_2,u_2') \\ z_1, z_2, t_1, t_2 \in [m_0,r_0]}} \!\!\!\!\!\!\!\!\!\!\!\!  \pab(a \stackrel{=t_1}{\lrfill} z_1) \pab(z_1 \stackrel{r_0-t_1}{\lrfill} u_1) \pab(z_1 \stackrel{r_0-t_1}{\lrfill} u_2) \pab(b \stackrel{=t_2}{\lrfill} z_2) \pab(z_2 \stackrel{r_0-t_2}{\lrfill} u_1') \pab(z_2 \stackrel{r_0-t_2}{\lrfill} u_2') \, ,
    \end{align*}
    \begin{align*} \sum_{\substack{a,b, (u_1,u_1'), (u_2,u_2') \\ z_1, z_2, t_1 \in [m_0,r_0]}} \!\!\!\!\!\!\!\!\!\!\!\!   \pab(a \stackrel{=t_1}{\lrfill} z_1) \pab(z_1 \stackrel{r_0-t_1}{\lrfill} u_1) \pab(z_1 \stackrel{r_0-t_1}{\lrfill} u_2) \pab(b \stackrel{\mnot}{\lrfill} z_2) \pab(z_2 \lrefmonep u_1') \pab(z_2 \lrefmonep u_2') \, ,
    \end{align*}
and
    \begin{align*}
    \sum_{\substack{a,b, (u_1,u_1')\\ (u_2,u_2'), z_1, z_2}} \!\!\!\!\!\!\!\!\!\!\!\! \pab(a \stackrel{\mnot}{\lrfill} z_1) \pab(z_1 \lrefmonep u_1) \pab(z_1 \lrefmonep u_2) \pab(b \stackrel{\mnot}{\lrfill} z_2) \pab(z_2 \lrefmonep u_1') \pab(z_2 \lrefmonep u_2') \, .
    \end{align*}
To bound the first sum, we use Lemma \ref{moreunifconnbd} and Lemma \ref{upperball}
to bound $\pab(a \stackrel{=t_1}{\lrfill} z_1) \leq CV^{-1}(1+\eps)^{t_1}$ and $\pab(b \stackrel{=t_2}{\lrfill} z_2)
\leq CV^{-1}(1+\eps)^{t_2}$. We then use Lemma \ref{squarediagram} and Lemma \ref{upperball}
to sum over $z_1,z_2,(u_1,u_1'),(u_2,u_2')$. This gives us an upper bound of
    $$
    CV^{-2}|\partial A||\partial B| m^2 \eps^{-6} (1+\eps)^{4r_0} + C|\partial A||\partial B|V^{-1} m^2 \mnot \alpham
    \sum_{t_1,t_2 \in  [m_0,r_0]} (1+\eps)^{t_1+t_2}
    $$
    $$\leq CV^{-2} |\partial A||\partial B| m^2  \eps^{-2} (\E|B(r_0)|)^4 \, ,
    $$
where the last inequality is due to Theorem \ref{lowerball} and our choice of $r_0$ in (\ref{rzerocond}). This is contained in $Q_3$. To bound the second sum, we use Lemma \ref{unifconnbdoff} to bound each of the last two terms by $C V^{-1} \E|B(r_0)|$. We then sum over $(u_1,u_1')$ and $(u_2,u_2')$ using Lemma \ref{upperball}. We then sum over $z_1,z_2$ using Lemma \ref{upperball} and finally over $a,b,t_1$ to get that this sum is at most
$$ C |\partial A| |\partial B| V^{-2} m^2 (\E|B(r_0)|)^4 \eps^{-1} \mnot \, ,$$
which is contained in $Q_3$ since $\eps \mnot = o(1)$. For the third sum we use Lemma \ref{unifconnbdoff} four times, and then sum over everything to get a bound of
$$ C |\partial A||\partial B| V^{-2} m^2 (\E|B(r_0)|)^4 \mnot^2 \, ,$$
which is also contained in $Q_3$, concluding our proof.
\end{proof}

%We say that $B_x(j_x)$ and $B_y(j_y)$ are $M$-good if the following hold:
%\begin{enumerate}
%    \item $|\partial B_x(j_x)| \geq c\eps^{-1}\log M$ and $|\partial B_y(j_y)|\geq c\eps^{-1}\log M$ where $c>0$ is the constant from Lemma \ref{bdrynotsmall} and
%    \item $\E \big [  S_{j_x,j_y,r_0}(x,y) {\bf 1}_{\{ \partial B_x(j_x) \stackrel{2r_0}{\lrlong} \partial B_y(j_y)  \}}\, \big |
%    \, B_x(j_x), B_y(j_y) \big ] = o\big (V^{-1} \eps^{-2} (\E|B(r_0)|)^2 m \big ) \, .$
%\end{enumerate}

\noindent {\bf Proof of Theorem \ref{largeclose.getsedge}.}
Instead of conditioning on $J(x)=j_x, J(y)=j_y$ we condition on $B_x(j_x)=A$ and $B_y(j_y)=B$ such that the event $\A(x,y,j_x,j_y,r_0,\beta,k)$ holds. This is a stronger conditioning and implies the assertion of the theorem.

By requirement (2) of $\A(x,y,j_x,j_y,r_0,\beta,k)$ and our choice of parameters
$$ r_0(|A|+|B|)V^{-1} \leq V^{-1} \eps^{-3} \log(\eps^3 V) \e^{3M} \leq (\log \eps^3 V)^{-1} \, , $$
and
$$ {\alpham^{1/2} \eps |\partial A|} \leq {\e^{3M} \alpham^{1/2}} \leq \alpham^{1/4} \, .$$
Hence the error term in Lemma \ref{sedgefirstmom} is at most
    $$
    \hbox{{\rm Err}} \leq  C \big [ (\log \eps^3 V)^{-1} + \alpham^{1/4} \big ] |\partial A||\partial B|V^{-1} m (\E|B(r_0)|)^2 \, .
    $$
Lemma \ref{sedgefirstmom} together with requirement (5) in the definition of $\A(x,y,j_x,j_y,r_0,\beta,k)$ and our choice of $\beta$ in \eqref{parameter-def-1} (in particular, that $\beta \ll \alpham^{1/4} \wedge (\log \eps^3 V)^{-1}$ by \eqref{Mchoice}) give
    $$
    \eab \Big [  S_{j_x,j_y,r_0}(x,y) {\bf 1}_{\{ \partial A \stackrel{2r_0}{\not\lr} \partial B  \}} \Big ]
    \geq (1-C\beta^{1/2}) V^{-1} \m (\E|B(r_0)|)^2 |\partial A||\partial B| \, .
    $$
Since $|\partial A|$ and $|\partial B|$ are at least $\e^{k/4}\eps^{-1}$,
    $$
    \eps^{-1}|\partial A|^2|\partial B| +
    \eps^{-1}|\partial A|^2|\partial B| + \eps^{-2} |\partial A||\partial B| \leq C \e^{-k/4} |\partial A|^2|\partial B|^2 \, ,
    $$
hence, by Lemma \ref{sedgesecondmom} and our choice of parameters,
    $$ \eab \Big [  S_{j_x,j_y,r_0}(x,y)^2 {\bf 1}_{\{ \partial A \stackrel{2r_0}{\not\lr} \partial B  \}} \big ] \leq \big ( 1+ O(\e^{-k/4}) \big ) V^{-2} m^{2} (\E|B(r_0)|)^4 |\partial A|^2|\partial B|^2 \, .$$
We conclude that
    $$
    \pab \big ( S_{j_x,j_y,r_0}(x,y) \geq 2\beta^{1/2} V^{-1} \m (\E|B(r_0)|)^2 |\partial A||\partial B| \big ) \geq 1 - O(\beta^{1/2}) \, ,
    $$
where we used the fact that $\e^{-k/4} = o(\beta)$ and \eqref{X>a-prob-bd}. This concludes our proof.
\qed

\section{Proofs of main theorems} \label{sec-sprinkling}

\subsection{Proof of Theorem \ref{mainthmgeneral}} \label{sec-improved-sprinkling-performed} In Section \ref{sec-unif-conn-NBW} we already proved Theorem \ref{mainthmgeneral}(a) hence we may assume that the finite triangle
condition (\ref{finitetriangle}) holds and focus on part (b) of the theorem. Since $|\C_1|\leq k_0+Z_{\geq k_0}$ where $k_0$ is from Theorem \ref{clustertail}, Lemma \ref{lem-conc-Z-I} immediately gives that $|\C_1| \leq (2+o(1)) \eps V$
\whp, showing the required upper bound on $|\C_1|$ --- note
that this argument only uses the finite triangle condition hence it is valid
for any $\epsm$ satisfying $\epsm \gg V^{-1/3}$ and $\epsm = o(1)$. For the lower
bound we will additionally assume, as part (b) requires, that $\epsm = o(\mnot^{-1})$ and show that
    \be
    \label{mainthm.toshow}
    \prob_p \big( |\C_1| \geq (2-o(1)) \eps V\big) = 1-o(1) \, .
    \ee
This establishes part (b) of Theorem \ref{mainthmgeneral}. Recall that
$p=p_c(1+\eps)$ is our percolation probability, let $\theta>0$ be an
arbitrary small constant and put $p_2, p_1$ to satisfy
    $$
    p_2 = \theta \eps /m \, , \qquad p_c(1+\eps) = p_1 + (1-p_1)p_2 \, ,
    $$
so that $p_c(1+(1-\theta)\eps) \leq p_1 \leq p_c(1+\eps)$ since $p_c \geq 1/\m$. Denote
by $G_{p_1}$ and $G_{p_2}$ two independent percolation instances of $G$ with
parameters $p_1$ and $p_2$, respectively. The sprinkling procedure relies
on the fact that $G_p$ is distributed as
$G_{p_1} \cup G_{p_2}$.  We first apply Theorem \ref{alphapairs} to $G_{p_1}$ and deduce
that for $M,r$ defined in (\ref{Mchoice}) and $r_0$ defined in (\ref{rzerocond}),
    \be
    \label{mainthm.alphapairsapp}
    \prob _{p_1} \big (\Park \geq (1-3\theta) 4\eps^2 V^2 \big ) \geq 1-o(1) \, .
    \ee
%Note that this event implies that there at least $(2-c\alpha-c\theta)\eps V$ vertices $x$ for which
%
%some positive constants $M,\alpha$ which depend only on $\theta$ ($M$ is large and $\alpha$ is small) and a sequence $\Km\to \infty$ such that if $r=M\eps^{-1}$, then with probability at least $1-\gamma_2$ there exists at least $(2-3\theta)\eps V$ vertices $x$ for which $|\C(x)| \geq \Km\eps^{-2}$ and at least $(2-5\theta)\eps^2 V^2$ pairs of $(\alpha,r,\Km)$-good vertices.
%
Now we wish to show that when we ``sprinkle'' this configuration in $G_{p_1}$,
that is, when we add to the configuration independent $p_2$-open edges,
most of these vertices join together to form one cluster of size roughly $2 \eps V$.
To make this formal, given $G_{p_1}$, we construct an auxiliary simple graph $H$ with vertex set
    $$
    V(H) = \big \{ x\in G_{p_1} \colon |\C(x)|\geq (\eps^3 V)^{1/4} \eps^{-2} \big \} \, ,
    $$
and edge set
    $$
    E(H) = \big \{ (x,y) \in V(H)^2 \colon x,y \hbox{ are $(r,r_0)$-good} \big \} \, .
    $$
Thus, using Lemma \ref{lem-conc-Z-I} with $k_0=\eps^{-2} (\eps^3 V)^{1/4}$ and (\ref{mainthm.alphapairsapp}), with probability at least $1-o(1)$,
    \be\label{auxgraph}
        |V(H)| =  (2+o(1)) \eps V \, , \quad\qquad |E(H)| \geq (1-3\theta) 4 \eps^2 V^2 \, .
    \ee

Denote $v=|V(H)|$ and write $x_1, \ldots, x_v$ for the vertices in $G_{p_1}$ corresponding
to those of $H$. Given $G_{p_1}$ for which the event in (\ref{auxgraph}) occurs,
we will show that \whp\ in $G_{p_1} \cup G_{p_2}$ there is no way to partition the set of vertices into  $M_1 \uplus M_2 = \{x_1,\ldots, x_v\}$ with $|M_1|\geq 3\theta v$ and $|M_2| \geq 3\theta v$ such that there is no open path in
$G_{p_1} \cup G_{p_2}$ connecting a vertex in $M_1$ with a vertex in $M_2$.
This implies that \whp\ the largest connected component in $G_{p_1} \cup G_{p_2}$ is of size at least $(1-3\theta)v$.

To show this, we first claim that the number of
such partitions is at most $2^{3 (\eps^3 V)^{3/4}}$
since $|\C(x_i)| \geq (\eps^3 V)^{1/4} \eps^{-2}$.
Secondly, given such a partition, we
claim that the number of edges $(u,u')$ such that $u \in M_1$ and $u'\in M_2$ (note that, by definition,
these edges must be $p_1$-closed) is at least $\e^{-40M}(\log M)^{-1}\theta \eps^2 V m$.
To see this, we consider the set of edges in $H$ for which both sides lie in either
$M_1$ or $M_2$ (more precisely, the vertices of $H$ corresponding to $M_1$ and $M_2$). This number is at most
    $$
    {M_1^2 + M_2^2} \leq (3\theta v)^2 + (1-3\theta)^2 v^2 \leq {(1-5\theta)v^2},
    $$
where we used the fact that $\theta>0$ is a small enough constant, $M_1+M_2=v$ and both $M_1$ and $M_2$ are in $[3\theta v, (1-3\theta)v]$.
By \eqref{auxgraph}, the number of edges in $H$ such that one end is in $M_1$ and the other in $M_2$ is at least $\theta \eps^2 V^2$. In other words, there are at least $\theta \eps^2 V^2$ pairs $(x,y)\in M_1\times M_2$ such that
$S_{2r+r_0}(x,y) \geq  (\log M)^{-1} V^{-1}\m \eps^{-2} (\E|B(r_0)|)^2$. Note that is a large number due to our condition (\ref{rzerocond}). In total, we counted at least
$\theta \eps^2 V^2 \cdot (\log M)^{-1} V^{-1} \m \eps^{-2} (\E|B(r_0)|)^2$ edges $(u,u')$ and no edge is counted more than
$|B_u(2r+r_0)|\cdot|B_{u'}(2r+r_0)|$ times, which is at most $\e^{40M} \eps^{-2} (\E|B(r_0)|)^2$
by the definition of $S_{2r+r_0}(x,y)$ and the second claim follows.

Hence, if $|\C_1| \leq (1-3\theta)v$, then there exists such a partition in which all of the
above edges $(u,u')$ are $p_2$-closed. By the two claims above, the probability of this is at most
    $$
    2^{3 (\eps^3 V)^{3/4}} (1-p_2)^{\e^{-40M} (\log M)^{-1} \theta \m \eps^2 V} \leq 2^{3 (\eps^3 V)^{3/4}}
    \e^{-\e^{-40M} (\log M)^{-1} \theta^2 \eps^3 V} = o(1) \, ,
    $$
since $p_2 = \theta \eps / \m$ and by our choice of parameters in (\ref{Mchoice}) and (\ref{epscond}). This concludes the proof of (\ref{mainthm.toshow}) since $\theta>0$ was arbitrary and establishes the required estimate on $|\C_1|$ of Theorem \ref{mainthmgeneral} (b).

We now use \eqref{mainthm.toshow} to show the required bounds on $\expec|\C(0)|$ and $|\C_2|$.
The upper bound $\E|\C(0)|\leq (4+o(1))\eps^2 V$ is stated in Lemma \ref{lem-conc-Z-I} and the
lower bound follows immediately from our estimate on $\C_1$. Indeed, write $\C_j$ for the
$j$th largest component. Then
    $$
    \E |\C(0)| = V^{-1} \sum_{v \in V(G)} \E |\C(v)|
    = V^{-1} \sum_{j \geq 1} \E |\C_j|^2 \geq V^{-1} \E|\C_1|^2 \geq (4-o(1))\eps^2 V \, ,
    $$
where the first equality is by transitivity, the second equality is because each
component $\C_j$ is counted $|\C_j|$ times in the sum on the left and the last
inequality is due to (\ref{mainthm.toshow}). Furthermore, by this inequality and
Lemma \ref{lem-conc-Z-I}, we deduce that
    $$
    \sum_{j \geq 2} \E|\C_j|^2 = o(\eps^2 V^2) \, ,
    $$
and hence $|\C_2| = o(\eps V)$ \whp. This concludes the proof of Theorem \ref{mainthmgeneral}.
\qed \\

\subsection{Proof of Theorem \ref{mainthm}}

In this section we restrict our attention to the hypercube and prove Theorem \ref{mainthm}. We begin by showing that $\mnot$, defined in Theorem \ref{mainthmgeneral} with $\alpham = \m^{-1} \log \m$, satisfies $\mnot = O(\m \log m)$. See Lemma \ref{lem-NBW-high-d-tori}. The proof of Theorem \ref{mainthm} is then split into two cases. In the first case we assume that $\epsm \leq 1/\m^2$ so that $\eps = o(\mnot^{-1})$ and appeal to Theorem \ref{mainthmgeneral}. In the second case we perform the classical sprinkling argument for the case $\eps \geq 1/\m^2$, as done in \cite{BCHSS3}.

\begin{lemma}[NBW estimates]
% on high-dimensional tori]
\label{lem-NBW-high-d-tori}
On the hypercube $\{0,1\}^{\m}$
    $$
    \Tm(\m^{-1}\log{\m}) =O(\m \log \m)\, ,
    $$
and for any integer $L \geq 1$
%the NBW finite triangle is bounded by
    \eqn{
    \label{NBW-triangle-hyper}
    \sup_{x,y}
    \sum_{u,v} \sum_{\substack{t_1,t_2,t_3=0\\t_1+t_2+t_3\geq 3}}^{L} \p^{t_1}(x,u) \p^{t_2}(u,v) \p^{t_3}(v,y) \leq O(1/\m^2)+O(L^3/V).
    }
%(b) On the $d$-dimensional Hamming graph $K_n^d$,
%    $$
%    \Tm(1/n) =O(\log{n})\, ,
%    $$
%and
%    $$
%    \sup_{x,y}
%    \sum_{u,v} \sum_{\substack{t_1,t_2,t_3=0\\t_1+t_2+t_3\geq 3}}^{\mnot} \p^{t_1}(x,u) \p^{t_2}(u,v) \p^{t_3}(v,y) \leq O(1/n)+O(\mnot^3/V)\, .
%    $$
%(c) On general products of complete graphs $K_n^d$,
%    $$
%    \Tm((dn)^{-1}) =O(d\log{(n\vee d)})\, ,
%    $$
%and
%    $$
%    \sup_{x,y}
%    \sum_{u,v} \sum_{\substack{t_1,t_2,t_3=0\\t_1+t_2+t_3\geq 3}}^{\mnot} \p^{t_1}(x,u) \p^{t_2}(u,v) \p^{t_3}(v,y) \leq O(d^{-2}(n-1)^{-1})+O(\mnot^3/V)\, .
%    $$
%(d) On the $d$-dimension nearest-neighbor torus ${\mathbb T}_n^d$,
%    $$
%    \Tm(1/d) =O(d r^2+d\log{d})\, ,
%    $$
%and
%    $$
%    \sup_{x,y}
%    \sum_{u,v} \sum_{\substack{t_1,t_2,t_3=0\\t_1+t_2+t_3\geq 3}}^{\mnot} \p^{t_1}(x,u) \p^{t_2}(u,v) \p^{t_3}(v,y) \leq O(1/d^2)+O(\mnot^3/V)\, .
%    $$
\end{lemma}

\begin{proof}
We make use of the results in \cite{FitHof11a}, as we explain now.
The bound on $\Tm(\m^{-1}\log{\m}) =O(\m \log \m)$ is
\cite[Theorem 3.5]{FitHof11a}. We next explain how to
prove \eqref{NBW-triangle-hyper}, which will give
condition (3) in Theorem \ref{mainthmgeneral} for $L=A\m\log{\m}$
and an appropriate $A>0$.

Let $D\colon \{0,1\}^{\m} \to [0,1]$ be the simple random walk transition
probability on the hypercube. Our proof of \eqref{NBW-triangle-hyper}
relies on Fourier theory. For convenience,
we take the Fourier dual of $\{0,1\}^{\m}$
to be $\{0,1\}^{\m}$. Then, the Fourier transform $\hat{f}(k)$
of $f\colon  \{0,1\}^{\m} \to {\mathbb R}$ is given by
    \eqn{
    \hat{f}(k)=\sum_{x\in \{0,1\}^{\m}} (-1)^{x\cdot k}f(x),
    }
with inverse Fourier transform
    \eqn{
    f(x)=\frac{1}{V}\sum_{k\in \{0,1\}^{\m}} (-1)^{x\cdot k}\hat{f}(k).
    }
For the hypercube, $\hat{D}(k)$ takes the appealingly simple form
    \eqn{
    \label{Dhat-hyper}
    \hat{D}(k)=1-2a(k)/\m,
    }
where $a(k)$ is the number of non-zero coordinates of $k$.

In \cite[Theorem 3.5]{FitHof11a} it is proved that,
when $\m\geq 2$ and $t\geq 1$, with $\hat{\p}^t(k)$
denoting the Fourier transform of $x\mapsto \p^t(0,x)$,
    \eqn{
    \label{pt-bd}
    |\hat{\p}^t(k)|\leq \max\big(|\hat{D}(k)|, 1/\sqrt{\m-1}\big)^{t-1},
    }
and $\hat{\p}^0(k)=1$. This gives us all the necessary bounds
to prove the NBW triangle condition \eqref{NBW-triangle-hyper}.

Denote the sum in \eqref{NBW-triangle-hyper} by $S$.
The contribution to $S$ where $t_1+t_2+t_3=3$ equals $O(1/\m^2)$.
Thus, we are left to bound the contribution
due to $t_1,t_2,t_3$ with $t_1+t_2+t_3\geq 4$.
For any $t\geq 1$,
    \eqn{
    \p^t(x,y)\leq \frac{\m}{\m-1} (D*\p^{t-1})(x,y),
    }
where, for $f,g\colon \{0,1\}^{\m}\to \R$, we define the convolution $f*g$ by
    \eqn{
    \label{f*g-hyper}
    (f*g)(x)=\sum_{y\in \{0,1\}^{\m}} f(y)g(x-y).
    }
Therefore,
    \eqn{
    S\leq C/\m^2+3^4 \big(\frac{\m}{\m-1}\big)^4
    \sup_{x,y}\sum_{\substack{s_1,s_2,s_3=0\\s_1+s_2+s_3\geq 0}}^{L}
    (D^{*4}*\p^{s_1}*\p^{s_2}*\p^{s_3})(x,y),
    }
where $3^4$ is an upper bound on the number of ways we can add
$4$ to the coordinates of $(s_1,s_2,s_3)$ to get $(t_1,t_2,t_3)$
with $t_1+t_2+t+3\geq 4$. The above can be bounded
in terms of Fourier transforms as
    \eqan{
    S&\leq C/\m^2+\frac{C}{V} \sup_{x,y}\sum_{k\in \{0,1\}^{\m}}(-1)^{k\cdot (y-x)}
    \sum_{\substack{s_1,s_2,s_3=0\\s_1+s_2+s_3\geq 0}}^{L}
    \hat{D}(k)^{4}\hat{\p}^{s_1}(k)\hat{\p}^{s_2}(k)\hat{\p}^{s_3}(k)\\
    &\leq C/\m^2+\frac{C}{V} \sum_{k\in \{0,1\}^{\m}}
    \sum_{\substack{s_1,s_2,s_3=0\\s_1+s_2+s_3\geq 0}}^{L}
    \hat{D}(k)^{4}|\hat{\p}^{s_1}(k)||\hat{\p}^{s_2}(k)||\hat{\p}^{s_3}(k)|.\nn
    }
The contribution to $k=0$ equals $L^3/V$ since $\hat{D}(0)=\hat{\p}^t(0)=1$,
and the contribution due to $k=1$ (where 1 denotes the
all 1 vector) obeys the same bound.
It is not hard to adapt the proof of \cite[Proposition 1.2]{BCHSS2}
to show that the sum over $k\neq 0,1$ is $O(1/\m^2)$. We perform the
details of this computation now.

Writing $x_+=\max(x,0)$ for $x\in \R$, and noting that there are at most
$2$ values of $s$ for which $(s-1)_{+}=t$, we obtain
    \eqan{
    S&\leq C/\m^2+2L^3/V+\frac{C}{V} \sum_{k\in \{0,1\}^{\m}\colon k\neq 0,1}
    \sum_{\substack{s_1,s_2,s_3=0}}^{L}
    \hat{D}(k)^{4}\max\big(|\hat{D}(k)|, 1/\sqrt{\m-1}\big)^{(s_1-1)_++(s_2-1)_++(s_3-1)_+}\nn\\
    &\leq C/\m^2+2L^3/V+\frac{C2^3}{V} \sum_{k\in \{0,1\}^{\m}\colon k\neq 0}
    \hat{D}(k)^{4}\sum_{s_1,s_2,s_3=0}^{\infty}
    \max\big(|\hat{D}(k)|, 1/\sqrt{\m-1}\big)^{s_1+s_2+s_3}\nn\\
    &=C/\m^2 +2L^3/V+\frac{C2^3}{V} \sum_{k\in \{0,1\}^{\m}\colon k\neq 0,1}
    \frac{\hat{D}(k)^{4}}{\big[1-\max\big(|\hat{D}(k)|, 1/\sqrt{\m-1}\big)\big]^3}.
    \label{S-bd-almost}
    }
We bound
    \eqan{
    &\frac{1}{V} \sum_{k\in \{0,1\}^{\m}\colon k\neq 0,1}
    \frac{\hat{D}(k)^{4}}{\big[1-\max\big(|\hat{D}(k)|, 1/\sqrt{\m-1}\big)\big]^3}\\
    &\qquad\leq \frac{1}{V} \sum_{k\in \{0,1\}^{\m}\colon k\neq 0,1}
    \hat{D}(k)^{4}\Big[\frac{1}{[1-|\hat{D}(k)|]^3}+\frac{1}{[1-1/\sqrt{\m-1}]^3}\Big].\nn
    }
We next use the fact that $\frac{1}{V} \sum_{k\in \{0,1\}^{\m}}
\hat{D}(k)^{4}$ is the probability that a four-step simple random walk
on the hypercube returns to its starting point, which is $O(1/\m^2)$.
Alternatively, and more useful for the proof that follows,
we can write
    \eqn{
    \label{4-loop-hyper}
    \frac{1}{V} \sum_{k\in \{0,1\}^{\m}}\hat{D}(k)^{4}
    =2^{-\m}\sum_{j=0}^{\m} {{\m}\choose {j}}(1-2j/\m)^4
    =\m^{-4} \expec[(2X-\m)^4]=O(1/\m^2),
    }
where $X$ has a binomial distribution with parameters $1/2$ and $\m$,
and we use that $\expec[(2X-\m)^4]=O(\m^2)$.
We use similar ideas to deal with the contribution
involving $[1-|\hat{D}(k)|]^{-3}$, which we rewrite as
    \eqn{
    \frac{1}{V} \sum_{k\in \{0,1\}^{\m}\colon k\neq 0,1}
    \frac{\hat{D}(k)^{4}}{[1-|\hat{D}(k)|]^3}
    =2^{-\m} \sum_{j=1}^{\m-1} {{\m}\choose {j}} \frac{(1-2j/\m)^4}{[(2j/\m)\wedge (2-2j/\m)]^3}.
    }
The sum $2^{-\m} \sum_{j\not\in[\m/4,3\m/4]} {{\m}\choose {j}}$ is exponentially small
in $\m$
by either Stirling's formula or large deviation bounds on the binomial distribution with parameters
$\m$ and $1/2$.
When $j\in[\m/4,3\m/4]$, we can bound $1/[(2j/\m) \wedge (2-2j/\m)]^3\leq 8$
to bound the above sum by $O(1/\m^2)$ in the same way as in \eqref{4-loop-hyper}.
%    \eqn{
%    C \m^3\prob(X\leq \m/4)+C 2^{-\m} \sum_{j=\m/4}^{3\m/4} {{\m}\choose {j}} (1-2j/\m)^4
%    = O(1/\m^2),
%    }
Together with \eqref{S-bd-almost}, this completes the proof of \eqref{NBW-triangle-hyper}.
%where $\beta=1/\m^2$
%in cases (a) and (d), while $\beta=1/\m$ in cases (b) and (c).
%The distinction between these cases arises due to the fact that
%the minimal cycle is a triangle within a single dimension in
%cases (b) and (c), while the minimal cycle is a square using
%2 dimensions in cases (a) and (d).
\end{proof}

\noindent {\bf Proof of Theorem \ref{mainthm}}. We start by proving the theorem in the case $\epsm \leq 1/\m^2$. We take $\alpham=\m^{-1} \log \m $ and Lemma \ref{lem-NBW-high-d-tori} shows that $\mnot = O(\m \log \m)$ and that condition (3) of Theorem \ref{mainthmgeneral} holds. Condition (2) of Theorem \ref{mainthmgeneral} holds by (\ref{pcest}). Condition (1) is fulfilled automatically, therefore, in this case Theorem \ref{mainthm} follows from Theorem \ref{mainthmgeneral}.

We now handle the case $\eps \geq 1/ \m^2$ and $\eps = o(1)$. We start by proving
(\ref{mainthm.toshow}) in this case. In \cite{BCHSS3}, it is proven that $|\C_1| \geq c \eps V$ \whp\ in this case, and
the argument used there is based on isoperimetry together with Lemma \ref{lem-conc-Z-I}
and suffices to prove the required $2\eps V$ estimate in our setting as well,
as we show now.

Let $\theta>0$ be a small arbitrary constant. As before, fix the sprinkling
probability $p_2=\theta \eps/\m$
and take $p_1$ such that $p=p_c(1+\vep)=p_1+(1-p_1)p_2$ so that
$p_1=p_c(1+(1-\theta+o(1))\eps)$. By Lemma \ref{lem-conc-Z-I}, \whp\ in $G_{p_1}$,
    $$
    2(1-2\theta)\eps V\leq Z_{\sss \geq k_0}\leq 2(1+\theta)\eps V \, ,
    $$
for $k_0=\eps^{-2} (\eps^3 V)^{1/4}$.
As a result, there are at most $2(1+\theta)\eps V/k_0=2(1+\theta) (\eps^3 V)^{3/4}$
clusters of size at least $k_0$. Denote these clusters by $(\D_i)_{i\in I}$,
so that $|I|\leq 2(1+\theta) (\eps^3 V)^{3/4}$.

As before, we now perform sprinkling and add the edges of $G_{p_2}$. We bound
the probability that after the sprinkling there is a partition of the clusters
$(\D_i)_{i\in I}$ into two sets $S, T$ both containing at least $\theta \eps V$
vertices such that there is no path in $G_{p_2}$ connecting them. If there is no
such partition, then the largest component in $G_{p_1} \cup G_{p_2}$ has size at
least $(2-3\theta)\eps V$ and we conclude the proof. We follow
\cite[Proof of Proposition 2.5]{BCHSS3}.

Since $|I| \leq 2(1+\theta) (\eps^3 V)^{3/4}$ the number of such partitions
is at most $2^{2(1+\theta) (\eps^3 V)^{3/4}}$. We bound the probability
that given such a partition there is no $p_2$-open path connecting them. By
\cite[Lemma 2.4]{BCHSS3}, whenever $\Delta \geq 1$ satisfies
    \be \label{epslarge}
    \e^{-\Delta^2/2\m}\leq \theta \eps/2,
    \ee
there is a collection of at least $\frac 12 \theta \eps \m^{-2\Delta}V$
edge disjoint paths connecting $S$ and $T$, each of length at most $\Delta$.
This is where the isoperimetric inequality on the hypercube is being used. Note
that $\Delta$ needs to be large, in fact, we put $\Delta=\m^{2/3}$ and use the
fact that $\eps \geq \m^{-2}$ so that (\ref{epslarge}) holds. The probability
that a path of length $\Delta$ has a $p_2$-closed edge in it is $1-p_2^{\Delta}$. Since the
paths are disjoint, these events are independent, and we learn that the probability
that they all have a $p_2$-closed edge in them is at most
    \eqn{
    \Big(1-p_2^{\Delta}\Big)^{\frac 12 \theta \eps \m^{-2\Delta}V}
    \leq \e^{-cp_2^{\Delta}\theta \eps \m^{-2\Delta}V}
    =\e^{-c\theta^{\Delta}\eps^{\Delta}\m^{-3\Delta}V} .
    }
Thus, the total probability that sprinkling fails is at most
    \eqn{
    2^{2(1+\theta) (\eps^3 V)^{1-\alpha}}
    \e^{-c\theta^{\Delta}\eps^{\Delta}\m^{-3\Delta}V} = \e^{-c 2^{(1-o(1)) \m}},
    }
since $\alpha>0$ and $\eps \geq \m^{-2}$ (in fact, this argument works as
long as $\eps \geq \e^{-c\m^{1/3}}$). The proof of (\ref{mainthm.toshow}) when $\eps \gg V^{-1/3}$ and $\eps = o(1)$ is now completed.

The remaining bounds on $|\C_1|$, $\E|\C(0)|$ and $|\C_2|$ only rely on \eqref{mainthm.toshow} and Lemma \ref{lem-conc-Z-I} and are performed exactly as in the conclusion of the proof of Theorem \ref{mainthmgeneral}. This completes the proof of Theorem \ref{mainthm}. \qed

%which is doubly exponentially small for any $\alpha>0$ as long as
%(a) $\eps^{\Delta}\geq V^{-\tau}$ for any $\tau<\alpha$; and
%(b) $\e^{-\Delta^2/2\m}\leq \theta \eps/2$, which we can
%achieve by taking $\eps\geq \e^{-c \m^{1/3}}$ and
%$\Delta=a \m^{2/3}$ for appropriately chosen $a,c$.
%This completes the proof that sprinkling works \whp.

\subsection{Proof of Theorem \ref{expander}} \label{sec-examples}
%%%%%%% This used to be in the introduction, I'm moving it here for now to clear things up.

Our expansion and girth assumption of the theorem allows us to deduce some crude yet sufficient bounds on $\p^t(\cdot,\cdot)$, namely, that there exists some constant $q > 0$ so that
\begin{center}
\begin{table}[h]
\begin{tabular}{ l  l }

	$
	\p^t(0,0) \leq
	\begin{cases}               V^{-q} & t \leq C \log{V}\, ,\\
                                           CV^{-1} & t \geq C \log{V} \, ,
        \end{cases}
	$ \quad \and \quad
  & $
	\p^t(x,y) \leq
	\begin{cases}
		(m-1)^{-t} & t \leq (c \log_{m-1}{V})/2 ,\\
        	V^{-q} & t \geq (c \log _{m-1} {V})/2 \, .
	\end{cases}
	$
\end{tabular}
\end{table}
\end{center}
Indeed, the second bound on $\p^t(0,0)$ comes from the classical fact that $\Tm(CV^{-1}) = O(\log V)$. See e.g. \cite[below (19)]{ABLS}. The first bound on $\p^t(0,0)$ comes from the girth assumption, indeed, the BFS tree of $G$ rooted at $0$ is a tree up to height $\lfloor g/2 \rfloor$, where $g$ is the girth of the graph. Hence, in order for the walker to return to $0$ at time $t$ it must be at distance $t-\lfloor g/2 \rfloor$ from $0$ and then take the unique path of length $\lfloor g/2 \rfloor$ to $0$ so that $q$ can be taken to be any number smaller than $c/2$. The bounds on $\p^t(x,y)$ are proved similarly.

We take $\alpham = C(\log V)^{-1}$ (which is at least $1/\m$ by our assumption that $\m \geq c\log V$) and prove that conditions (2) and (3) of Theorem \ref{mainthmgeneral} hold. Note that $\mnot = O(\log V)$. To show condition (2) we show that percolation with $p=(\m-1)^{-1} (1+\alpham/\log{V})$ has $\E_p |\C(0)| \gg V^{1/3}$, whence $p_c \leq p$ and thus condition (2) holds. To show this lower bound on $\E_p|\C(0)|$ in this regime of $p$ it is possible to use a classical sprinkling argument. However, it is quicker to use \cite[Theorem 4]{Nach09} and verify that
	\be
	\label{expander.toverify}
	 \eps^{-1} r \sum_{t=1}^{2r} [(1+\eps)^{t \wedge r} -1] \p^t(0,0) = o(1) \, ,
	\ee
where $\eps = \alpham/\log{V}$ and $r=\eps^{-1} [ \log (\eps^3 V) -3\log \log(\eps^3 V)]$. Theorem 4 of \cite{Nach09} then yields that $\prob(|\C_1| \geq b \eps V/(\log(\eps^3 V))^3) = 1 - o(1)$ for some $b>0$ which immediately gives a lower bound on $\E|\C(0)|$ since
	$$
	\E |\C(0)| \geq V^{-1} \E|\C_1|^2 \geq (1+o(1)) b^2 \eps^2 V/(\log(\eps^3 V))^6 \gg V^{1/3} \, ,
	$$
by our choice of $\eps$.
%To verify (\ref{expander.toverify}) we use the following bounds on $\p^t(0,0)$:
%	$$
%	\p^t(0,0) \leq
%	\begin{cases}               V^{-q} & t \leq C \log{V}\, ,\\
%                                           CV^{-1} & t \geq C \log{V} \, ,
%        \end{cases}
%	$$
%for some constants $C,q$.  The first bound on $\p^t(0,0)$ is due to the girth assumption.
%For example, for $K_n^d$, the first bound holds with $q=1/d$.
%\todo{Asaf: can you give a reference for the first bound? I do not see this in general...}
%\todo{There is no girth assumption in Theorem \ref{expander}!}
%The second bound is due to the fact that $\mnot = O(\log V)$.
We use our bounds on $\p^t(0,0)$ above and sum (\ref{expander.toverify}) separately for $t \leq C\log V$ and $t \geq C\log V$. For $t \leq C\log V$ we bound $(1+\eps)^t - 1 = O(\eps t)$ and use our first bound $\p^t(0,0) \leq V^{-q}$ to get
	$$
	\eps^{-1} r \sum_{t=1}^{C\log V} [(1+\eps)^{t \wedge r} -1] \p^t(0,0)
	\leq r \sum_{t=1}^{C\log V} t V^{-q} = o(1) \, .
	$$
When $t \geq C\log V$ we bound
	$$
	(1+\eps)^{t\wedge r} - 1 \leq (1+\eps)^r = \eps^3 V (\log \eps^3 V)^{-3} =
	O(\eps^3 V (\log V)^{-3}) \, ,
	$$
by our choice of $\eps$. We use our second bound $\p^t(0,0) \leq CV^{-1}$ to bound
	$$
	\eps^{-1} r \sum_{t= C\log V}^{2r} [(1+\eps)^{t \wedge r} -1] \p^t(0,0)
	= O(r^2 \eps^2 (\log V)^{-3}) = o(1) \, ,
	$$
since $r \leq C (\log V)^2$. This concludes the verification of condition (2) of Theorem \ref{mainthmgeneral}. \\

To verify condition (3) we need to prove the bound
	\be
	\label{expander.toverify2}
	\sum_{u,v} \sum_{t_1,t_2,t_3 \colon t_1+t_2+t_3 \geq 3}^{C\log V}
	\p^{t_1}(x,u)\p^{t_2}(u,v)\p^{t_3}(v,y)
	= O( (\log V)^{-2}) \, .
	\ee
We first handle the special case of $(t_1,t_2,t_3)=(1,1,1)$. An immediate calculation with Lemma \ref{lem-NBW-ext} gives that (on any regular graph of degree $\m$)
	$$
	\sum_{u,v}\p^{1}(x,u)\p^{1}(u,v)\p^{1}(v,y) = O(1/\m^{2}) \, .
	$$
In all other cases of $(t_1,t_2,t_3)$ we use our bound on $\p^{t_i}(x,y)$ for $i\in {1,2,3}$ such that $t_i$ is the largest of $t_1,t_2,t_3$ (which must be at least $2$). We pull this bound out of the sum, and sum the other two terms over $u$ and $v$ to get a multiplicative contribution of precisely $1$. The sum over $(t_1,t_2,t_3)$ such that $3 \leq t_1+t_2+t_3 < 15$ is bounded by $C(\log V)^{-2}$ since the number of such triplets is bounded, and each contributes at most $C(\log V)^{-2}$ because one of the $t_i$'s is at least $2$, so that our bounds on $\p^t(x,y)$ guarantee that for this $t_i$ we have $\p^{t_i}(\cdot,\cdot)\leq O(1/\m^2)\leq O(1/(\log{V})^2)$ by the assumption that $\m\geq c\log{V}$. Similarly, the sum over triplets $(t_1,t_2,t_3)$ such that $t_1 + t_2+ t_3 \geq 15$ and $t_i \leq \mnot$ is also bounded by $C(\log V)^{-2}$ since the number of such triplets is at most $C (\log V)^3$, and each contributes at most $C (\log V)^{-5}$ because at least one of the $t_i$'s is at least $5$ and for this $t_i$ we have $\p^{t_i}(\cdot,\cdot) \leq C (\log V)^{-5}$ again by our assumption that $\m \geq c \log V$. This concludes our verification of conditions (2) and (3) of Theorem \ref{mainthmgeneral} and concludes our proof. \qed

\section{Open problems}
\label{sec-open}
\begin{enumerate}
\item In this paper we prove a law of large numbers for $|\C_1|$ above the critical window for percolation on the hypercube. Show that $|\C_1|$ satisfies a central limit theorem in this regime. In $G(n,p)$ this and much more was established by Pittel and Wormald \cite{PitWor05}.\\

\item Show that $|\C_2| = (2+o(1)) \eps^{-2} \log(\eps^3 2^m)$ when $p=p_c(1+\eps)$ and that $|\C_1| = (2+o(1)) \eps^{-2} \log(\eps^3 2^m)$ when $p=p_c(1-\eps)$ for $\eps \gg V^{-1/3}$ and $\eps = o(1)$. This is the content of \cite[Conjectures 3.1 and 3.3]{BCHSS3}. In \cite{BolKohLuc92} this is proved for $\eps\geq 60 (\log{n})^3/n$ in the supercritical regime, and for $\eps\geq (\log{n})^2/(n^{1/2}\log\log{n})$
    in the subcritical regime. In $G(n,p)$ these results are proved in
    \cite{PitWor05} and \cite[Theorem 5.6]{JLR00}.\\

\item Show that $(|\C_j|2^{-2m/3})_{j\geq 1}$ converges in distribution when $p=p_c(1+t 2^{-\m/3})$ and $t \in  {\mathbb R}$ is fixed and identify the limit distribution. Up to a time change, this should be the limiting distribution of $(|\C_j|n^{-2/3})_{j\geq 1}$ in $G(n,p)$ with $p=(1+ t n^{-1/3})/n$ identified by Aldous \cite{Aldo97}. \\

\item Consider percolation on the nearest-neighbor torus $\Z_n^d$ where $d$ is a large fixed constant and $n\to \infty$ with $p=p_c(1+\eps)$ such that $\eps \gg n^{-d/3}$ and $\eps = o(1)$. Show that $|\C_1|/(\eps n^d)$ converges to a constant. Does this constant equal the limit as $\vep\downarrow 0$ of $\eps^{-1} \theta_{\Z^d} (p_c(1+\eps))$? Here $\theta_{\Z^d}(p)$ denotes the probability that the cluster of the origin is infinite at $p$-bond percolation on the infinite lattice $\Z^d$. The techniques of this paper are not sufficient to show this mainly because condition (2) of Theorem \ref{mainthmgeneral} does not hold in $\Z_n^d$ (in fact, it is easy to see that $p_c - (2d-1)^{-1} \geq c > 0$ for some positive constant $c=c(d)$ --- this is always the case when our underlying transitive graph has constant degree and short cycles). The critical regime of this graph is well understood by the works \cite{BCHSS1, BCHSS2, HH, HH2}. \\% and the results of Section \ref{sec-bchss} hold as long as $d$ is a large enough fixed constant. \\

\item Show that the finite triangle condition (\ref{finitetriangle}) holds on any family of expander graphs. \\

\item Let $\delta>0$ be a fixed constant and consider the giant component $\C_1$ obtained by performing percolation on the hypercube with $p= (1+\delta)/\m$. Show that \whp\ the mixing time of the simple random walk on $\C_1$ is polynomial in $\m$. Is this mixing time of order $\m^2$? This is what one expects by the analogous question on $G(n,p)$, see \cite{BKW06, FR08}. Further analogy with the near-critical $G(n,p)$ (see \cite{DLP}) suggests that \whp\ the mixing time on $\C_1$ when $p=p_c(1+\eps)$ with the usual condition that $\eps \gg 2^{-\m/3}$ and $\eps = o(1)$ is of order $\eps^{-3} \log(\eps^3 2^\m)$.

\end{enumerate}

\appendix
\section{Asymptotics of the super-critical cluster tail}
\label{sec-clustertail}
Our goal in this section is to prove Theorem \ref{clustertail}. In
\cite{BCHSS1}, Theorem \ref{clustertail} is proved without the precise constant
$2$. Here we sharpen this proof to get this constant. We assume that $G$ is a general
transitive graph having degree $\cn$ and volume $V$ satisfying the finite triangle condition (\ref{finitetriangle}).
%We assume that there exists
%$\beta_{\sss \nabla}$ sufficiently small and $C>0$ such that for all $p\leq p_c$
%    \eqn{
%    \lbeq{nabla-asymp}
%    \nabla^{\rm {max}}(p) \leq \beta_{\sss\nabla}+C\chi(p)^3/V,
%    }
%where $\chi(p)=\expec|\Ccal(0)|$ denotes the expected cluster size
%and we define the \emph{maximal triangle diagram} $\nabla^{\rm {max}}(p)$ by
%    \eqn{
%    \lbeq{tria-def}
%    \nabla^{\rm {max}}(p)
%    =\sup_{x\neq y} \sum_{u,v} \tau_{p}(u-x)\tau_{p}(u-v)\tau_{p}(y-v),
%    }
In order to stay close to the notation in \cite{BCHSS1}, we define
$$ \nabla^{\rm {max}}_p = \sup_{x\neq y} \nabla_p(x,y) \, ,$$
and
    \eqn{
    \lbeq{tau-def}
    \tau_p(x)=\prob_p(0\conn x).
    }
%The assumption in \refeq{nabla-asymp} is closely related to the
%finite triangle condition in \eqref{finitetriangle}, for which
%$\beta_{\sss \nabla}=C/\m$, and we restrict to $x\neq y$.
%Our main bounds on the cluster tail are as follows:

\begin{prop}[Upper bound on the cluster tail]
\label{prop-upperbd-cluster-tail-gen}
Let $G$ be a finite transitive graph of degree $m$ on $V$ vertices such that the finite
triangle condition (\ref{finitetriangle}) holds and put $p=p_c(1+\eps)$
where $\eps = o(1)$ and $\eps \gg V^{-1/3}$.
Then, for every $k=k_\vep$ satisfying
$k_\vep \geq \vep^{-2}$,
    \eqn{
    \prob_p(|\Ccal(0)|\geq k)\leq 2\vep (1+O(\vep+(\vep^3V)^{-1}+(\vep^2k)^{-1/4}+\alpham)).
    }
\end{prop}

\begin{prop}[Lower bound on the cluster tail]
\label{prop-lowerbd-cluster-tail}
Let $G$ be a finite transitive graph of degree $m$ on $V$ vertices such that the finite
triangle condition (\ref{finitetriangle}) holds and put $p=p_c(1+\eps)$
where $\eps = o(1)$ and $\eps \gg V^{-1/3}$.
Then, for every $\alpha\in (0,1/3)$, there exists a $c=c(\alpha)>0$ such that
    \eqn{
    \prob_p\Big(|\Ccal(0)|\geq \eps^{-2} (\eps^3 V)^{\alpha}\Big)\geq
    2\vep (1+O(\vep+(\vep^3V)^{-c}+\alpham)).
    }
\end{prop}
%\todo{Check error terms in lower bound on cluster tail!}

\noindent
{\bf Remark.} The above propositions apply also
to infinite transitive graphs (where $(\eps^3V)^{-c}$ is replaced by
0), assuming that (\ref{finitetriangle}) holds with $\chi(p)^3/V$ replaced by 0.\\

\noindent {\bf Proof of Theorem \ref{clustertail}.} This proof follows immediately from the above propositions.
%Theorem \ref{clustertail} follows immediately from Propositions
%\ref{prop-upperbd-cluster-tail-gen} and \ref{prop-lowerbd-cluster-tail}.
%Indeed, \refeq{nabla-asymp} with $\beta_{\sss \nabla}=O(1/\m)$ follows
%from \eqref{finitetriangle}, while $\beta_{p_c}=O(1/\m)$
%follows from \eqref{finitetriangle} and \cite[(1.30)]{BCHSS1}.
%For general
%$p\leq p_c(\lambda)$, it follows from \cite[]{BCHSS1}
\qed
%\todo{Fix problem with triangle condition for general graphs!}

\subsection{Differential inequalities}
\label{sec-diff-ineq}
We follow \cite[Section 5]{BCHSS1}. For $p,\gamma\in [0,1]$, we
define the \emph{magnetization} by
    \eqn{
    \lbeq{Mdef}
    M(p,\gamma)
    =\sum_{k=1}^{V} [1-(1-\gamma)^{k}] \prob_{p}(|\Ccal(0)|=k).
    }
For fixed $p$, the function $\gamma \mapsto M(p,\gamma)$ is strictly
increasing, with $M(p,0)=0$ and $M(p,1)=1$. When we color all vertices
independently \emph{green} with probability $\gamma$, and we let $\Gcal$
denote the set of green vertices, then \refeq{Mdef} has the appealing
probabilistic interpretation of
    \eqn{
    \lbeq{Mdef-rep}
    M(p,\gamma)=\prob_{p,\gamma}(0\conn \Gcal),
    }
where $\prob_{p,\gamma}$ is the probability measure of the
joint bond and site percolation model, where bonds and sites
have an independent status. This representation is important for the derivation of useful differential inequalities involving the magnetization.

%This representation is important for the
%derivation of differential inequalities
%
% that lie at the core of our
%analysis of the supercritical phase of our random graphs. We start by
%stating these differential inequalities:

\begin{lemma}[Differential inequalities for the magnetization] \label{lem-diff-ineq} Let $G$ be a finite transitive graph on $V$ vertices and degree $m$. Then at any $p, \gamma \in (0,1)$
        \eqn{
        \lbeq{ineq1}
        (1-p)\frac{\partial M}{\partial p}
        \leq \cn (1-\gamma) M \frac{\partial M}{\partial \gamma},
        }
        \eqn{
        \lbeq{ineq2}
        M \leq \gamma \frac{\partial M}{\partial \gamma}
        +\big[\frac{1}{2}\cn p M^2+\gamma M\big]+ \big[\frac{1}{2}\cn p M +\gamma\big]p\frac{\partial M}{\partial p},
        }
and
        \eqn{
        \lbeq{rdi}
         M \geq
        \cn p\big[\gamma + (1-\gamma)\frac{1}{2} \cn(\cn-1) p^2 \alpha(p) M^2\big]\frac{\partial M}{\partial\gamma},
        }
where
    \eqn{
    \lbeq{alpha(p)-def}
    \alpha(p)=(1-2p)^2-(1+\cn p+2(\cn p)^2)\nabla^{\rm max}_p-
    \cn p M-(\cn p)^2 M^2.
    }
\end{lemma}

The inequality \refeq{ineq1} is proved in \cite{AB87}, where it was used
to prove the sharpness of the percolation phase transition on $\Z^d$,
and was first stated in the context of finite graphs in \cite[(5.14)]{BCHSS1}.
The differential inequality in \refeq{ineq2} is an adaptation of
another differential inequality proved and used in \cite{AB87},
which is improved here in order to obtain sharp constants in our bounds.
The bound in \refeq{rdi} is an adaptation of \cite[(5.16)]{BCHSS1},
which was used there in order to prove an upper bound on $M(p,\gamma)$.
Again, the inequality is adapted in order to obtain the optimal constants. We will first use Lemma \ref{lem-diff-ineq} to obtain Propositions \ref{prop-upperbd-cluster-tail-gen} and \ref{prop-lowerbd-cluster-tail}.

\subsection{The magnetization for subcritical $p$}
\label{sec-mag-pc}
We take $p=p_c(1-\vep)$ with $\vep=o(1)$ and $\vep^3 V\gg 1$,
and we take $\gamma=o(1)$. Then, \cite[Lemma 5.3]{BCHSS1} shows that
$M(p,\gamma)=O(\sqrt{\gamma})$. The main aim of
this section is to improve upon this bound, using
the improved differential inequality in \refeq{rdi}.

We have that $M(p,\gamma)=O(\sqrt{\gamma})$ and $\chi(p)=O(1/\vep)$ by \cite[Theorem 1.5]{BCHSS1}. Further more, assumption (\ref{finitetriangle}) gives that $\nabla^{\rm max}_p = O(\alpham + (\eps^3 V)^{-1})$ and \cite[(1.30)]{BCHSS1} then implies that $mp \leq 1+O(\alpham)$. Putting all this into \refeq{alpha(p)-def} yields
    \eqn{
    \lbeq{alpha(p)-bd}
    \alpha(p)\geq 1+O\big(\sqrt{\gamma}+(\vep^3V)^{-1}+\alpham \big).
    }
Substituting \refeq{alpha(p)-bd} into
\refeq{rdi} in turn gives that
    \eqn{
    \lbeq{diff-ineq-mag-rep}
    M \geq \big(1+O\big(\sqrt{\gamma}+(\vep^3V)^{-1}+\alpham \big)\big)
    \big[\gamma +\frac{1}{2}M^2\big]\frac{\partial M}{\partial\gamma}.
    }
We now use this to prove the following lemma:

\begin{lemma}[Upper bound on the slightly subcritical magnetization]
\label{lem-ub-crit-mag}
Let $G$ be a finite transitive graph of degree $m$ on $V$ vertices such that the finite
triangle condition (\ref{finitetriangle}) holds.
Let $\gamma=o(1)$ and put $p=p_c(1-\vep)$ with $\vep=o(1)$ and $\vep^3V\gg 1$. Then,
    \eqn{
    M(p,\gamma)\leq \sqrt{2\gamma}\big(1+O\big(\sqrt{\gamma}+(\vep^3V)^{-1}+\alpham \big)\big).
    }
\end{lemma}

A similar bound as in Lemma \ref{lem-ub-crit-mag} was proved in \cite[Lemma 5.3]{BCHSS1},
whose proof we adapt here, with $\sqrt{2\gamma}$ replaced with $\sqrt{12 \gamma}$,
and a less precise error bound. The precise constant $\sqrt{2}$ is important for us here as it relates to the constant $2$ for the $2\vep(1+o(1))$ survival probability.
 %as it will be a crucial ingredient to identify the asymptotics of $\prob_p(|\Ccal(0)|\geq k)$ for $k\gg 1/\vep^2$.

\proof We note that \refeq{diff-ineq-mag-rep} implies that
    \eqn{
    M \geq
    \frac{B}{2}M^2\frac{\partial M}{\partial\gamma},
    }
where we abbreviate $B=1+O\big(\sqrt{\gamma}+(\vep^3V)^{-1}+\alpham \big)$. Therefore,
    \eqn{
    \frac{\partial [M^2]}{\partial\gamma}\leq 4/B.
    }
Integrating between $0$ and $\gamma$, and using that $M(p,0)=0$ yields that
    \eqn{
    M^2 \leq 4\gamma/B,
    }
so that $M\leq \sqrt{\gamma}(2/\sqrt{B}).$
Now, when we have this inequality, we can further bound
    \eqn{
    \gamma \geq \frac{B}{4} M^2,
    }
so that by \refeq{diff-ineq-mag-rep} we get
    \eqn{
    M \geq B\big[1/2+B/4\big]M^2\frac{\partial M}{\partial\gamma}.
    }
Performing the same integration steps, we arrive at
    \eqn{
    M^2\leq \frac{2}{B/2+B^2/4}\gamma.
    }
Therefore, the constant has become a little better (recall that $B$ is close to $1$). Iterating these steps yields that, for every $k\geq 1$,
    \eqn{
    \lbeq{M2-ind}
    M^2\leq \frac{2}{\sum_{l=1}^k (B/2)^j}\gamma.
    }
We prove \refeq{M2-ind} by induction on $k$, the initialization for
$k=1,2$ having been proved above. To advance the induction hypothesis,
suppose that \refeq{M2-ind} holds for $k\geq 1$. Define
    \eqn{
    A_k=\sum_{l=1}^k (B/2)^j,
    }
so that \refeq{M2-ind} is equivalent to $M^2\leq 2\gamma/A_k$.
In turn, this yields that $\gamma\geq A_k M^2/2$,
so that
    \eqn{
    M \geq B\Big[A_k/2+1/2\Big]M^2\frac{\partial M}{\partial\gamma},
    }
which in turn yields that
    \eqn{
    M^2 \leq \frac{2}{B[1+A_k]/2} \gamma.
    }
Note that
    \eqn{
    B[1+A_k]/2=A_{k+1},
    }
which advances the induction. By \refeq{M2-ind}, we obtain that
    \eqn{
    M^2\leq \frac{2}{\sum_{l=1}^{\infty} (B/2)^j}\gamma
    = 2[2-B]\gamma/B.
    }
Finally, the fact that
    \eqn{
    [2-B]/B=1+O\big(\sqrt{\gamma}+(\vep^3V)^{-1}+\alpham \big)
    }
completes the proof.
\qed

\subsection{The magnetization for supercritical $p$}
\label{sec-mag-p}
In this section, we use \emph{extrapolation inequalities} to obtain a
bound on the supercritical magnetization from the subcritical one derived in
Lemma \ref{lem-ub-crit-mag}. Our precise result is the following:

\begin{lemma}[Upper bound on the slightly supercritical magnetization]
\label{lem-ub-supercrit-mag} Let $G$ be a finite transitive graph of degree $m$ on $V$ vertices such that the finite
triangle condition (\ref{finitetriangle}) holds and put $p=p_c(1+\eps)$
where $\eps = o(1)$ and $\eps \gg V^{-1/3}$.
Then for any $c\in (0,1/3)$,
    \eqn{
    \lbeq{ub-super-crit-mag}
    M(p,\gamma)\leq \Big(\vep + \sqrt{2\gamma+\vep^2}\Big)
    \big(1+O\big(\vep+\sqrt{\gamma}+(\vep^3V)^{-c}+\alpham \big)\big).
    }
\end{lemma}

\proof We follow the proof in \cite[Section 5.3]{BCHSS1}, paying special attention to
the constants and error terms. Indeed, we use  \refeq{ineq1} and the chain rule to deduce that, with $A=(1-2p_c)^{-1}$,
and $\tilde M(p,h)=M(p,1-\e^{-h})$,
    \eqn{
    \frac{\partial \tilde M}{\partial p}
    \leq \cn A \tilde M \frac{\partial \tilde M}{\partial h}.
    }
Take $P_1=(p_c(1+\vep), h)$ and write $m_1=\tilde M(P_1)$. Further, take
$\eta=\vep (\vep^3V)^{-c}$ for some $c\in (0,1/3)$, so that $\eta=o(\vep)$
and $\eta^3V\rightarrow \infty$, and take $P_2=(p_c(1-\eta), Am_1 \vep')$,
where
    \eqn{
    \vep'=\vep+\eta+\frac{h}{Am_1}.
    }
Then, with $m_2=\tilde M(P_2)$, we have that $m_2\geq m_1$ (see e.g.,
\cite[(5.46)]{BCHSS1}). Therefore, by Lemma \ref{lem-ub-crit-mag}
and again writing $B=1+O\big(\sqrt{\gamma}+(\vep^3V)^{-1}+\alpham \big)$
with $\gamma=1-\e^{-h}$,
    \eqan{
    M(p,1-\e^{-h})&=m_1\leq m_2\leq \sqrt{2B(1-\e^{-Am_1\vep'})}\\
    &=(1+O(m_1\vep'))\sqrt{2ABm_1\vep'}\nn\\
    &=(1+O(m_1\vep))\sqrt{2ABm_1(\vep+\eta)+2Bh}\nn\\
    &=(1+O(\vep+(\vep^3V)^{-c}))\sqrt{2ABm_1\vep+2Bh},\nn
    }
where in the last inequality we use that $\eta=\vep (\vep^3V)^{-c}\ll \vep$ and $m_1\leq 1$.
The inequality
    \[m_1\leq \sqrt{2ABm_1\vep+2Bh}
    \]
has roots
    \eqn{
    m^{\pm} =AB\vep \pm \sqrt{2Bh +(AB\vep)^2}.
    }
Since $m_1\geq 0$ and $m_+\geq 0$ while $m_-\leq 0$, we deduce that
    \eqn{
    M(p_c+\vep/\cn, 1-\e^{-h})=m_1\leq (1+O(\vep+(\vep^3V)^{-c}))(AB\vep + \sqrt{2Bh +(AB\vep)^2}).
    }
We have that $\gamma=1-\e^{-h}=h(1+O(h))$ and $A=1+O(\alpham)$ (by \cite[(1.30)]{BCHSS1}) and
$B=1+O\big(\sqrt{\gamma}+(\vep^3V)^{-1}+\alpham \big)$. Putting all this together in the last inequality completes the proof.
\qed
\medskip

\noindent
{\bf Proof of Proposition \ref{prop-upperbd-cluster-tail-gen}.} We note that, for any $l\geq k\geq 1$ and
$a>0$,
    \eqn{
    1-\big(1-a/k\big)^l\geq 1-\e^{-a}.
    }
Therefore, by \refeq{Mdef},
    \eqn{
    \prob_p(|\Ccal(0)|\geq k)\leq [1-\e^{-a}]^{-1}M(p,a/k).
    }
Recall that $k\gg \vep^{-2}$ and take $a=(\vep^2 k)^{1/2}$
so that $a/k=\vep^2(\vep^2k)^{-1/2}=o(\vep^2)$.
We note that for $\gamma=\vep^2 (\vep^2k)^{-1/2}$,
\refeq{ub-super-crit-mag} reduces to
    \eqn{
    \lbeq{ub-super-crit-mag-rep}
    M(p,\gamma)\leq 2\vep\big(1+O\big(\vep+(\vep^3V)^{-1}+(\vep^2k)^{-1/4}+\alpham \big)\big).
    }
Then, by \refeq{ub-super-crit-mag-rep} and the fact that $1-\e^{-a}=1+o((\vep^2 k)^{-1/4})$,
    \eqn{
    M(p,a/k)\leq 2\vep\big(1+O\big(\vep+(\vep^3V)^{-1}+(\vep^2k)^{-1/4}+ \alpham \big)\big).
    }
This completes the proof of Proposition
\ref{prop-upperbd-cluster-tail-gen}.
\qed

\subsection{Lower bound on tail probabilities}
\label{sec-pf-prop-lowerbd-cluster-tail}
In the remainder of this section, we shall prove
Proposition \ref{prop-lowerbd-cluster-tail}.
Throughout this proof, we will take $p=p_c(1+\vep)$.
%\todo{Both: Check form of error terms!}

We shall assume that with $k_0=\eps^{-2} (\eps^3 V)^{\alpha} \gg \eps^{-2}$ and $\alpha\in (0,1/3)$,
there exists $b_{10}=b_{10}(\alpha)$ such that
    \eqn{
    \lbeq{theta-prel-bd}
    \prob_p(|\Ccal(v)|\geq \eps^{-2} (\eps^3 V)^{\alpha})\geq b_{10} \vep.
    }
The bound in \refeq{theta-prel-bd} is proved for finite graphs in
\cite[Theorem 1.6(i)]{BCHSS1} and in
\cite{BA91}, in conjunction with \cite{HS90},
on infinite lattices satisfying the triangle condition.
The proof of \refeq{theta-prel-bd} is similar to the
argument we shall give for the improved bound,
and shall be omitted here.
In turn, \refeq{theta-prel-bd} implies that, for $\gamma=1/k_0=\eps^{2} (\eps^3 V)^{-\alpha}=o(\vep^{2})$,
there exists a constant $\tilde{b}_{10}$ such that
    \eqn{
    \lbeq{M-lb-input}
    M(p,\gamma)\geq [1-[1-\gamma]^{k_0}]\prob_p(|\Ccal(v)|\geq k_0)\geq \tilde b_{10} \vep.
    }
Equation \refeq{M-lb-input} will be an essential ingredient in our proof.
We start by proving the following lemma:

\begin{lemma}[Lower bound on the magnetization]
\label{lem-lowerbd-mag} Let $G$ be a finite transitive graph of degree $m$ on $V$ vertices such that the finite
triangle condition (\ref{finitetriangle}) holds and put $p=p_c(1+\eps)$
where $\eps = o(1)$ and $\eps \gg V^{-1/3}$. Then, for $\gamma=\eps^{2} (\eps^3 V)^{-\alpha}$
with $\alpha\in (0,1/3)$ and any $c<1$,
    \eqn{
    M(p,\gamma)\geq 2\vep \big[1+O\big(\vep+(\vep^3V)^{-c}+\alpham \big)\big].
    }
\end{lemma}

\proof Throughout the proof, we fix $\alpha\in (0,1/3)$.
We recall the differential inequality \refeq{ineq2}
        \eqn{
        \lbeq{ineq2-rep-2}
        M \leq \gamma \frac{\partial M}{\partial \gamma}
        +\big[\frac{1}{2}\cn p M^2+\gamma M\big]+ \big[\frac{1}{2}\cn p M +\gamma\big]p\frac{\partial M}{\partial p}.
        }
By \refeq{M-lb-input}, and the fact that $\gamma\mapsto M(p,\gamma)$
is increasing, for any $\gamma= \vep^2 (\vep^3 V)^{-\alpha}$ we have that $\gamma=O(M \vep)$.
Further, $\cn p\leq 1+O(\vep+\alpham)$, so that, for some
$A>1$ with $A=1+O(\vep+\alpham)$ we obtain
    \eqn{
    \lbeq{ineq2-simplified}
    M \leq \gamma \frac{\partial M}{\partial \gamma}
    +\frac{A}{2}M^2+ \frac{A}{2}Mp\frac{\partial M}{\partial p}.
    }
We rewrite \refeq{ineq2-simplified} as
    \eqn{
    0\leq \frac{1}{M}\frac{\partial M}{\partial \gamma}
        +\frac{1}{\gamma}\frac{\partial}{\partial p}[\frac{A}{2}pM-p],
    }
and integrate for $\gamma\in [\gamma_0,\gamma_1]$ and
$p\in [p_0, p_1]$, where $\gamma_0=(\delta\eps)^{2} (\delta^3\eps^3 V)^{-\alpha}$.
%, for some $\alpha'<\alpha$ to be chosen later on.
We note that \refeq{M-lb-input} holds for $p_0=p_c(1+\vep\delta)$ for any
$\delta=o(1)$ and $\gamma=\gamma_0$. We further take
    \eqn{
    p_0=p_c(1+\delta \vep),
    \qquad
    p_1=p_c(1+\vep),
    \qquad
    \gamma_1=\e^{(\log{(1/\delta)})^a} \gamma_0,
    }
where $a>1$ is chosen below.

Then, as in \cite[(5.57) and the argument below it]{Grim99}, by the fact that
$p\mapsto M(p,\gamma)$ and $\gamma\mapsto M(p,\gamma)$ are non-decreasing,
    \eqn{
    \lbeq{M-integrated-pgamma}
    0\leq (p_1-p_0)\log\frac{M(p_1,\gamma_1)}{M(p_0,\gamma_0)}
    +\log(\gamma_1/\gamma_0)[\frac{A}{2}p_1M(p_1,\gamma_1)-(p_1-p_0)].
    }
Now,
    \eqn{
    \log(\gamma_1/\gamma_0)=(\log{(1/\delta)})^a,
    }
while, by Lemma \ref{lem-ub-supercrit-mag} and \refeq{M-lb-input},
    \eqn{
    M(p_1,\gamma_1)\leq 2\vep\big(1+O\big(\vep+(\vep^3V)^{-1}\big)\big),
    \qquad
    M(p_0,\gamma_0)\geq \tilde{b}_{10} \delta \vep,
    }
so that, for $\delta>0$ sufficiently small,
    \eqn{
    \log \frac{M(p_1,\gamma_1)}{M(p_0,\gamma_0)}\leq \log(2\vep/(\tilde{b}_{10}\delta\vep))\leq 2\log(1/\delta).
    }
Dividing \refeq{M-integrated-pgamma} through by $(\log{(1/\delta)})^a$, we arrive at
    \eqn{
    \frac{A}{2}p_1M(p_1,\gamma_1)
    \geq  p_c(1-\delta)\vep\Big[1-2(\log{(1/\delta)})^{1-a}\Big].
    }
Recalling that $p_1=p_c(1+\vep)$ and that $a>1$, as well as the fact that
$A=1+O(\vep+\alpham)$, this yields
   \eqn{
   M(p,\gamma_1)
   \geq  2\vep[1+O(\vep+(\log{(1/\delta)})^{1-a}+\alpham)].
    }
Finally, note that
    \eqan{
    \gamma_1&=\e^{(\log{(1/\delta)})^a} \gamma_0
    =\e^{(\log{(1/\delta)})^a}(\delta\eps)^{2} (\delta^3\eps^3 V)^{-\alpha}\\
    &=\eps^{2} (\eps^3 V)^{-\alpha}  \big(\e^{(\log{(1/\delta)})^a}\delta^{2-3\alpha}\big)\geq \eps^{2} (\eps^3 V)^{-\alpha},\nn
    }
when we take $\delta=\e^{-(\vep^3V)^{1/a}}$ for any $a>1$. Indeed, then $\e^{(\log{(1/\delta)})^a}\delta^{2-3\alpha}
\rightarrow \infty$ as $\delta \to 0$.
Since $\gamma\mapsto M(p,\gamma)$ is
increasing, this implies that
    \eqn{
    M(p,\gamma)\geq M(p,\gamma_1)
   \geq  2\vep[1+O(\vep+(\log{(1/\delta)})^{1-a}+\alpham)]
   =2\vep\big[1+O\big(\vep+(\vep^3V)^{1/a-1}+\alpham \big)\big].
   }
Denoting $c=1-1/a$, this proves the claim.
\qed
\medskip

\noindent
{\bf Proof of Proposition \ref{prop-lowerbd-cluster-tail}.}
We use \cite[(6.5) in Lemma 6.1]{BCHSS1},
which states that, for any $0\leq \gamma_0,\gamma_1\leq 1$,
    \eqn{
    \lbeq{PgeqM}
    \prob_p(|\cluster(0)|\geq k)
    \geq
    M(p,\gamma_1)
    -  \frac {\gamma_1}{\gamma_0}\,\e^{\gamma_0 k}
    M(p,\gamma_0).
    }
Now we take $\gamma_1=\eps^{2} (\eps^3 V)^{-\alpha'}$ with $\alpha'\in (0,1/3)$
taken as in Lemma \ref{lem-lowerbd-mag},
$\gamma_0=\eps^{2} (\eps^3 V)^{-\alpha}$ with $\alpha<\alpha'$, and $k=1/\gamma_0$.
Then $\e^{\gamma_0 k}=\e$,
while, by Lemma \ref{lem-ub-supercrit-mag} and the fact that
$\gamma_0=o(\vep^2)$, we obtain that
    \eqn{
    M(p,\gamma_0)\leq 2\vep(1+o(1)).
    }
Therefore, by Lemma \ref{lem-lowerbd-mag} and \refeq{PgeqM},
taking $c=1/2$ in Lemma \ref{lem-lowerbd-mag},
    \eqn{
    \lbeq{PgeqM-rep}
    \prob_p(|\cluster(0)|\geq k)
    \geq
    2\vep(1+O(\vep+(\vep^3V)^{-1/2}+\alpham))
    -(\eps^3 V)^{\alpha'-\alpha}O(\vep).
    }
We obtain that
    \eqn{
    \lbeq{PgeqM-rep2}
    \prob_p(|\cluster(0)|\geq \eps^{-2} (\eps^3 V)^{\alpha})
    \geq
    2\vep(1+O(\vep+(\vep^3V)^{-1/2}+(\eps^3 V)^{\alpha-\alpha'}+\alpham)).
    }
This proves the claim in Proposition \ref{prop-lowerbd-cluster-tail}
with $c=\alpha-\alpha'\in(0,1/3)$.
\qed

%\bibliographystyle{plain}
%\bibliography{mybib}
%\bibliography{../../onderzoek/bib/bib}

%\listoftodos
%\renewcommand{\thesection}{\Alph{section}}
%\setcounter{section}{0}

%\section{Appendix: Derivation of the differential inequalities}
%\label{sec-app}

\subsection{Derivation of \refeq{rdi}}
\label{sec-ineq2}
We follow the proof in \cite[Appendix A.2]{BCHSS1} as closely as possible,
deviating in one essential inequality. Indeed, in \cite[(A.23-A.32)]{BCHSS1},
it is proved that
    \eqn{
    M(p,\gamma) \geq
        p\cn \frac{\partial M}{\partial\gamma}(p,\gamma) \prob_{p,\gamma}(0\dbc \Gcal)
        -X_2-X_3.
    }
We copy the bounds on $X_2$ and $X_3$ in \cite[(A.46)]{BCHSS1}
and \cite[(A.53)]{BCHSS1} respectively, which prove that
    \eqn{
    X_2\leq p^2 \cn M^2(p,\gamma) \frac{\partial M}{\partial\gamma}(p,\gamma),
    \qquad
    X_3\leq \nabla^{\rm max}_p p\cn M^2(p,\gamma) \frac{\partial M}{\partial\gamma}(p,\gamma),
    }
and we improve upon the lower bound on $\prob_{p,\gamma}(0\dbc \Gcal)$ only.
Our precise results is contained in the following lemma:

\begin{lemma}[Improved lower bound on the double connection]
\label{lem-lb-dbc}
For all $p, \gamma\in [0,1]$,
    \eqn{
    \lbeq{double-conn-G}
    \prob_{p,\gamma}(0\dbc \Gcal)
    \geq \gamma + (1-\gamma)\frac{1}{2} \cn(\cn-1) p^2 \alpha(p) M^2(p,\gamma),
    }
where
    \eqn{
    \alpha(p)=(1-2p)^2-(1+\cn p+2(\cn p)^2)\nabla^{\rm max}_p-
    \cn p M(p,\gamma)-(\cn p)^2 M(p,\gamma)^2.\nn
    }
\end{lemma}

\proof Note that if $0\in \Gcal$, then $0\dbc \Gcal$ occurs.
Therefore, we obtain
    \eqn{
    \prob_{p,\gamma}(0\dbc \Gcal)=\gamma + (1-\gamma)\prob_{p,\gamma}(0\dbc \Gcal\mid 0\not\in \Gcal).
    }
Thus, we are left to obtain a lower bound on $\prob_{p,\gamma}(0\dbc \Gcal\mid 0\not\in \Gcal)$.
For this, we follow the original argument in \cite[Section A.2]{BCHSS1},
adapting it when necessary.

For a directed bond $b=(x,y)$, we write $\bb=x$ and $\tb=y$
for its top and bottom. Let $e,f$ be two distinct bonds with $\bbe=\bbf=0$, and let
$E_{e,f}$ be the event that the bonds $e$ and
$f$ are occupied, and that in the reduced graph
$G^-=(V^-, E^-)$
obtained by removing the bonds $e$ and $f$, the
following three events occur:  $\te \conn \Gcal$, $\tf \conn
\Gcal$, and $\Ccal(\te) \cap \Ccal(\tf) = \emptyset$.

Let $\prob_{p,\gamma}^-$ denote the joint bond/vertex measure on
$G^-$.  We note that the
event $\{0\dbc \Gcal\}$ contains the event $\cup_{e,f}E_{e,f}$,
where the (non-disjoint) union is over unordered pairs of bonds $e,f$ incident to the
origin. Then, by Bonferroni's inequality and since $E_{e,f}$ is independent of $0\not\in \Gcal$, we get
    \eqan{
    \prob_{p,\gamma}(0 \dbc \Gcal\mid 0\not\in \Gcal)
    &
    \geq \prob_{p,\gamma} \big(\cup_{e,f} E_{e,f}\mid 0\not\in \Gcal\big)
    \geq
    \sum_{\{e,f\}}\prob_{p,\gamma}(E_{e,f})
    -Y_1
    \lbeq{pdbG1}
    \\
    &= p^2\sum_{e,f}
    \prob_{p,\gamma}^-(\te \conn \Gcal, \;\; \tf \conn \Gcal, \;\;
    \Ccal(\te) \cap \Ccal(\tf) = \emptyset)-Y_1,\nn
    }
where
    \eqn{
    Y_1=\frac{1}{2}\sum_{\{e_1,f_1\}\neq \{e_2,f_2\}}\prob_{p,\gamma}(E_{e_1,f_1}\cap E_{e_2,f_2}\mid 0\not\in \Gcal).
    }
We first bound $Y_1$. For this, we note that there are two contributions to $Y_1$, depending
on the number of distinct elements in $\{e_1,f_1,e_2,f_2\}$, which can be 3 or 4,
and whose contributions we denote by $Y_{1,3}$ and
$Y_{1,4}$, respectively.

We start by bounding $Y_{1,3}$. The number of pairs of pairs of edges $\{e_1,f_1\} \neq \{e_2,f_2\}$ such that $|\{e_1,f_1,e_2,f_2\}|=3$ is $m(m-1)(m-2)$. For such a pair, let $x_1,x_2,x_3$ denote the distinct elements of $\{\te_1,\tf_1,\te_2,\tf_2\}$ such that $x_1$ corresponds to the end of the edge that appears twice in $\{e_1,f_1,e_2,f_2\}$.
% for which we let $x_1,x_2,x_3$ denote the elements in
%with $|\{\te_1,\tf_1,\te_2,\tf_2\}|=3$, where $x_1=e_1$ and $x_2=f_1$. Given a feasible triple
%$(x_1,x_2,x_3)$, there are 2 ways in which
%$\{\te_1,\tf_1,\te_2,\tf_2\}=\{x_1,x_2, x_3\}$, namely, $(e_2,f_2)=(x_2, x_3)$ and
%$(e_2,f_2)=(x_3, x_2)$. By symmetry, one of these is over counted.
If $E_{e_1,f_1}\cap E_{e_2,f_2}$ occurs, then either
    $$
    \{(0,x_1)\text{ occ.}\}\circ\{(0,x_2)\text{ occ.}\}\circ
    \{(0,x_3)\text{ occ.}\}\circ\{x_1\conn \Gcal\}\circ \{x_2\conn \Gcal\}\circ \{x_3\conn \Gcal\}
    $$
occurs, or there exists a $z$ such that
    \eqan{
    &\{(0,x_1)\text{ occ.}\}\circ\{(0,x_2)\text{ occ.}\}\circ
    \{(0,x_3)\text{ occ.}\}\circ \{x_1\conn \Gcal\}\nn\\
    &\qquad \circ \{x_2\conn z\}\circ \{x_3\conn z\}\circ \{z\conn \Gcal\}\nn
    }
occurs. Therefore,
    \eqn{
    Y_{1,3}\leq (1-\gamma)
    \frac{1}{2}\cn(\cn-1)(\cn-2) p^3 M(p,\gamma)^2 [M(p,\gamma)+\nabla_p^{\rm max}] \, ,
    }
where we bounded
$$ \sum_{z} \prob_p (x_2\conn z) \prob_p(x_3\conn z) \leq \nabla_p^{\rm max} \, ,$$
which is wasteful, but sufficient for our purposes.

For $Y_{1,4}$, we sum over $\{e_1,f_1\}\neq \{e_2,f_2\}$ with the constraint that
all these edges are distinct. The number of such pairs of pairs is $m(m-1)(m-2)(m-3)/4$. Then, a similar computation
as for $Y_{1,3}$ yields that
    \eqn{
    Y_{1,4}\leq (1-\gamma)
    \frac{1}{8}\cn(\cn-1)(\cn-2)(\cn-3)p^4 M(p,\gamma)^2 [M(p,\gamma)^2+8\nabla_p^{\rm max}].
    }

We continue to bound the sum over $\{e,f\}$ in \refeq{pdbG1} from below.
Let
    \eqn{
    W=W_{e,f}=\{\te \conn \Gcal, \;\; \tf \conn \Gcal, \;\;
    \Ccal(\te) \cap \Ccal(\tf) = \emptyset\},
    }
denote the event whose probability appears on the right
side of \refeq{pdbG1}.  Conditioning on the set $\Ccal(\te)=A \subset
V^-$, we see that
    \eqn{
    \prob_{p,\gamma}^-( W)
    =
    \sum_{A : \tf\not\in A} \prob_{p,\gamma}^-(\Ccal(\te)=A, \;\;
    \te \conn \Gcal, \;\; \tf \conn \Gcal, \;\;
    \Ccal(\te) \cap \Ccal(\tf) = \emptyset).
    }
This can be rewritten as
    \eqan{
    \lbeq{P-W}
    \prob_{p,\gamma}^-(W)
    & =
    \sum_{A\colon  \tf\not\in A} \prob_{p,\gamma}^-( \Ccal(\te)=A, \;\;
    \te \conn \Gcal , \;\; \tf \conn \Gcal
    \mbox{  in  } V^-\setminus A),
    }
where $\{\tf \conn \Gcal$ in $V^- \setminus A\}$ is the event that
there exists $x\in \Gcal$ such that $\tf \conn x$ in $V^- \setminus A$.
The intersection of the first two events on the right hand side
of \refeq{P-W} is independent of the third event,
and hence
    \eqn{
    \lbeq{PW}
    \!\!\!\prob_{p,\gamma}^-(W)
    =
    \sum_{A\colon \te\not\in A} \prob_{p,\gamma}^-( \Ccal(\te)=A, \;\;
    \te \conn \Gcal) \; \prob_{p,\gamma}^-(\tf \conn \Gcal
    \mbox{  in  } V^-\setminus A).
    }
Let $M^-(x) = \prob_{p,\gamma}^-(x \conn \Gcal)$, for $x \in V^-$.
Then, by the BK inequality and the fact that the two-point function on
$G^-$ is bounded above by the two-point function on $G$,
    \eqan{
    \lbeq{PWlb}
    \prob_{p,\gamma}^-(\tf \conn \Gcal \mbox{  in  } V^-\setminus A)
    &= M^-(\tf) - \prob_{p,\gamma}^-(\tf \conn \Gcal \onlyon A)\\
    &\geq
    M^-(\tf) - \sum_{y \in A}\tau_p(\tf,y)M^-(y).\nn
    }
By definition and the BK inequality,
    \eqan{
    \lbeq{M*}
    M^-(x) &= M(p,\gamma) - \prob_{p,\gamma}(\mbox{$e$ or $f$ is occ.\ and piv.\ for }x \conn \Gcal)\\
    &\geq
    M(p,\gamma)(1-2p).\nn
    }

It follows from \refeq{PW}--\refeq{M*} and the upper bound
$M^-(x) \leq M(p,\gamma)$ that
    \eqan{
    \prob_{p,\gamma}^-( W)
    & \geq
    M(p,\gamma) \sum_{A : \te\not\in A}
    \prob_{p,\gamma}^-( \Ccal(\te)=A, \;\;
    \te \conn \Gcal)
    \big[(1-2p)-\sum_{y \in A}\tau_p(\tf,y) \big]
    \nn\\
    \lbeq{M**}
    & =
    M(p,\gamma) \big[ M^-(\te) (1-2p)
    - \sum_{y \in V^-} \tau_p(\tf,y)
    \prob_{p,\gamma}^-( \te\conn y, \; \te \conn \Gcal)\big].
    }
It is not difficult to show, using the BK inequality, that
    \eqn{
    \prob_{p,\gamma}^-( \te\conn y, \; \te \conn \Gcal)
    \leq
    \sum_{w \in V^-} \tau_p(\te,w) \tau_p(w,y) M^-(w),
    }
and hence, by \refeq{M*}--\refeq{M**},
    \eqan{
    \prob_{p,\gamma}^-(W)
    & \geq
    M(p,\gamma) \big[ M^-(\te)(1-2p)
    - \sum_{y,w \in V^-} \tau_p(\tf,y)
    \tau_p(\te,w) \tau_p(w,y) M^-(w) \big]
    \nn\\
    & \geq
    M^2(p,\gamma)\big[(1-2p)^2-\nabla_p^{\rm max}\big].\nn
    }
This completes the proof of \refeq{double-conn-G}.
\qed

\subsection{Derivation of \refeq{ineq2}}
\label{sec-rdi}
In this section, we prove \refeq{ineq2}, which is an adaptation of the proof of the
related inequality
        \eqn{
        \lbeq{ineq2-simple}
        M \leq \gamma \frac{\partial M}{\partial \gamma}
        +M^2+ pM\frac{\partial M}{\partial p},
        }
which is proved in \cite{AB87} (see also \cite[Lemma (5.53)]{Grim99}).
The main difference between \refeq{ineq2-simple} and \refeq{ineq2} is in the precise
constants. Indeed, we have that $p\cn\approx 1$ and $M\gg \gamma$, so that,
\refeq{ineq2} is morally equivalent to
    \eqn{
    M \leq \gamma \frac{\partial M}{\partial \gamma}
        +\frac 12 M^2+ \frac 12 pM\frac{\partial M}{\partial p},
        }
i.e., in the inequality in \refeq{ineq2-simple} the last two terms are multiplied by $1/2$.
%This suffices to improve upon the constant.

We follow the proof of \cite[Lemma (5.53)]{Grim99} as closely as possible,
deviating only when necessary. Indeed,
    \be \label{M-split}
    M(p,\gamma)=\prob_{p,\gamma}(\cluster(0)\cap \Gcal\neq \emptyset)
    =\prob_{p,\gamma}(|\cluster(0)\cap \Gcal|=1)
    +\prob_{p,\gamma}(|\cluster(0)\cap \Gcal|\geq 2).
    \ee
The first term on the r.h.s.\ of (\ref{M-split}) equals
$\gamma \frac{\partial M}{\partial \gamma}$, as derived in
\cite[(5.69)]{Grim99}. For the second term, we define
$A_x$ to be the event that either $x\in \Gcal$ or that
$x$ is connected by an occupied path to a vertex $g\in \Gcal$.
%We further write $A_x \circ A_y$ for set of realizations $\omega$
%such that there exist
%disjoint sets of vertices $V_1$ and $V_2$ in $\Gcal$ and
%disjoint sets of occupied edges $E_1$ and $E_2$ such that
%$A_x$ occurs for all configurations $\omega'$ which agree with $\omega$
%on $(V_1,E_1)$ and such that $A_y$ occurs for all
%configurations $\omega''$ which agree with $\omega$
%on $(V_2,E_2)$. This is the usual notion of disjoint occurrence,
%now applied to the joint bond and site configurations.
Then,
    \eqan{
    \prob_{p,\gamma}(|\cluster(0)\cap \Gcal|\geq 2)
    &=\prob_{p,\gamma}(A_0\circ A_0)\\
    &\qquad+\prob_{p,\gamma}(|\cluster(0)\cap \Gcal|\geq 2, A_0\circ A_0
    \text{ does not occur}).\nn
    }
In the derivation of \refeq{ineq2-simple}, we simply apply the BK-inequality to obtain
    \eqn{
    \prob_{p,\gamma}(A_0\circ A_0)
    \leq \prob_{p,\gamma}(A_0)^2= M(p,\gamma)^2,
    }
leading to the second term in \refeq{ineq2-simple}. Instead, we split,
depending on whether $0\in \Gcal$ or not. If
$0\in \Gcal$, then $0\in \Gcal$ occurs disjointly from
$A_0$, so that the BK-inequality yields
    \eqn{
    \prob_{p,\gamma}(A_0\circ A_0, 0\in \Gcal)
    \leq \prob_{p,\gamma}(A_0\circ \{0\in \Gcal\})
    \leq \gamma \prob(A_0)=\gamma M(p,\gamma).
    }
When, instead, $0\not\in \Gcal$,  there must be \emph{at least two} neighbors $e$ of the origin
for which the event
    \eqn{
    \lbeq{AeA0-event}
    A_e\circ A_0\circ \{(0,e) \text{ occ.}\}
    }
occurs. Therefore, we can bound, with $N$ denoting the number of
neighbors $e$ for which the event in \refeq{AeA0-event} occurs,
so that $N\geq 2$ a.s.\ and Markov's inequality yields
    \eqan{
    \prob_{p,\gamma}(A_0\circ A_0, 0\not\in \Gcal)
    &\leq \sum_{e\sim 0} \expec_{p,\gamma}\Big[\frac{1}{N}\indic{A_e\circ A_0\circ \{(0,e) \text{ occ.}\}}
    \Big]\nn\\
    &\leq \frac 12 \sum_{e\sim 0} \prob_{p,\gamma}(A_e\circ A_0\circ \{(0,e) \text{ occ.}\}).
    \lbeq{factor-half}
    }
Therefore, again by
the BK-inequality,
    $$
    \prob_{p,\gamma}(A_0\circ A_0, 0\not\in \Gcal)\leq
    \frac{1}{2} \sum_{e}\prob_{p,\gamma}(A_e)\prob_{p,\gamma}(A_0) p=
    \frac{1}{2} p\cn M(p,\gamma)^2,
    $$
so that
    \eqn{
    \lbeq{A0circ-bd}
    \prob_{p,\gamma}(A_0\circ A_0)
    \leq \frac{1}{2} p\cn M(p,\gamma)^2+\gamma M(p,\gamma),
    }
which yields the second term in \refeq{ineq2}.

We move on to the bound on the probability of the event that
$|\cluster(0)\cap \Gcal|\geq 2$, but that $A_0\circ A_0$ does not occur.
This event is equivalent to the existence of an
edge $b=(x,y)$ for which the following occurs:

\begin{enumerate}
\item[(i)] the edge $b$ is occupied; and

\item[(ii)] in the subgraph of $G$ obtained by deleting $b$, the following events occur:\\
(a) no vertex of $\Gcal$ is joined to the origin by an open path;\\
(b) $x$ is joined to 0 by an occupied path;\\
(c) the event $A_y\circ A_y$ occurs.
\end{enumerate}

The events in (ii) are independent of the occupation status of the edge $b=(x,y)$ so that
    \eqan{
    &\prob_{p,\gamma}(|\cluster(0)\cap \Gcal|\geq 2, A_0\circ A_0
    \text{ does not occur})\\
    &\qquad=\frac{p}{1-p} \sum_{x\sim y}\prob_{p,\gamma}((x,y) \text{ closed}, x\in \cluster(0),
    \cluster(0)\cap \Gcal=\emptyset,
    A_y\circ A_y)\nn\\
    &\qquad \leq \frac{p}{1-p} \sum_{x\sim y}\prob_{p,\gamma}(x\in \cluster(0),
    \cluster(0)\cap \Gcal=\emptyset,
    A_y\circ A_y),\nn
    }
where we write $x\sim y$ to denote that $(x,y)$ is a bond.
We condition on $\cluster(0)$ to obtain
    \eqan{
    &\prob_{p,\gamma}(\cluster(0)\cap \Gcal=\emptyset,
    A_y\circ A_y)\\
    &\qquad=\frac{p}{1-p} \sum_{x\sim y}\sum_{A}\prob_{p}(\cluster(0)=A)
    \prob_{p,\gamma}(\cluster(0)\cap \Gcal=\emptyset,
    A_y\circ A_y\mid \cluster(0)=A),\nn
    }
where the sum over $A$ is over all sets of vertices which contain $0$ and $x$
but not $y$. Conditionally on $\cluster(0)=A$, the events
$\cluster(0)\cap \Gcal=\emptyset$ and $A_y\circ A_y$ are \emph{independent},
since $\cluster(0)\cap \Gcal=\emptyset$ is defined on the vertices in
$A$, while $A_y\circ A_y$ depends on the vertices in $A^c$ and the edges between
them. Thus,
    \eqan{
    &\prob_{p,\gamma}(\cluster(0)\cap \Gcal=\emptyset,
    A_y\circ A_y\mid \cluster(0)=A)\\
    &\quad=\prob_{p,\gamma}(\cluster(0)\cap \Gcal=\emptyset\mid \cluster(0)=A)
    \prob_{p,\gamma}(A_y\circ A_y\text{ off }A),\nn
    }
where we write $\{A_y\circ A_y\text{ off }A\}$ for the event that
$A_y\circ A_y$ occurs in the graph where all edges with at least one
endpoint in the set $A$ are removed. So far, the derivation equals
that in the proof of \cite[Lemma (5.53)]{Grim99}.
Now we shall deviate. We split, depending on whether
$y\in \Gcal$ or not, to obtain
    $$
    \prob_{p,\gamma}(A_y\circ A_y\text{ off }A)
    =\prob_{p,\gamma}(A_y\circ A_y\text{ off }A, y\in \Gcal)
    +\prob_{p,\gamma}(A_y\circ A_y\text{ off }A, y\not\in\Gcal).
    $$
When $y\in \Gcal$,
    \eqn{
    \{A_y\circ A_y\text{ off }A, y\in \Gcal\}
    = \{(\{y\in \Gcal\}\circ A_y)\text{ off }A\},
    }
so that, by the BK-inequality,
    \eqn{
    \prob_{p,\gamma}(A_y\circ A_y\text{ off }A, y\in \Gcal)
    \leq \gamma \prob_{p,\gamma}(A_y\text{ off }A).
    }
As a result,
    \eqan{
    &\frac{p}{1-p} \sum_{x\sim y}\prob_{p,\gamma}(x\in \cluster(0),
    \cluster(0)\cap \Gcal=\emptyset,
    A_y\circ A_y, y\in \Gcal)\lbeq{first-term-ineq2}\\
    &\quad \leq \frac{\gamma p}{1-p} \sum_{x\sim y} \sum_{A}\prob_{p}(\cluster(0)=A)\prob_{p,\gamma}(\cluster(0)\cap \Gcal=\emptyset\mid \cluster(0)=A)\prob_{p,\gamma}(A_y\text{ off }A)\nn\\
    &\quad = \frac{\gamma p}{1-p} \sum_{x\sim y} \sum_{A}\prob_{p}(\cluster(0)=A)\prob_{p,\gamma}(\cluster(0)\cap \Gcal=\emptyset, A_y\mid \cluster(0)=A)\nn\\
    &\quad =\frac{\gamma p}{1-p} \sum_{x\sim y} \prob_{p,\gamma}(x\in \cluster(0),\cluster(0)\cap \Gcal=\emptyset, \cluster(y)\cap \Gcal \neq \emptyset)=\gamma p \frac{\partial M}{\partial p},\nn
    }
where the first equality follows again by conditional independence, and the last equality by
the fact that (see \cite[(5.67)]{Grim99})
    \eqn{
    \lbeq{partialMp-form}
    (1-p)\frac{\partial M}{\partial p}=\sum_{x\sim y} \prob_{p,\gamma}(x\in \cluster(0),\cluster(0)\cap \Gcal=\emptyset, \cluster(y)\cap \Gcal \neq \emptyset).
    }

We are left to bound the contribution where $y\not\in \Gcal$.
For this, we note that when $A_y\circ A_y$ occurs off $A$ and
$y\not \in \Gcal$, then there must be \emph{at least two} neighbors
$z$ of $y$ for which the event
    \eqn{
    \lbeq{AeA0-event-rep}
    \Big\{\big(A_z\circ A_y\circ \{(y,z) \text{ occ.}\}\big) \text{ off }A\Big\}
    }
occurs. Therefore, by a similar argument as in \refeq{factor-half},
    \eqn{
    \prob_{p,\gamma}(A_y\circ A_y\text{ off }A, y\not\in\Gcal)
    \leq \frac 12 \sum_{z\sim y} \prob_{p,\gamma}\Big((A_z\circ A_y\circ \{(y,z) \text{ occ.}\}) \text{ off }A\Big).
    }
By the BK-inequality,
    \eqan{
    \prob_{p,\gamma}\Big((A_z\circ A_y\circ \{(y,z) \text{ occ.}\}) \text{ off }A\Big)
    &\leq p \prob_{p,\gamma}(A_z \text{ off }A)\prob_{p,\gamma}(A_y \text{ off }A)\nn\\
    &\leq p M(p,\gamma) \prob_{p,\gamma}(A_y \text{ off }A).\nn
    }
Repeating the steps in \refeq{first-term-ineq2}, we thus arrive at
    \eqan{
    &\frac{p}{1-p} \sum_{x\sim y}\prob_{p,\gamma}(x\in \cluster(0),
    \cluster(0)\cap \Gcal=\emptyset,
    A_y\circ A_y, y\not\in \Gcal)
    \leq
    \frac{1}{2} \cn p^2 M(p,\gamma) \frac{\partial M}{\partial p}.
    \lbeq{second-term-ineq2}
    }
Therefore, summing the two bounds in \refeq{first-term-ineq2} and \refeq{second-term-ineq2},
we arrive at
    \eqan{
    &\prob_{p,\gamma}(|\cluster(0)\cap \Gcal|\geq 2, A_0\circ A_0
    \text{ does not occur})
    \leq \big[\frac{1}{2} \cn p M(p,\gamma)+\gamma\big]p\frac{\partial M}{\partial p},
    \nn
    }
which is the third term in \refeq{ineq2}. This completes the proof of \refeq{ineq2}.
\qed

\bigskip

\end{document}